%% file: main.tex
\documentclass{article}
\usepackage[utf8]{inputenc}
\usepackage{comment}
\usepackage{amssymb,amsmath,amsthm,graphicx,xcolor,mathtools,enumerate,mathrsfs}
\usepackage{fullpage,enumitem}
\usepackage{thmtools}
\usepackage{thm-restate}
\usepackage{hyperref}

\usepackage{tikz}
\usetikzlibrary{math}
\usetikzlibrary{snakes,calc,decorations.pathmorphing,through}
\tikzset{snake it/.style={decorate, decoration=snake}}

\usepackage{cleveref}
\theoremstyle{plain}
\newtheorem{theorem}{Theorem}[section]
\crefname{theorem}{Theorem}{Theorems}

\newtheorem{proposition}[theorem]{Proposition}
\crefname{proposition}{Proposition}{Propositions}

\newtheorem{corollary}[theorem]{Corollary}
\crefname{corollary}{Corollary}{Corollaries}

\newtheorem{lemma}[theorem]{Lemma}
\crefname{lemma}{Lemma}{Lemmas}

\crefname{conjecture}{Conjecture}{Conjectures}

\crefname{problem}{Problem}{Problem}

\newtheorem{claim}[theorem]{Claim}
\crefname{claim}{Claim}{Claims}

\crefname{observation}{Observation}{Observations}

\crefname{setup}{Setup}{Setups}

\crefname{fact}{Fact}{Facts}

\crefname{algorithm}{Algorithm}{Algorithms}

\crefname{remark}{Remark}{Remarks}

\crefname{example}{Example}{Examples}

\theoremstyle{definition}
\newtheorem{definition}[theorem]{Definition}
\crefname{definition}{Definition}{Definitions}

\crefname{construction}{Construction}{Constructions}

\crefname{question}{Question}{Questions}

\numberwithin{equation}{section}
\newcommand{\itref}[1]{\emph{\ref{#1}}}

\def\eps{\varepsilon}

\renewcommand{\int}[1]{\mathop{\mkern 0mu\mathrm{int}}\nolimits(#1)}

\newcommand\ceil[1]{\left\lceil#1\right\rceil}
\newcommand\floor[1]{\left\lfloor#1\right\rfloor}
\newcommand\E{\mathbb{E}}
\renewcommand\P{\mathbb{P}}

\newcounter{propcounter}

\definecolor{DarkDesaturatedBlue}{HTML}{3A3556}
\definecolor{VividOrange}{HTML}{F15918}
\definecolor{PureOrange}{HTML}{FFBA00}
\definecolor{LightGrayishPink}{HTML}{EEC5D5}
\definecolor{VerySoftBlue}{HTML}{B5AFDB}

\title{Ramsey numbers of bounded degree trees versus general graphs}
\author{Richard Montgomery\thanks{Mathematics Institute, University of Warwick, Coventry CV4 7AL, UK. richard.montgomery@warwick.ac.uk}
\and Mat\'ias Pavez-Sign\'e\thanks{Centro de Modelamiento Matem\'atico (CNRS IRL2807), Universidad de Chile, Santiago, Chile. mpavez@dim.uchile.cl} \and Jun Yan\thanks{
Mathematics Institute, University of Warwick, Coventry CV4 7AL, UK. jun.yan@warwick.ac.uk.
\newline \vspace{-0.2cm}
\newline RM and MPS supported by the European
Research Council (ERC) under the European Union Horizon 2020 research and innovation programme (grant agreement No. 947978) and the Leverhulme Trust.
\newline \vspace{-0.2cm}
\newline JY supported by the Warwick Mathematics Institute Centre for Doctoral Training, and by funding from the UK EPSRC (Grant number: EP/W523793/1)}}
\date{}

\begin{document}
	\maketitle
	\begin{abstract}
For every $k\ge 2$ and $\Delta$, we prove that there exists a constant $C_{\Delta,k}$ such that the following holds. For every graph $H$ with $\chi(H)=k$ and every tree with at least $C_{\Delta,k}|H|$ vertices and maximum degree at most $\Delta$, the Ramsey number $R(T,H)$ is $(k-1)(|T|-1)+\sigma(H)$, where $\sigma(H)$ is the size of a smallest colour class across all proper $k$-colourings of $H$. This is tight up to the value of $C_{\Delta,k}$, and confirms a conjecture of Balla, Pokrovskiy, and Sudakov.
\end{abstract}

\section{Introduction}\label{sec:intro}
\input{1intro}
%.tex 2prelim.tex 3mconstant.tex 4coververtices.tex 5maink2proof.tex 6inductiononk.tex}

%%%%%%%%%%%%%%%%%%%%%%%%%%%%%%%%%%%%%%%%%%%%%%%%%%%%%%%%%%%%%%%%%%%%%%%%%%%%%%%%%%%%%%%%%%%%%%%%%%%%%%%%%%%%%%%%%%%%%%%%%%%%%%%%%%%%%%%%%%%%%%%%%%%%%%%%%%%%

\section{Preliminaries}\label{sec:prelim}
\input{2prelim}

%%%%%%%%%%%%%%%%%%%%%%%%%%%%%%%%%%%%%%%%%%%%%%%%%%%%%%%%%%%%%%%%%%%%%%%%%%%%%%%%%%%%%%%%%%%%%%%%%%%%%%%%%%%%%%%%%%%%%%%%%%%%%%%%%%%%%%%%%%%%%%%%%%%%%%%%%%%%

\section{Proof of Theorem~\ref{theorem:dense}}\label{sec:proof:k=2:dense}
\input{3mconstant}

%%%%%%%%%%%%%%%%%%%%%%%%%%%%%%%%%%%%%%%%%%%%%%%%%%%%%%%%%%%%%%%%%%%%%%%%%%%%%%%%%%%%%%%%%%%%%%%%%%%%%%%%%%%%%%%%%%%%%%%%%%%%%%%%%%%%%%%%%%%%%%%%%%%%%%%%%%%

\section{Embedding trees to cover vertex subsets}\label{sec:embedtocover}
\input{4coververtices}

%%%%%%%%%%%%%%%%%%%%%%%%%%%%%%%%%%%%%%%%%%%%%%%%%%%%%%%%%%%%%%%%%%%%%%%%%%%%%%%%%%%%%%%%%%%%%%%%%%%%%%%%%%%%%%%%%%%%%%%%%%%%%%%%%%%%%%%%%%%%%%%%%%%%%%%%%%%%%

\section{Proof of Theorem~\ref{theorem:k=2}}\label{section:k=2}
\input{5maink2proof}

 %%%%%%%%%%%%%%%%%%%%%%%%%%%%%%%%%%%%%%%%%%%%%%%%%%%%%%%%%%%%%%%%%%%%%%%%%%%%%%%%%%%%%%%%%%%%%%%%%%%%%%%%%%%%%%%%%%%%%%%%%%%%%%%%%%%%%%%%%%%%%%%%%%%%%%%%%%%%

\section{Ramsey goodness of bounded degree trees}\label{sec:largerk}
\input{6inductiononk}

\bibliographystyle{abbrv}
\bibliography{Ramseygood}

\end{document}

%% file: 1intro.tex
Given two graphs $G$ and $H$, the \textit{Ramsey number} $R(G,H)$ is defined as the smallest integer $N$ such that every red/blue colouring of the edges of the complete graph $K_N$ contains either a red copy of $G$ or a blue copy of $H$. The fundamental result of Ramsey~\cite{Ramsey} implies that $R(G,H)$ is well-defined for any pair of graphs $G$ and $H$. The exact value of $R(G,H)$ is known for very few pairs of graphs $(G,H)$, and in general it is difficult even to give good bounds on $R(G,H)$. Of the few exact Ramsey numbers known, many share the same general extremal lower bound construction given below. The area of \emph{Ramsey goodness} studies the graphs $G$ and $H$ for which this lower bound is tight.

Erd\H{o}s showed in 1947 that the Ramsey number of an $n$-vertex path $P_n$ and a complete $m$-vertex graph $K_m$ is $R(P_n,K_m)=(m-1)(n-1)+1$. Here, the lower bound construction is the disjoint union of $m-1$ red $(n-1)$-vertex cliques with all possible edges in blue added between them. As Chv\'atal and Harary~\cite{ChvatalHarary} observed, this construction contains no connected red $n$-vertex subgraph and no blue subgraph with chromatic number at least $m$; if $G$ is connected, we therefore have $R(G,H)\ge (\chi(H)-1)(|G|-1)+1$.
%  that the extremal example given by Erd\H{o}s gives a lower bound for $R(G,H)$ more generally, when $G$ is connected. Indeed, consider $\chi(H)-1$ disjoint red cliques, each of order $|G|-1$, where $\chi(H)$ denote the chromatic number of $H$, and colour all the crossing edges blue. Since $G$ is connected, this colouring cannot contain a red copy of $G$ or a blue copy of $H$, and thus $R(G,H)\ge (\chi(H)-1)(|G|-1)+1$.
Let $\sigma(H)$ denote the size of a smallest colour class across all proper $\chi(H)$-colourings of $H$. If $2\leq \sigma(H)\leq |G|$, then this construction can be improved by adding further a red clique of order $\sigma(H)-1$ connected to the rest of the construction with blue edges. This observation, made by Burr~\cite{Burr81}, implies that for every connected graph $G$ with $|G|\geq\sigma(H)$ we have
    \begin{equation}\label{equation:ramseygood}
    R(G,H)\geq(|G|-1)(\chi(H)-1)+\sigma(H).
    \end{equation}
As coined by Burr and Erd\H os~\cite{BurrErdos} in 1983, we say that a graph $G$ is \emph{$H$-good} if~\eqref{equation:ramseygood} holds with equality, that is, if $R(G,H)=(|G|-1)(\chi(H)-1)+\sigma(H)$. A family of graphs $\mathcal G$ is said to be \emph{$H$-good} if there is some $n_0$ such that every graph $G\in \mathcal G$ is $H$-good if $|G|\ge n_0$. Chv\'{a}tal~\cite{chvatal} had already shown in 1977 that the family of trees is $K_r$-good for all $r$. Inspired by this, Burr and Erd\H os sought to determine which families $\mathcal G$ are $K_r$-good for all $r$. They showed that, for each fixed $b$, the family of graphs with \textit{bandwith} at most $b$ is $K_r$-good for all $r$, and raised a series of questions concerning the $K_r$-goodness of various natural families~\cite{BurrErdos}.
%An early result of Chv\'{a}tal~\cite{chvatal} shows that the family of trees is $K_r$-good for all $r$. Inspired by this to define the concept of Ramsey goodness, Burr and Erd\H os sought to determine which families $\mathcal G$ are $K_r$-good for all $r$. For example, they proved~\cite{BurrErdos} that, for a fixed $b$, the family of graphs with \textit{bandwith} at most $b$ is $K_r$-good for all $r$.

In 1985, Erd\H os, Faudree, Rousseau, and Schelp~\cite{erdHos1985multipartite} showed that the family of bounded degree trees is $H$-good for any graph $H$, where some (possibly much weaker) restriction on the maximum degree is needed (see~\cite{BURR198879}).
%cannot be totally dropped, as the star $S_n$ with $n$ leaves satisfies $R(S_n,C_4)\ge n+n^{\frac{1}{2}}-5n^{3/10}$ (see~\cite{BURR198879}).
All but one of the Burr-Erd\H os questions raised in~\cite{BurrErdos} were solved by Nikiforov and Rousseau~\cite{nikiforov2009ramsey} in 2009, mostly as a result of their proof of the $K_r$-goodness for each $r$ of any family of graphs with constant maximum degree which have a `small' set of vertices whose removal divides the graph into small linear-sized components (see~\cite{nikiforov2009ramsey} for a more precise statement). The last question in \cite{BurrErdos} was solved by Fiz Pontiveros, Griffiths, Morris, Saxton, and Skokan~\cite{fiz2014ramsey} in 2014, who proved that the $n$-dimensional hypercube $Q_n$ is $H$-good if $n$ is sufficiently large. Beyond these questions, Allen, Brightwell, and Skokan~\cite{allen2013ramsey} showed in 2013 that the family of graphs $G$ with constant maximum degree and bandwith $o(|G|)$ is $H$-good for every graph $H$.

More recently, focus has been turned on quantitative considerations of the Ramsey goodness problem. That is, when $\mathcal{G}$ is an $H$-good family, how large does $n_0$ need to be before every $G\in \mathcal{G}$ is $H$-good if $|G|\geq n_0$.
%In this article, we will focus on more quantitative considerations on the Ramsey goodness problem. For an $H$-good family $\mathcal G$, we would like to determine the minimum number $n_0=n_0(H)$ such that every member $G\in\mathcal{G}$ is $H$-good for $|G|\ge n_0$.
%Let $C_n$ denote the cycle with $n$ vertices.
As their result quoted above shows that for any graph $H$, the path $P_n$ and the $n$-vertex cycle $C_n$ are $H$-good if $n$ is sufficiently large depending on $H$, Allen, Brightwell, and Skokan~\cite{allen2013ramsey} conjectured that $n\ge \chi(H)|H|$ should be sufficient. In the case of paths, this was proved in a strong form when $\chi(H)\geq 4$ by Pokrovskiy and Sudakov~\cite{pokrovskiy2017ramsey}, who showed that, if $n\geq 4|H|$, then $P_n$ is $H$-good. %does not depend on the chromatic number of $H$, as follows.
%\begin{theorem}[Pokrovskiy and Sudakov~\cite{pokrovskiy2017ramsey}]\label{theorem:ramseygoodnessforpaths}Let $H$ be a graph with chromatic number $k_r$. Then for all $n\geq4|H|$, $R(P_n,H)=(k-1)(n-1)+\sigma(H)$.
%\end{theorem}
For cycles, Pokrovskiy and Sudakov~\cite{pokrovskiy2020ramsey} later showed that $C_n$ is $H$-good as long as $n\ge 10^{60}|H|$ and $\sigma(H)\ge \chi(H)^{22}$. This result was recently improved by Haslegrave, Hyde, Kim, and Liu~\cite{haslegrave2021ramsey}, who showed that, for some universal $C>0$, $n\ge C|H|\log^4\chi(H)$ is sufficient for the cycle $C_n$ to be $H$-good, which is optimal up to the logarithmic factor of $\chi(H)$ and confirms the conjecture of Allen, Brightwell, and Skokan for graphs with large enough chromatic number.

Though not mentioned explicitly, quantitative bounds can be taken more generally for the Ramsey goodness of bounded degree trees from the proof of Erd\H os, Faudree, Rousseau, and Schelp~\cite{erdHos1985multipartite} mentioned above. That is, the proof can be used to show that, for each $k$ and $\Delta$, there is some $C_{\Delta,k}$ such that for each graph $H$ with $\chi(H)=k$ and each $n\geq C_{\Delta,k}|H|^4$, every $n$-vertex tree $T$ with maximum degree at most $\Delta$ is $H$-good.
This was improved by Balla, Pokrovskiy, and Sudakov~\cite{balla2018} in 2018, who proved that $n\ge C_{\Delta,k}|H|\log^4|H|$ suffices, for some appropriate $C_{\Delta,k}$. Furthermore, they showed the factor $\log^4|H|$ is not needed when $T$ has $\Omega(n)$ leaves, and conjectured this should be true for every tree, that is, that $n\geq C_{\Delta,k}|H|$ should suffice for $T$ to be $H$-good.
%for fixed $\Delta$ and $k$, there is a constant $C_{\Delta,k}$ such that for any graph $H$ with $\chi(H)=k$ and $n\ge C_{\Delta,k}|H|\log^4|H|$, every tree $T$ on $n$ vertices and maximum degree at most $\Delta$ is $H$-good. Moreover, they showed that this bound is tight up to the $\log^4|H|$ factor, and they removed this factor when $T$ has $\Omega(|H|)$ leaves.
%This led them to conjecture that $|T|\geq C_{\Delta,k}|H|$ would suffice for general bounded degree trees as well. In this paper, we confirm their conjecture, as follows.
The purpose of this paper is to confirm this conjecture, as follows.
\begin{theorem}\label{theorem:main:2}
	For all $\Delta$ and $k$, there exists a constant $C_{\Delta,k}$ such that the following holds. Given a graph $H$ with $\chi(H)=k$ and $n\ge C_{\Delta,k}|H|$, every $n$-vertex tree $T$ with maximum degree at most $\Delta$ is $H$-good. In other words, $R(T,H)=(k-1)(n-1)+\sigma(H)$.
\end{theorem}

%Generalising slightly an observation of Balla, Pokrovskiy and Sudakov~\cite{balla2018}, every $n$-vertex tree is not $H$-good if $|H|<\chi(H)n$ for each $k\geq 2$. Therefore,
Theorem~\ref{theorem:main:2} is tight up to the value of $C_{\Delta,k}$, as it is easy to see that $R(G,H)\geq |H|$ for any non-empty graphs $G$ and $H$, so $G$ is not $H$-good if $|G|<|H|/k$, and hence $C_{\Delta,k}=\Omega(1/k)$.
%. Indeed, more generally, if $k=\chi(H)$ and $n<|H|/k$, then $(k-1)(n-1)+\sigma(H)+1\leq (k-1)(n-1)+|H|/k+1\leq (k-1)n+|H|/k<|H|$, so that by considering a blue clique with $|H|-1$ vertices, we have that no $n$-vertex graph $G$ is $H$-good. Certainly, then,  in
%Theorem~\ref{theorem:main:2} we must have $C_{\Delta,k}\geq 1/k$ for every $\Delta$.
However, without much more difficulty, it can be seen that we must have $C_{\Delta,k}\geq \Delta/100k\log \Delta$ for sufficiently large $\Delta$, by taking $m$ to be sufficiently large, letting $H$ be the $k$-partite complete graph with $m$ vertices in each class, and letting $T$ be any tree with maximum degree $\Delta$ on $n=\Delta m/100\log \Delta$ vertices. To see this, let $N=\Delta m/10\log \Delta$, and take a blue complete $\lfloor k/3\rfloor$-partite graph with $N$ vertices in each class, and colour the edges within each class red/blue so that the maximum red degree is at most $\Delta-1$ and there is no blue copy of $K_{m,m}$. This colouring is possible using the probabilistic method, more precisely, colouring edges of $K_{2N}$ red independently at random with probability $\Delta/4N$ and removing vertices with at least $\Delta$ red neighbouring edges. Call the final red/blue coloured graph $G$.
If there is a blue copy of $H$ in $G$, then at least $3m$ vertices from this copy are within one class of $N$ vertices in the original partition of $G$, and hence the colouring within this class must contain a blue copy of $K_{m,m}$, a contradiction. As $G$ has maximum red degree at most $\Delta-1$, it contains no red copy of $T$, and therefore $R(T,H)>\lfloor k/3\rfloor N >kn$. Thus, $T$ is not $H$-good, so we must have $C_{\Delta,k}\geq |T|/|H|\geq \Delta/100k\log \Delta$.
We have not optimised the value of $C_{\Delta,k}$ in Theorem~\ref{theorem:main:2} as our proof would give an upper bound on $C_{\Delta,k}$ that is very far from this $\Delta/100k\log\Delta$ lower bound.

%% file: 2prelim.tex
After covering the basic notation that we use, we will give a detailed sketch of the proof of Theorem~\ref{theorem:main:1} in Section~\ref{sec:sketch} before outlining the rest of the paper in Section~\ref{sec:overview}.

%%%%%%%%%%%%%%%%%%%%%%%%%%%%%%%%%%%%%%%%%%%%%%%%%%%%%%%%%%%%%%%%%%%%%%%%%%%%%%%%%%%%%%%%%%%%%%%%%%%%%%%%%%%%%%%%%%%%%%%%%%%%%%%%%%%%%%%%%%%%%%%%%%%%%%%%%%%%%
%%%%%%%%%%%%%%%%%%%%%%%%%%%%%%%%%%%%%%%%%%%%%%%%%%%%%%%%%%%%%%%%%%%%%%%%%%%%%%%%%%%%%%%%%%%%%%%%%%%%%%%%%%%%%%%%%%%%%%%%%%%%%%%%%%%%%%%%%%%%%%%%%%%%%%%%%%%%%
%%%%%%%%%%%%%%%%%%%%%%%%%%%%%%%%%%%%%%%%%%%%%%%%%%%%%%%%%%%%%%%%%%%%%%%%%%%%%%%%%%%%%%%%%%%%%%%%%%%%%%%%%%%%%%%%%%%%%%%%%%%%%%%%%%%%%%%%%%%%%%%%%%%%%%%%%%%%%
%%%%%%%%%%%%%%%%%%%%%%%%%%%%%%%%%%%%%%%%%%%%%%%%%%%%%%%%%%%%%%%%%%%%%%%%%%%%%%%%%%%%%%%%%%%%%%%%%%%%%%%%%%%%%%%%%%%%%%%%%%%%%%%%%%%%%%%%%%%%%%%%%%%%%%%%%%%%%
%%%%%%%%%%%%%%%%%%%%%%%%%%%%%%%%%%%%%%%%%%%%%%%%%%%%%%%%%%%%%%%%%%%%%%%%%%%%%%%%%%%%%%%%%%%%%%%%%%%%%%%%%%%%%%%%%%%%%%%%%%%%%%%%%%%%%%%%%%%%%%%%%%%%%%%%%%%%%
%%%%%%%%%%%%%%%%%%%%%%%%%%%%%%%%%%%%%%%%%%%%%%%%%%%%%%%%%%%%%%%%%%%%%%%%%%%%%%%%%%%%%%%%%%%%%%%%%%%%%%%%%%%%%%%%%%%%%%%%%%%%%%%%%%%%%%%%%%%%%%%%%%%%%%%%%%%%%

\subsection{Notation}
A graph $G$ has vertex set $V(G)$ and edge set $E(G)$, and we write $|G|=|V(G)|$ and $e(G)=|E(G)|$. A vertex $v\in V(G)$ has neighbourhood $N(v)$ and degree $d(v)=|N(v)|$. Given a vertex set $S\subset V(G)$, the set of neighbours of $S$ (including those in $S$) is $N'(S)=\cup_{s\in S}N(s)$ and its external neighbourhood is $N(S)=\cup_{s\in S}N(s)\setminus S$. For a vertex $v\in V(G)$ and sets $U,S\subset V(G)$, we write $d(v,U)=|N(v)\cap U|$ and $N(S,U)=N(S)\cap U$. Given a subset $S\subset V(G)$, $G[S]$ is the subgraph of $G$ induced on $S$, with vertex set $S$ and every edge of $G$ contained within $S$. We write $G-S$ for the graph $G[V(G)\setminus S]$. For a subset $X\subset V(G)$,  $I(X)$ denotes the subgraph of $G$ with vertex set $X$ and no edges. The complement of $G$, denoted $G^c$, is the graph with  $V(G^c)=V(G)$ and such that $xy$ forms an edge in $G^c$ if and only if $xy$ is not an edge in $G$.

Given two subgraphs $H_1,H_2\subset G$, $H_1\cup H_2$ is the subgraph with vertex set $V(H_1)\cup V(H_2)$ and edge set $E(H_1)\cup E(H_2)$. If $H\subset G$ is a subgraph and $e=xy\in E(G)$, then $H+e$ denotes the subgraph with vertex set $V(H)\cup\{x,y\}$ and edge set $E(H)\cup \{xy\}$. Given $x,y\in V(G)$, an $x,y$-path is a path with endvertices $x$ and $y$. A path $P$ with $\ell$ vertices has length $\ell-1$. We use $H+P$ to denote $H\cup P$, and $H-P$ to denote the graph resulting from removing the edges of $P$ from $H$ and any internal vertex of $P$ which is now isolated. For subsets $S_1,S_2\subset V(G)$, an $S_1,S_2$-path is a path with one endpoint in $S_1$ and the other endpoint in $S_2$.

For a positive integer $n\in\mathbb N$, we write $[n]=\{1,\ldots,n\}$ and $[n]_0=[n]\cup \{0\}$. We will use the standard hierarchy notation, that is, for $a,b\in (0,1]$, we will write $a\ll b$ to mean that there exists a non-decreasing function $f:(0,1]\to (0,1]$ such that if $a\le f(b)$ then the rest of the proof holds with $a$ and $b$. Hierarchies with more constants are defined in a similar way
and are to be read from the right to the left. For simplicity we will ignore floor and ceiling signs whenever this does not affect the argument.

%%%%%%%%%%%%%%%%%%%%%%%%%%%%%%%%%%%%%%%%%%%%%%%%%%%%%%%%%%%%%%%%%%%%%%%%%%%%%%%%%%%%%%%%%%%%%%%%%%%%%%%%%%%%%%%%%%%%%%%%%%%%%%%%%%%%%%%%%%%%%%%%%%%%%%%%%%%%%
%%%%%%%%%%%%%%%%%%%%%%%%%%%%%%%%%%%%%%%%%%%%%%%%%%%%%%%%%%%%%%%%%%%%%%%%%%%%%%%%%%%%%%%%%%%%%%%%%%%%%%%%%%%%%%%%%%%%%%%%%%%%%%%%%%%%%%%%%%%%%%%%%%%%%%%%%%%%%
%%%%%%%%%%%%%%%%%%%%%%%%%%%%%%%%%%%%%%%%%%%%%%%%%%%%%%%%%%%%%%%%%%%%%%%%%%%%%%%%%%%%%%%%%%%%%%%%%%%%%%%%%%%%%%%%%%%%%%%%%%%%%%%%%%%%%%%%%%%%%%%%%%%%%%%%%%%%%
%%%%%%%%%%%%%%%%%%%%%%%%%%%%%%%%%%%%%%%%%%%%%%%%%%%%%%%%%%%%%%%%%%%%%%%%%%%%%%%%%%%%%%%%%%%%%%%%%%%%%%%%%%%%%%%%%%%%%%%%%%%%%%%%%%%%%%%%%%%%%%%%%%%%%%%%%%%%%
%%%%%%%%%%%%%%%%%%%%%%%%%%%%%%%%%%%%%%%%%%%%%%%%%%%%%%%%%%%%%%%%%%%%%%%%%%%%%%%%%%%%%%%%%%%%%%%%%%%%%%%%%%%%%%%%%%%%%%%%%%%%%%%%%%%%%%%%%%%%%%%%%%%%%%%%%%%%%
%%%%%%%%%%%%%%%%%%%%%%%%%%%%%%%%%%%%%%%%%%%%%%%%%%%%%%%%%%%%%%%%%%%%%%%%%%%%%%%%%%%%%%%%%%%%%%%%%%%%%%%%%%%%%%%%%%%%%%%%%%%%%%%%%%%%%%%%%%%%%%%%%%%%%%%%%%%%%

\subsection{Outline of the proof of Theorem~\ref{theorem:main:1}}\label{sec:sketch}
%Let $H$ have chromatic number $k$. Pick a $k$-colouring of $H$ which has a colour class of size $\sigma(H)$, and let the size of the colour classes be $m_1\leq\ldots\leq m_k$, so that $m_1=\sigma(H)$.
%Let $K_{m_1,\ldots,m_k}$ denote the complete $k$-partite graph whose parts have sizes $m_1,\ldots, m_k$. Clearly, if a graph $G$ contains $K_{m_1,\ldots,m_k}$ as a subgraph, then $G$ contains a copy of $H$. Therefore, to prove Theorem~\ref{theorem:main:2} it is enough to prove our main theorem while assuming that $H$ is a complete $k$-partite graph, as follows.
Let $H$ have chromatic number $k$ such that $|H|\leq \mu n$. Let $m=\sigma(H)$. For convenience, if not elegance, sake, we will use $K^{k-1}_{\mu n}\times K^c_m$ to denote the complete $k$-partite graph whose first $k-1$ classes have size $\mu n$ and whose last class has size $m$. Thus, if a graph $G$ contains a copy of $K^{k-1}_{\mu n}\times K^c_m$, then it contains a copy of $H$, so that, to prove Theorem~\ref{theorem:main:2} it is enough to prove the following (where we also take $\mu_{\Delta,k}=1/C_{\Delta,k}$).
\begin{restatable}{theorem}{maintheorem}\label{theorem:main:1}
	For all $\Delta$ and $k$, there exists a constant $\mu_{\Delta,k}$ such that the following holds for each $m\leq \mu_{\Delta,k}n$. Every tree $T$ with $n$ vertices and maximum degree at most $\Delta$ satisfies $R(T,K^{k-1}_{\mu_{\Delta,k}n}\times K_m^c)=(k-1)(n-1)+m$.
\end{restatable}%	For all $\Delta$ and $k$, there exists a constant $C_{\Delta,k}$ such that the following holds. Given positive integers $m_1\le \dots\le m_k$ and $n\ge C_{\Delta,k}m_k$, every tree $T$ with $n$ vertices and maximum degree at most $\Delta$ satisfies $R(T,K_{m_1,\dots,m_k})=(k-1)(n-1)+m_1$.
%$K(k-1,\mu_{\Delta,k}n,m)$

We will prove Theorem~\ref{theorem:main:1} by induction on $k$, and therefore the proof falls into two parts, the base case ($k=2$) and the inductive step. This follows the work of Balla, Pokrovskiy, and Sudakov~\cite{balla2018}, who proved Theorem~\ref{theorem:main:1} under a stronger assumption equivalent to replacing $\mu_{\Delta,k}n$ with $\mu_{\Delta,k}n/\log^4n$ throughout. The critical case for making our improvement is the case $k=2$, and we concentrate on this in this sketch, before outlining briefly how the case for $k\geq 3$ follows by induction.

For the case $k=2$, it is sufficient to prove the following for some $\mu=\mu(\Delta)$. That, for every $n$-vertex tree $T$ with maximum degree $\Delta$ and all $m\leq \mu n$, the following holds.
\stepcounter{propcounter}
\begin{enumerate}[label = \textbf{\Alph{propcounter}}]
\item If $G$ has $n+m-1$ vertices, and the complement of $G$ is $K_{m,\mu n}$-free, then $G$ contains a copy of $T$.\label{situation}
\end{enumerate}
An $(n+m-1)$-vertex graph $G$ whose complement is $K_{m,\mu n}$-free satisfies the natural expansion condition that, for each set $U$ of $m$ vertices, $|N_G(U)|\geq |G|-|U|-\mu n\geq (1-2\mu)n$.
Beginning with work by Friedman and Pippenger~\cite{FP1987}, through a key development by Haxell~\cite{haxell}, trees can be embedded into graphs satisfying expansion conditions (see, for example, Corollary~\ref{corollary:extendable:embedding}).
A typical such application of these methods would give the following.

\stepcounter{propcounter}
\begin{enumerate}[label = \textbf{\Alph{propcounter}}]
\item If every set $U$ of $m$ vertices in a graph $G$ satisfies $|N(U)|\geq |G|-2\mu n$, then a tree $T$ with maximum degree $\Delta$ can be embedded in $G$ if $G$ has at least $20\Delta\cdot \max\{\mu n,m\}$ more vertices than $T$.\label{principle}
\end{enumerate}
Our trouble, then, is the gulf between \ref{situation} and \ref{principle}: we only have $m-1$ more vertices in $G$ than $T$ in \ref{situation}, far from the $20\Delta\cdot \mu n$ spare vertices required to use \ref{principle}, particularly as $m$ can be any value between 1 and $\mu n$. Critically, we use the following observation for our graph $G$ whose complement is $K_{m,\mu n}$-free. That, for some large constant $K$, if we take a set $V\subset V(G)$ of $Km$ vertices at random, then,
any set $U\subset V(G)$ of $m$ vertices has at most $\mu n$ non-neighbours in $G$, so can expect to have at most $2\mu Km$ non-neighbours in $V$. With a little care, then, we can show, with high probability, that every set $U\subset V$ with $|U|=m$ satisfies $|N_{G[V]}(U)|\geq (1-3\mu)|V|$. Then, in $G[V]$, we can apply \ref{principle} to embed a tree with maximum degree $\Delta$ using only $O(\Delta\cdot \max\{\mu |V|,m\})=O(\Delta\cdot \max\{\mu Km,m\})$ spare vertices. When we embed using $m-1$ spare vertices (as at \ref{situation}), this is still not enough, but for small $m$ we are much closer, needing $O(m)$ spare vertices instead of $O(n)$ spare vertices. With more analysis, and choosing the constants involved carefully, we can show that, for some small $\lambda$, $G[V]$ here is likely to be $K_{\lambda m,\lambda m}$-free, which will allow us to use \ref{principle} to embed a tree with maximum degree $\Delta$ into $G[V]$ while using only $m-1$ spare vertices.

This key idea for finishing the embedding leads to the following strategy. We take a nested sequence of random subsets $V(G)=U_0\supset U_1\supset\ldots\supset U_\ell$ with geometrically-decreasing size, before embedding $T$ in $\ell$ stages, at each stage $i\in [\ell]$ using all unused vertices in $U_{i-1}\setminus U_{i}$ and as few vertices in $U_i$ as possible, until we finish by embedding the last part of $T$ within $G[U_{\ell}]$. 
This approach is inspired by the `cover down methods' in the iterative absorption techniques used by the first author with Glock, K\"uhn, Lo, and Osthus~\cite{glock2019decomposition} (developing earlier work by Barber, K\"uhn, Lo, and Osthus~\cite{barber2016edge}), and similarly we also call our nested sequence of sets a \emph{vortex}.

%The size of the last set in the vortex  we use depends on our decomposition of $T$ into pieces with decreasing size. If the remaining piece to be embedded has linearly (in its size) many leaves then it turns out this is not too difficult to embed with $m-1$ spare vertices using a method of  Balla, Pokrovskiy and Sudakov~\cite{balla2018}, so one we have found such a piece we do not need to continue to find a longer vortex. On the other hand, if the piece does not have linearly many leaves, then it has many short bare paths (see Lemma~\ref{lema:leaves}), and these help us to embed the piece of the tree while covering vertices not in the next set, $U_i$ say, in the vortex. Continuing in this fashion, we ultimately may have that $U_\ell$ has $O(m)$ vertices, where we can then finish the embedding as described above, using a version of \ref{principle}.

Between this sketch and our implementation, there are several complications. We cannot take a simple vortex and must run a cleaning process on an initial random vortex. The `vortex' embedding method also requires $m$ to be at least some large constant; when it is not we must use a separate (though much easier) embedding. Finally, before doing anything, we need to efficiently gain an expansion property for small sets (i.e., those with fewer than $m$ vertices). As in previous work (in particular,~\cite{balla2018}), we can remove a small set of vertices (effectively some small maximal set which does not expand) to gain this property, but removing these vertices gives us fewer `spare vertices' to work with than the $m-1$ we need. If we end up with ${m}'-1$ spare vertices, then it will not be hard to show that sets with size ${m}'$ expand significantly, but this change in the value of $m$ needs to be dealt with carefully, which we do in Section~\ref{sec:subd}.% Thus, we need to be more careful than in previous work and adjust the value of $m$  and may need to adjust the value of $m$ ().

%This we do by taking a maximal set $W\subset V(G)$ with $|W|\leq 2m$ and $|N(W)|<D|W|$ for $D=\mu n/10m$, letting $G'=G-W$ and $m'=m-|W|$, and observing %the following three points.
%\begin{enumerate}[label = \roman{enumi})]
%\item $|\hat{G}|=n+\hat{m}-1$.
%\item If $\hat{G}^c$ contains a copy of $K_{\hat{m},2\mu n}$, then $G$ contains a copy of $K_{m,\mu n}$.
%\item For each $U\subset V(\hat{G})$ with $|U|\leq m$, we have $|N_{\hat{G}}(U)|\geq D|U|$ (for otherwise we get a contradiction by considering the non-expansion of $U\cup W$, as is shown in Proposition \ref{prop:expander:subgraphnew}).
%\end{enumerate}
%Thus, it suffices to find a copy of an $n$-vertex tree $T$ with maximum degree $\Delta$ in a $K_{\hat{m},2\mu n}$-free graph $\hat{G}$ with $n+\hat{m}-1$, where we have gained a condition for the expansion of small sets, the condition iii) above.

Finally, we note that, while some of the tree embeddings we use are quite intricate, there exist very good methods to carry them out in a flexible fashion using expansion properties. More specifically, we use an `extendability' method for embedding trees that combines the inductive embedding of trees in expanding graphs by Haxell~\cite{haxell} (developing work by Friedman and Pippenger~\cite{FP1987}) with a `roll-back' idea by Johannsen (see~\cite{draganic2022rolling}), as used by Glebov, Johannsen, and Krivelevich~\cite{Glebov2013}, and Draganic, Krivelevich, and Nenadov~\cite{draganic2022rolling}. For embedding part of the tree into our vortex while making sure we cover all of the unused vertices (as discussed above) we use results of the first author from~\cite{montgomery2019}, which we develop into the form we require in Section~\ref{sec:embedtocover}.

\medskip

\noindent \textbf{Case $k\ge 3$.} Having proved the $k=2$ case, we will then proceed by induction on $k$. To prove Theorem~\ref{theorem:main:1} for $k$, we will have a graph $G$ with $(k-1)(n-1)+m$ vertices and an $n$-vertex tree $T$ with maximum degree $\Delta$, and $m\leq \mu n$ for some small $\mu>0$, and look to show that either $G$ contains a copy of $T$ or $G^c$ contains a copy of $K^{k-1}_{\mu n}\times K^c_m$. We proceed differently according to three cases, \textbf{a)}--\textbf{c)}, the proof of which are carried out in Sections~\ref{sec:manyleaves}--\ref{sec:fewleaveswellconnected} where a more detailed proof sketch for each case can be found.

\smallskip

\textbf{a)} In this case, $T$ has linearly (in $n$) many leaves, and we follow a method of Balla, Pokrovskiy, and Sudakov~\cite{balla2018}.%If there is a set $U\subset V(G)$ with $|U|\geq \mu n$  and $|U\cup N_G(U)|<n$, then we can take a set $W$ of $(k-2)(n-1)+m$ vertices in $V(G)\setminus (U\cup N_G(U))$ and apply induction to get either a copy of $T$ in $G[W]$, or a copy of $K^{(k-2)}_{\mu n}$ in $G^c[W]$ which, combined with $U$, would give a copy of $K^{(k-1)}_{\mu n}\times K^c_m$ in $G^c$.
%Thus, we can assume that $|U\cup N_G(U)|\geq n$ for each $U\subset V(G)$ with $|U|=\mu n$. Following Balla, Pokrovskiy and Sudakov~\cite{balla2018}, this condition allows us to remove some leaves of $T$, embed the remaining subtrees, and then show Hall's matching criterion holds in the correct subgraph to attach leaves to this embedding to get a copy of $T$.

\smallskip

\textbf{b)} In this case, $G$ is not well connected in the sense that there is a small set $V_0$ such that $G-V_0$ has two large disjoint parts. We can assume (using induction) that the subgraphs of $G$ induced on these large parts both satisfy an expansion condition. If there is a vertex with $\Delta$ neighbours in each part then we can embed part of the tree $T$ in each part, connected together appropriately by this vertex. Where there is no such vertex, we partition $V(G)$ into two large sets which have few edges between them in $G$. If $G$ has no copy of $T$, then we can apply induction to both subgraphs induced by the sets of this partition to find complete partite graphs in their complement, which we combine before removing some vertices in the larger classes in order to find our desired complete partite graph in $G^c$.

\smallskip

\textbf{c)} In the final case, $G$ is well-connected and $T$ has few leaves. The embedding we use here is the most complicated one in the inductive step. Effectively, though, we use expansion conditions to embed $T$ with many bare paths (i.e., paths which no branching vertices in $T$) removed, which exist as $T$ has few leaves, before using a results on the Ramsey goodness of paths to find paths that we then  connect into the embedding of $T$ using the connectivity property. This sketch serves more to motivate our eventual embedding, which cannot follow this sketch closely, but will be discussed in full in Section~\ref{sec:fewleaveswellconnected}.

\subsection{Organisation of the paper}\label{sec:overview}
In the rest of this section, we will first show, in Section~\ref{sec:subd} that the case $k=2$ of Theorem~\ref{theorem:main:1} can be divided into 2 subcases, covered by Theorem~\ref{theorem:dense} and Theorem~\ref{theorem:k=2}, where the latter represents the critical case sketched in Section~\ref{sec:sketch}. We then find our tree decompositions (Section~\ref{sec:treedecomp}), give some basic results involving expansion properties (Section~\ref{sec:furtherexpand}), and cover the tree embedding tools we will use (in Section~\ref{sec:extendability}). In Section~\ref{sec:proof:k=2:dense}, we prove Theorem~\ref{theorem:dense}. In Section~\ref{sec:embedtocover}, we show how to embed trees while covering a vertex set, to be applied at each stage when embedding the tree in the critical case. This then allows us to prove the result in this critical case, Theorem~\ref{theorem:k=2}, in Section~\ref{section:k=2}. Finally, in Section~\ref{sec:largerk}, we use induction on $k$ to prove Theorem~\ref{theorem:main:1} from the $k=2$ case.

%%%%%%%%%%%%%%%%%%%%%%%%%%%%%%%%%%%%%%%%%%%%%%%%%%%%%%%%%%%%%%%%%%%%%%%%%%%%%%%%%%%%%%%%%%%%%%%%%%%%%%%%%%%%%%%%%%%%%%%%%%%%%%%%%%%%%%%%%%%%%%%%%%%%%%%%%%%%%
%%%%%%%%%%%%%%%%%%%%%%%%%%%%%%%%%%%%%%%%%%%%%%%%%%%%%%%%%%%%%%%%%%%%%%%%%%%%%%%%%%%%%%%%%%%%%%%%%%%%%%%%%%%%%%%%%%%%%%%%%%%%%%%%%%%%%%%%%%%%%%%%%%%%%%%%%%%%%
%%%%%%%%%%%%%%%%%%%%%%%%%%%%%%%%%%%%%%%%%%%%%%%%%%%%%%%%%%%%%%%%%%%%%%%%%%%%%%%%%%%%%%%%%%%%%%%%%%%%%%%%%%%%%%%%%%%%%%%%%%%%%%%%%%%%%%%%%%%%%%%%%%%%%%%%%%%%%
%%%%%%%%%%%%%%%%%%%%%%%%%%%%%%%%%%%%%%%%%%%%%%%%%%%%%%%%%%%%%%%%%%%%%%%%%%%%%%%%%%%%%%%%%%%%%%%%%%%%%%%%%%%%%%%%%%%%%%%%%%%%%%%%%%%%%%%%%%%%%%%%%%%%%%%%%%%%%
%%%%%%%%%%%%%%%%%%%%%%%%%%%%%%%%%%%%%%%%%%%%%%%%%%%%%%%%%%%%%%%%%%%%%%%%%%%%%%%%%%%%%%%%%%%%%%%%%%%%%%%%%%%%%%%%%%%%%%%%%%%%%%%%%%%%%%%%%%%%%%%%%%%%%%%%%%%%%
%%%%%%%%%%%%%%%%%%%%%%%%%%%%%%%%%%%%%%%%%%%%%%%%%%%%%%%%%%%%%%%%%%%%%%%%%%%%%%%%%%%%%%%%%%%%%%%%%%%%%%%%%%%%%%%%%%%%%%%%%%%%%%%%%%%%%%%%%%%%%%%%%%%%%%%%%%%%%

\subsection{Subcases when $k=2$}\label{sec:subd}
As discussed in Section~\ref{sec:sketch}, the case when $k=2$ in Theorem~\ref{theorem:main:1} is equivalent to the following result, which for convenience we state separately.

\begin{theorem}\label{corollary:k=2} Let $\mu\ll 1/\Delta$ and $1\le m\leq \mu n$, and let $T$ be an $n$-vertex tree with $\Delta(T)\leq \Delta$. Then, $R(T,K_{m,\mu n})=n+m-1$.
\end{theorem}

We will separate Theorem~\ref{corollary:k=2} into two subcases, but first need the following definition for the main forms of expansion used throughout this paper.

	\begin{definition}\label{defn:exp}
	Let $m,m'\in\mathbb N$, $d>0$ and let $G$ be a graph.
\begin{enumerate}[label = \roman{enumi})]
\item We say that $G$ is a \textit{$(d,m)$-expander} if $|N(U)|\ge d|U|$ for all subsets $U\subset V(G)$ with $|U|\le m$.\label{defn:i}
\item We say that $G$ is \textit{$(m,m')$-joined} if every pair of disjoint sets $A$ and $B$, with $|A|= m$ and $|B|= m'$, has an edge between them in $G$. When $m=m'$, we simply say the graph is \textit{$m$-joined}.\label{defn:ii}
\end{enumerate}
\end{definition}
Note that a graph is $(m,m')$-joined exactly when its complement is $K_{m,m'}$-free. Using these definitions, we can now state two results, Theorems~\ref{theorem:dense} and~\ref{theorem:k=2}, before proving that Theorem~\ref{corollary:k=2} follows from them using an additional result, Proposition~\ref{prop:expander:subgraphnew}. Theorem~\ref{theorem:dense} covers Theorem~\ref{corollary:k=2} when $m$ is at most some constant depending on $\mu$, while Theorem~\ref{theorem:k=2}, roughly speaking, will cover Theorem~\ref{corollary:k=2} when $m$ is larger and some stronger expansion condition holds.
For convenience in their proofs, we adjust the parameters so that they are applied to trees with $n-m+1$ vertices in an $n$-vertex graph (instead of $n$ and $n+m-1$, respectively).

\begin{theorem}\label{theorem:dense}Let $\mu\ll 1/ \Delta,1/m$ with $m\leq \mu n$, and let $G$ be an $n$-vertex $(m,\mu n)$-joined graph.

Then, $G$ contains a copy of every $(n-m+1)$-vertex tree $T$ with $\Delta(T)\le \Delta$.
\end{theorem}
\begin{theorem}\label{theorem:k=2} Let $1/D,1/m,\mu\ll1/\Delta$ and $m\leq \mu n$, and let $G$ be an $n$-vertex $(m,\mu n)$-joined graph which is a $(D,m)$-expander.

Then, $G$ contains a copy of every $(n-m+1)$-vertex tree $T$ with $\Delta(T)\le \Delta$.
\end{theorem}
Before showing these two results imply Theorem~\ref{corollary:k=2}, we need to find an expander in the sense of Definition~\ref{defn:exp}\ref{defn:i} in our graphs covered by this theorem, which we do in the following proposition.

	% satisfy $n_0\ge (2d+2)m$
\begin{proposition}\label{prop:expander:subgraphnew}Let $n_0,m\in\mathbb N$ and $d>0$, and let $G$ be an $(m,n_0)$-joined graph with $|G|\geq n_0+(2d+2)m$.
Then, there exists some $W\subset V(G)$ such that $|W|<m$ and, setting $m'=m-|W|$, $G-W$ is an $(m',n_0+dm)$-joined $(d,m)$-expander.
\end{proposition}
\begin{proof}
Let $W\subset V(G)$ be a maximal subset subject to $|W|< 2m$ and $|N_G(W)|\leq d|W|$, possible as $W=\emptyset$ satisfies these conditions. Note that if $|W|\geq m$, then, as $G$ is $(m,n_0)$-joined, $|V(G)\setminus (W\cup N_G(W)|<n_0$, so that $|N_G(W)|\geq |G|-|W|-n_0\geq 2dm> d|W|$, a contradiction. Thus, we have $|W|<m$. Let $m'=m-|W|\geq 1$.

Now, if $U\subset V(G)\setminus W$ with $|U|\leq m$ and $U\neq\emptyset$, then, as $|W|< m$, we must have that $|U\cup W|< 2m$ and $|U\cup W|>|W|$. To avoid contradicting the choice of $W$, we must have $|N_G(U\cup W)|\geq d|U\cup W|$, and hence $|N_{G-W}(U)|\geq d|U\cup W|-d|W|\geq d|U|$. Thus, $G-W$ is an $(d,m)$-expander.

Furthermore, if $X,Y\subset V(G-W)$ are disjoint sets with $|X|=m'$, $|Y|=n_0+dm$, and $e_{G'}(X,Y)=0$, then $e_G(X\cup W,Y\setminus N_G(W))=0$, $|X\cup W|=m$ and $|Y\setminus N_G(W)|\ge n_0+dm-d|W|\geq n_0$, which contradicts that $G$ is $(m,n_0)$-joined. Thus, $G-W$ is $(m',n_0+dm)$-joined.
\end{proof}

We can now deduce Theorem~\ref{corollary:k=2} from Theorems~\ref{theorem:dense} and~\ref{theorem:k=2} using Proposition~\ref{prop:expander:subgraphnew}, as follows.

%From this two results, we can deduce the case $k=2$ of Theorem~\ref{theorem:main:1}, which we state as the following corollary.

\begin{proof}[Proof of Theorem~\ref{corollary:k=2}] Let $\mu\ll 1/\Delta$ and $1\le m\leq \mu n$, let $T$ be an $n$-vertex tree with $\Delta(T)\leq \Delta$, and let $G$ be a graph on $n+m-1$ vertices such that $G^c$ does not contain a copy of $K_{m,\mu n}$. To prove Theorem~\ref{corollary:k=2}, then, we need to show that $G$ contains a copy of $T$.
Using that $\mu\ll 1/\Delta$, let $\bar{\mu}$ and $\bar{m}$ satisfy $\mu\ll\bar{\mu}\ll 1/\bar{m} \ll 1/\Delta$ and let $D=\bar{\mu} n/100m\geq \bar{\mu}/100\mu$ (as $m\leq \mu n$), so that $1/D\ll 1/\Delta$.

We have $n\geq \bar{\mu} n/2+(2D+2)m$, and that $G$ is $(m,\bar{\mu} n/2)$-joined as its complement is $K_{m,\mu n}$-free and $\mu\ll\bar{\mu}$. Thus, using Proposition \ref{prop:expander:subgraphnew} and $D m\leq \bar{\mu}n/2$, we can find a set $W\subset V(G)$ with $|W|<m$ such that, setting $m'=m-|W|$ and $G'=G-W$, $G'$ is an $(m',\bar{\mu} n)$-joined $(D,m)$-expander. Note that $|G'|=n+m'-1$ and $m'\leq m\leq \mu n\leq \bar{\mu}n$.

If we have $m'<\bar{m}$, then, as $\bar{\mu}\ll 1/\bar{m},1/\Delta$ and $G'$ is an $(m',\bar{\mu} n)$-joined graph with $n+m'-1$ vertices, $G'$ contains a copy of $T$ by Theorem~\ref{theorem:dense}.
If $m'\ge \bar{m}$, then since $G'$ is a $(D,m)$-expander and $m'\leq m$, it is also a $(D,m')$-expander. Then, as $1/\bar{m},1/D,\bar{\mu}\ll 1/\Delta$, $m'\geq \bar{m}$, and $G'$ is $(m',\bar{\mu} n)$-joined, $G'$ contains a copy of $T$ by Theorem~\ref{theorem:k=2}, as required.
\end{proof}

%. %Firstly, note that if $m_1\leq m$, then Theorem~\ref{theorem:dense} shows that $G$ contains a copy of $T$ as $\mu\ll 1/m,1/\Delta$. Therefore, we can assume that $m_1\geq m$. %Then, awe have $D\geq \mu n/100m\geq \bar{m}/100m$, so that $1/D\ll 1/\Delta$.

%%%%%%%%%%%%%%%%%%%%%%%%%%%%%%%%%%%%%%%%%%%%%%%%%%%%%%%%%%%%%%%%%%%%%%%%%%%%%%%%%%%%%%%%%%%%%%%%%%%%%%%%%%%%%%%%%%%%%%%%%%%%%%%%%%%%%%%%%%%%%%%%%%%%%%%%%%%%%%%
%%%%%%%%%%%%%%%%%%%%%%%%%%%%%%%%%%%%%%%%%%%%%%%%%%%%%%%%%%%%%%%%%%%%%%%%%%%%%%%%%%%%%%%%%%%%%%%%%%%%%%%%%%%%%%%%%%%%%%%%%%%%%%%%%%%%%%%%%%%%%%%%%%%%%%%%%%%%%%%

\subsection{Structural decompositions of trees}\label{sec:treedecomp}
To show our cases cover all trees, we will need the following useful result which states that, if a tree has few leaves, then it has many bare paths in compensation.
\begin{lemma}[{\cite[Lemma 2.1]{Krivelevich2010}}]\label{lemma:leaves}
Let $k, \ell,n\in\mathbb N$  and let $T$ be a tree with $n$ vertices and at most $\ell $ leaves. Then $T$ contains at least $\frac n{k+1}-2 \ell+2$ vertex disjoint bare paths, each of length $k$.
\end{lemma}

We now give our main tree decomposition, decomposing the edges of a bounded degree tree $T$ into subtrees $T_1,\ldots ,T_\ell$ such that $T_i$ and $T_{i+1}$ intersect only on a leaf of $T_i$, as in the following definition.

\begin{definition} Let $\ell\in \mathbb N$ and let $\mathbf{n}=(n_1,\ldots,n_\ell)\in \mathbb N^{\ell}$. An $\mathbf{n}$-decomposition of a tree $T$ is a tuple $(T_1,\ldots, T_\ell)$ of edge-disjoint subtrees of $T$ such that
\begin{enumerate}[label = \roman{enumi})]
    \item $E(T)=E(T_1)\cup\ldots\cup E(T_\ell)$,
    \item $|T_1|=n_1$, and $|T_i|=n_i+1$ for each $2\leq i\leq \ell$, and
    \item for each $i\in [\ell-1]$, $T_i$ contains no vertices in $T_{i+2}\cup \ldots\cup T_\ell$, and $V(T_i)\cap V(T_{i+1})$ contains exactly one vertex, which is a leaf of $T_i$. %the only vertex of $T_i$ in
		%for every $i\in[\ell]$, there is a (single) edge between $T_{i-1}$ and $T_{i}$ in $T$.%unique $u_i\in V(T_i)$ and $v_{i-1}\in V(T_{i-1})$ such that $u_iv_{i-1}\in E(T)$.
\end{enumerate}
\end{definition}
We want to find such a decomposition where the trees $T_1, \ldots, T_\ell$ decrease in size by approximately a constant factor each time. That is, an $\mathbf{n}$-decomposition where $\mathbf{n}$ is \emph{$(\gamma_1,\gamma_2)$-descending}, as follows.

\begin{definition}
For $0<\gamma_1<\gamma_2<1$, we say that a tuple $\mathbf{n}=(n_1,\ldots,n_\ell)\in \mathbb N^{\ell}$ is \emph{$(\gamma_1,\gamma_2)$-descending} if $\gamma_1n_i\le n_{i+1}\le\gamma_2n_i$ for every $i\in [\ell-1]$.
\end{definition}

We will find descending decompositions of trees, by iteratively using the following proposition.

\begin{proposition}\label{prop:splittree}
Let $\gamma\in (0,1)$ and $n,\Delta\in\mathbb N$ satisfy $\gamma n\ge \Delta$ and $n\geq 2$. Let $T$ be a tree with $n$ vertices and $\Delta(T)\le \Delta$, and let $t\in V(T)$. Then, there exists two subtrees $T_1, T_2\subseteq T$ such that \begin{enumerate}[label = \upshape\roman{enumi})]
    \item $E(T)=E(T_1)\cup E(T_2)$,
    \item $T_1$ and $T_2$ share exactly one vertex $v$, which is a leaf of $T_1$,
    \item $t\in V(T)\setminus V(T_2)$, and
    \item $\gamma n\ge|T_2|\ge \gamma n/2\Delta$.
    % and there is a unique edge connecting it with $T'$.
\end{enumerate}
\end{proposition}
\begin{proof}
View $T$ as being rooted at $t$. Let $v$ be a vertex furthest from $t$ subject to the condition that the subtree $T_2$ of $T$ rooted at $v$ has size $|T_2|\geq \gamma n/2\Delta$. Note that this is possible as $t$ has at most $\Delta(T)\leq \Delta$ neighbours, the subtree rooted at one of these must have size at least $(n-1)/\Delta\geq n/2\Delta$ and hence is a candidate for $v$ as $\gamma<1$. Moreover, this shows $v\not=t$. Let $T_1$ be subgraph of $T$ induced by the vertex set $(V(T)\setminus V(T_2))\cup\{v\}$. Note that $T_1$ is a subtree of $T$ and $v$ is leaf in $T_1$. It follows that $T_1,T_2$ satisfies conditions i), ii) and iii). For condition iv), let $v'$ be any neighbour of $v$ in $T_2$. To avoid contradicting the definition of $v$, the subtree rooted at $v'$ must have size at most $\floor{\gamma n/2\Delta}$. Since there are at most $\Delta$ such $v'$, we have $|T_2|\leq1+\Delta\cdot\floor{\gamma n/2\Delta}\leq\gamma n$, as required.
%Let us consider a partial ordering of $V(T)$ defined in the following way. We say that a vertex $v$ is above a vertex $u$ in the ordering if $u$ lies in the unique path connecting $v$ with $t$. For each vertex $v$ in $T$, denote the tree induced by all vertices above $v$ by $T_v$. Let $x\in V(T)$ be a maximal vertex so that the $|T_x|>\gamma n$. Let $N^+(x)$ denote the set of neighbours of $x$ which are above $x$, and let $y$ be a vertex in $N^+(x)$ maximising $|T_y|$. Then, by maximality of $x$ we have $|T_y|\leq \lfloor\gamma n\rfloor$. Moreover,
%\[\lfloor\gamma n\rfloor+1\leq|T_x|=1+\sum_{v\in N^+(x)}|T_v|\leq 1+|N^+(x)||T_v|\leq 1+\Delta |T_v|,\] and so $|T_y|\geq\lfloor\gamma n\rfloor/\Delta\ge \gamma n/2\Delta$. Finally, it is easy to see that $t\not\in T_y$, $T-T_y$ is still a tree and $xy$ is the unique edge connecting $T_y$ and $T-T_y$.
\end{proof}

We can now find our desired descending decompositions of large bounded degree trees, as follows.
%The next key proposition shows that every large tree with bounded degree has an $(\mathbf{n},\gamma_1,\gamma_2)$-decomposition for suitable choices of $\mathbf{n}$ and $\gamma_1<\gamma_2$.
\begin{lemma}\label{lemma:ngammadecomposition}
Let $0<\gamma<1/2$ and let $n,N,\Delta\in\mathbb N$ satisfy $n-1>N\geq\Delta/\gamma$. Then for any $n$-vertex tree $T$ with maximum degree at most $\Delta$ and any $t\in V(T)$,  there exists a $(\frac\gamma{4\Delta},2\gamma)$-descending  tuple $\mathbf{n}=(n_1,\ldots,n_\ell)\in\mathbb{N}^{\ell}$ with $\frac\gamma{3\Delta}N\leq n_\ell\leq N$ such that $T$ has an $\mathbf{n}$-decomposition $(T_1,\ldots, T_\ell)$ with $t\in V(T_1)$.
\end{lemma}
\begin{proof}
We construct the desired decomposition iteratively, beginning with $T_1'=T$, and after $i$ iterations obtain a sequence of trees  $(T_1,T_2,\ldots,T_{i-1},T_i')$ forming a $(|T_1|,|T_2|-1,\ldots,|T_{i-1}|-1,|T_i'|-1)$-decomposition of $T$.
For each $i\geq 1$, if $|T_i'|-1\leq N$, then stop iterating.
Otherwise, let $t_i$ be the unique common vertex of $T_i'$ and $T_{i-1}$ if $i\geq 2$ and let $t_1=t$.
Noting that $|T_i'|>N\geq \Delta/\gamma$, we can apply Proposition~\ref{prop:splittree} to obtain subtrees $T_i$ and $T'_{i+1}$ of $T_{i}'$ such that $E(T_i')=E(T_i)\cup E(T'_{i+1})$, $T_i$ and $T_{i+1}'$ share a unique vertex $t_{i+1}$ which is a leaf of $T_i$, $t_i\in V(T_i)\setminus V(T'_{i+1})$, and $\gamma |T_i'|\geq |T'_{i+1}|\geq \gamma |T_i'|/2\Delta$. This implies that $(T_1,T_2,\ldots,T_{i},T_{i+1}')$ is a $(|T_1|,|T_2|-1,\ldots,|T_{i}|-1,|T_{i+1}'|-1)$-decomposition of $T$.

As $\gamma<1$, this process must end; suppose it ends with a sequence $(T_1,\ldots,T_{\ell-1},T_\ell')$. Let $n_1=|T_1|$, $n_i=|T_i|-1$ for each $2\leq i\leq\ell-1$, $n_\ell=|T'_\ell|-1$, and $\mathbf{n}=(n_1,\ldots,n_\ell)$, so that $(T_1,\ldots,T_{\ell-1},T_\ell')$ is an $\mathbf{n}$-decomposition of $T$. Noting that $\ell\geq 2$ as $|T_1'|=n>N+1$, and since the process stopped after $\ell$ and not $\ell-1$ iterations, we have $n_\ell\leq N$ and $|T_{\ell-1}'|>N$, the latter of which implies $n_\ell\geq \gamma |T_{\ell-1}'|/2\Delta-1>\gamma N/3\Delta$. It is left then only to show that $\mathbf{n}$ is $(\frac\gamma{4\Delta},2\gamma)$-descending.

For each $i\in[\ell-1]$, we have $|T_{i+1}'|\leq \gamma |T_i'|$, so that $|T_i'|\geq |T_i|=1+|T_i'|-|T_{i+1}'|\geq 1+(1-\gamma)|T_i'|$. Thus, for every $i\in[\ell]$, we have $(1-\gamma)|T_i'|\leq n_i\leq |T_i'|$. Then, for each $i\in[\ell-1]$, we have $n_{i+1}\leq |T_{i+1}'|\leq \gamma |T_i'|\leq \gamma n_i/(1-\gamma)\leq 2\gamma n_i$. Furthermore, for each $i\in[\ell-1]$, we have $n_{i+1}\geq (1-\gamma)|T_{i+1}'|\geq (1-\gamma)\gamma |T_i'|/2\Delta\geq (1-\gamma)\gamma n_i/2\Delta\geq \gamma n_i/4\Delta$. Thus, $\mathbf{n}$ is $(\frac\gamma{4\Delta},2\gamma)$-descending, as required.
%For each $0\leq i\leq \ell-2$, we have $\frac\gamma{2\Delta}|T_{i}'|\leq|T_{i+1}'|\leq\gamma|T_{i}'|$ and $\frac\gamma{2\Delta}|T_{i+1}'|\leq|T_{i+2}'|\leq\gamma|T_{i+1}'|$. It follows that
%\[\frac{\gamma}{3\Delta}\leq\frac{\gamma}{2\Delta-\gamma}\cdot (1-\gamma)\leq\frac{|T_{i+1}|}{|T_{ i}|}=\frac{|T_{i+1}'|}{|T_{ i}|}\cdot \frac{|T_{i+1}|}{|T_{ i+1}'|}\leq\frac{\gamma}{1-\gamma}\cdot \frac{\Delta-\gamma}{\Delta}\leq 2\gamma.\]
%Similarly, $\frac\gamma{2\Delta}|T_{\ell-1}'|\leq|T_{\ell}'|\leq\gamma|T_{\ell}'|$, and so
%\[\frac\gamma{3\Delta}|T_{\ell-1}|\leq\frac{\gamma}{2\Delta-\gamma}|T_{\ell-1}|\leq|T_\ell'|\leq\frac\gamma{1-\gamma}|T_{\ell-1}|\leq 2\gamma|T_{\ell-1}|.\]
%Hence, $(T_0,\dots,T_{\ell-1},T_\ell')$ is an $(\mathbf{n},\frac\gamma{3\Delta},2\gamma)$-decomposition of $T$. Finally, since $|T_{\ell-1}'|>N$, we have $\frac{\gamma}{3\Delta}N\leq |T_\ell'|\leq N$, as required.
\end{proof}

We will use Lemma~\ref{lemma:ngammadecomposition} directly, as well as through the following corollary.

\begin{corollary}\label{corollary:ngammadecomposition}
Let $0<\gamma<1/4$ and let $n,k,\Delta\in\mathbb N$ satisfy $n>(8\Delta/\gamma)^{k+1}$. Then for any $n$-vertex tree $T$ with maximum degree at most $\Delta$ and any $t\in V(T)$,  there exists a $(\frac\gamma{8\Delta},2\gamma)$-descending  tuple $\mathbf{n}=(n_1,\ldots,n_k)\in\mathbb{N}^{k}$ such that $T$ has an $\mathbf{n}$-decomposition $(T_1,\ldots,T_k)$ with $t\in V(T_1)$.
\end{corollary}
\begin{proof} Let $N=\frac{1}{2}(\frac{\gamma}{8\Delta})^kn$ and note that $N>\Delta/\gamma$. Then, by Lemma~\ref{lemma:ngammadecomposition}, there is  a $(\frac\gamma{8\Delta},\gamma)$-descending tuple $\mathbf{n}'=(n'_1,\ldots,n'_\ell)\in\mathbb{N}^{\ell}$ with $\frac\gamma{4\Delta}N\leq n_\ell\leq N$ such that $T$ has an $\mathbf{n}$-decomposition $(T_1',\ldots,T_\ell')$. Now, as $N\geq |T_{\ell}'|-1\geq (\gamma/8\Delta)^\ell|T_1'|\geq (\gamma/8\Delta)^\ell n/2$, we have that $\ell\geq k$. Let $T_i=T_i'$ for each $i\in [k-1]$ and $T_k=\bigcup_{i=k}^\ell T_i'$. Let $n_i=n'_i$ for each $i\in [k-1]$, and $n_k=\sum_{i=k}^{\ell}n_i'-(\ell-k)$, and let $\mathbf{n}=(n_1,\ldots,n_k)$. Note that $\mathbf{n}$ is $(\frac\gamma{8\Delta},2\gamma)$-descending, and $(T_1,\ldots,T_k)$ is an $\mathbf{n}$-decomposition of $T$.
\end{proof}

%%%%%%%%%%%%%%%%%%%%%%%%%%%%%%%%%%%%%%%%%%%%%%%%%%%%%%%%%%%%%%%%%%%%%%%%%%%%%%%%%%%%%%%%%%%%%%%%%%%%%%%%%%%%%%%%%%%%%%%%%%%%%%%%%%%%%%%%%%%%%%%%%%%%%%%%%%%%%%%
%%%%%%%%%%%%%%%%%%%%%%%%%%%%%%%%%%%%%%%%%%%%%%%%%%%%%%%%%%%%%%%%%%%%%%%%%%%%%%%%%%%%%%%%%%%%%%%%%%%%%%%%%%%%%%%%%%%%%%%%%%%%%%%%%%%%%%%%%%%%%%%%%%%%%%%%%%%%%%%

\subsection{Another expansion property}\label{sec:furtherexpand}
We will use the following variant of Proposition~\ref{prop:expander:subgraphnew}, which is proved in a similar manner.

\begin{restatable}{proposition}{propexpand}\label{prop:expander:subgraphnewnew}
Let $n_0,m\in\mathbb N$ and $d>0$, and let $G$ be an $(m,n_0)$-joined graph which contains a set $V\subset V(G)$ with $|V|\geq n_0+(2d+2)m$.
Then, there exists some $W\subset V(G)$ such that $|W|<m$ and, for each $U\subset V(G)\setminus W$ with $|U|\leq m$,
$|N_G(U,V\setminus W)|\geq d|U|$.% setting $m'=m-|W|$, $G-W$ is an $(m',n_0+dm)$-joined $(d,m)$-expander.
\end{restatable}
\begin{proof}
Let $W\subset V$ be a maximal subset subject to $|W|< 2m$ and $|N_G(W,V)|\leq d|W|$, possible as $W=\emptyset$ satisfies these conditions. Note that if $|W|\geq m$, then, as $G$ is $(m,n_0)$-joined, $|V\setminus (W\cup N_G(W)|<n_0$, so that $|N_G(W,V)|\geq |V|-|W|-n_0\geq 2dm\geq d|W|$, a contradiction. Thus, we have $|W|<m$.

Now, if $U\subset V(G)\setminus W$ satisfies $|U|\leq m$ and $U\neq\emptyset$, then, as $|W|< m$, we have $|U\cup V|< 2m$ and $|U\cup W|>|W|$. To avoid contradicting the choice of $W$, we must have $|N_G(U\cup W,V)|\geq d|U\cup W|$, and hence $|N_{G}(U,V\setminus W)|\geq d|U\cup W|-d|W|\geq d|U|$. Thus, for each $U\subset V(G)$ with $|U|\leq m$,
$|N_G(U,V\setminus W)|\geq d|U|$.
\end{proof}

\subsection{Tree embeddings and $(d,m)$-extendability}\label{section:extendability}\label{sec:extendability}
As noted in Section~\ref{sec:sketch}, to carry out our tree embedding schemes we use techniques that come from combining a development by Haxell~\cite{haxell} of a tree embedding method of Friedman and Pippenger~\cite{FP1987} with a `roll-back' idea of Johannsen (see~\cite{draganic2022rolling}) as used by Draganic, Krivelevich and Nenadov~\cite{draganic2022rolling}. We use the language of extendability (first used by Glebov, Johannsen, and Krivelevich~\cite{Glebov2013}): in a graph with certain expansion conditions we embed a tree iteratively, each time adding a leaf so that the embedded subgraph maintains an `extendability condition'. Key to the flexibility of this method is that removing a leaf (`rolling-back') maintains the extendability condition.

%In 1987, Friedman and Pippenger~\cite{FP1987} introduced a vertex-by-vertex embedding technique to find small trees in expander graphs. By considering a different notion of expansion, Haxell~\cite{haxell} improved Friedman and Pippenger's method to allow the embedding of larger trees, and since then, this kind of techniques have played a major role in embedding problems in sparse graphs. We will use language introduced by Glebov, Johannsen, and Krivelevich~\cite{Glebov2013}, who modified Haxell's method to give a much more flexible framework for embedding trees.

The key definition here is that of a \textit{$(d,m)$-extendable subgraph} $S$ in a graph $G$, which is a subgraph with the property that every small subset of $V(G)$ has many neighbours outside $S$. Note that the following definition uses the set $N'_G(U)=\cup_{v\in U}N_G(v)$.
\begin{definition}Let $d,m\in\mathbb N$ be such that $d\ge 3$ and $m\ge 1$. Let $G$ be a graph and let $S\subset G$ be a subgraph. We say that $S$ is \emph{$(d,m)$-extendable in $G$} if $S$ has maximum degree at most $d$ and, for all $U\subset V(G)$ with $1\le |U|\le 2m$, we have
    \begin{equation}\label{def:extendability}
        |N_G'(U)\setminus V(S)|\ge (d-1)|U|-\sum_{u\in U\cap V(S)}(d_S(u)-1).
    \end{equation}
\end{definition}
The next lemma shows how to verify the extendability condition using only the external neighbourhood.

\begin{proposition}\label{prop:weakerextendability}
Let $d,m\in\mathbb N$ be such that $d\ge 3$ and $m\ge 1$. Let $G$ be a graph and let $S\subset G$ be a subgraph with maximum degree at most $d$. If, for all $U\subset V(G)$ with $1\leq|U|\leq 2m$ we have $|N(U,V(G)\setminus V(S))|\geq d|U|$, then $S$ is $(d,m)$-extendable.
\end{proposition}
\begin{proof} For each $U\subset V(G)$ with $1\leq |U|\leq 2m$, we have
\[
|N'(U)\setminus V(S)|\geq|N(U)\setminus V(S)|=|N(U,V(G)\setminus V(S))|\geq d|U|\geq(d-1)|U|-\sum_{u\in U\cap V(S)}(d_S(u)-1),
\]
as required.
\end{proof}
The next three lemmas are the core of the extendability method, and can be found as Lemmas 5.2.6, 5.2.7 and 5.2.8 in~\cite{Glebov2013}. They will allow us to manipulate an extendable graph by adding/removing a vertex or an edge while remaining extendable. The latter two are simple to verify, while the first follows the original inductive step in the argument by Haxell~\cite{haxell}. The first lemma we state in our $(m,m')$-joined language, and it follows from Lemma~5.2.6 in~\cite{Glebov2013} by observing that in such a graph $G$ any set $U$ of $m$ vertices satisfies $|N_G(U)|\geq |G|-m'-|U|= |G|-m'-m$.
%Although we state the lemmas in the more general framework of $(m,m')$-joined graphs, we will omit the proofs as they are straightforward modification of the original arguments from~\cite{Glebov2013,montgomery2019}.
%\begin{proposition}Let $d,m\in\mathbb N$ be such that $d\ge 3$ and $m\ge 1$. Let $G$ be a graph and let $S\subset G$ be a $(d,m)$-extendable subgraph of $G$. Suppose that every subset $U\subset V(G)$, with $m\le |U|\le 2m$, satisfies
%\begin{equation}
%    |N'_G(U)|\ge |S|+2dm+1.
%\end{equation}
%Then, for every $s\in V(S)$ with $d_S(s)\le d-1$, there exists $y\in N_G(s)\setminus V(S)$ such that $S+sy$ is $(d,m)$-extendable.
%\end{proposition}

\begin{lemma}[Adding a leaf]\label{lemma:adding:leaf}
Let $d,m,m'\in\mathbb N$ be such that $d\ge 3$ and $m,m'\ge 1$, and let $G$ be an $(m,m')$-joined graph. Let $S$ be a $(d,m)$-extendable subgraph of $G$ such that $|G|\ge |S|+(2d+2)m+m'+1$. Then, for every $s\in V(S)$ with $d_S(s)\leq d-1$, there exists $y\in N_G(s)\setminus V(S)$ such that $S+sy$ is $(d,m)$-extendable.
\end{lemma}

\begin{lemma}[Removing a leaf]\label{lemma:remove:leaves}Let $d,m\in\mathbb N$ be such that $d\ge 3$ and $m\ge 1$, let $G$ be a graph and let $S\subset G$ be a subgraph of $G$. Suppose that there exist vertices $s\in V(S)$ and $y\in N_G(s)\setminus V(S)$ so that $S+sy$ is $(d,m)$-extendable. Then, $S$ is $(d,m)$-extendable.
\end{lemma}
\begin{lemma}[Adding an edge]\label{lemma:adding:edge} Let $d,m\in\mathbb N$ be such that $d\ge 3$ and $m\ge 1$. Let $G$ be a graph and let $S$ be a $(d,m)$-extendable subgraph of $G$. If $s,t\in V(S)$ with $d_S(s),d_S(t)\le d-1$ and $st\in E(G)$, then $S+st$ is $(d,m)$-extendable in $G$.
\end{lemma}
We will use Lemma~\ref{lemma:adding:leaf} through the follow corollary where a tree is built on to an extendable subgraph by applying it iteratively (see also Corollary~3.7 in~\cite{montgomery2019}).%A direct corollary of these results is that we can attach a tree to an extendable graph so that the resulting subgraph remains extendable.
\begin{corollary}\label{corollary:extendable:embedding}
Let $d,m,m'\in\mathbb N$ be such that $d\ge 3$, and let $G$ be an $(m,m')$-joined graph. Let $T$ be a tree with $\Delta(T)\le d/2$ and let $R$ be a $(d,m)$-extendable subgraph of $G$ with maximum degree at most $d/2$. If $|R|+|T|\le|G|-(2d+2)m-m'$, then for every vertex $t\in V(T)$ and $v\in V(R)$, there is a copy $S$ of $T$ in $G-(V(R)\setminus \{v\})$ in which $t$ is copied to $v$ and $S\cup R$ is a $(d,m)$-extendable subgraph of $G$.
\end{corollary}

\subsection{Concentration results}\label{sec:conc}
We will use the following standard version of Chernoff's bound for binomial or hypergeometric random variables
(see, e.g.,~\cite{JLR2000} for the standard definition of such variables with parameters $n$ and $p$ and with parameters $N$, $n$ and $m$, respectively).
\begin{lemma}[see, e.g., Corollary 2.3 and Theorem 2.10 in~\cite{JLR2000}]\label{hypergeom}\label{lemma:chernoff}
Let $X$ be a hypergeometric random variable with parameters $N$, $n$ and $m$, or a binomial random variable with parameters $n$ and $p$. Then, for any $0<\eps\leq 3/2$,
\[
\P(|X-\E X|\geq \eps \E X)\leq 2\exp\left(-\frac{\eps^2}{3}\E X\right).
\]
\end{lemma}
A sequence of random variables  $(X_i)_{i \geq 0}$ is a submartingale if $\E[X_{i+1} \mid X_0, \ldots, X_i] \geq X_i$ for each $i\geq 0$. We will use the following Azuma-type bound for submartingales.

\begin{lemma}[see, e.g., \cite{wormald1999differential}]\label{lem:azuma}
Let $(X_i)_{i \geq 0}$ be a submartingale and let $c_i>0$ for each $i\geq 1$. If $|X_i -X_{i-1}| <c_i$ for each $i\geq 1$, then, for each $n\geq 1$,
\[
\mathbb{P}(X_n-X_0\leq t) \leq 2 \exp \left(-\frac{t^2}{2\sum_{i=1}^nc_i^2} \right).
\]
\end{lemma}

%% file: 3mconstant.tex
We now prove Theorem~\ref{theorem:dense}, which covers Theorem~\ref{corollary:k=2} when $m$ is at most some large constant depending on $\Delta$. As our $n$-vertex graph $G$ here is $(m,\mu n)$-joined (for some small $\mu$), sets with constant size (i.e., $m$) already have a neighbourhood covering almost all of the graph, making the graph very dense, with $\Theta(n^2/m)$ edges. Embedding a tree is not so difficult in this case, and, for example, we could use an approach of Balla, Pokrovskiy, and Sudakov~\cite{balla2018} if $T$ has linearly many leaves and an absorption approach like that in~\cite{montgomery2019} if $T$ does not have linearly many leaves. To take a unified approach, however, we will use an embedding inspired by work by the first author and Kathapurkar~\cite{kathapurkar2022spanning} and the first author and B{\"o}ttcher, Parczyk, and Person~\cite{bottcher2020embedding}.

To do this we first embed a small linear portion of the tree randomly, before extending this greedily vertex by vertex. Where a new leaf cannot be embedded because some embedded vertex, $w$ say, has no more unused neighbours, we show that a vertex in the set $U$ of (at least $m$) unused vertices can be swapped into the embedding to free up a neighbour of $w$, allowing the embedding to continue. This swapping property will follow from the random embedding of the first part of the tree, as the $(m,\mu n)$-joined property of $G$ implies that for some $u\in U$, many \emph{good} neighbours of $w$ have a large common neighbourhood with $u$, so that, for many of these good neighbours, $w'$ say, $w'$ will have a vertex of $T$ embedded to it with all of its neighbours in the embedding in $T$ embedded to common neighbours of $u$ and $w'$ (see Claim~\ref{claim:new}) -- exactly what we need to swap $w'$ with $u$, and then use $ww'$ to extend the embedding.
%, in combination with the -proany set of $m$ vertices in this case will have a large neighbourhood due to the random embedding of the first part of the tree.

\begin{proof}[Proof of Theorem~\ref{theorem:dense}]
To recap our situation, we have $\mu\ll 1/\Delta,1/m$ with $m\leq \mu n$, an $n$-vertex $(m,\mu n)$-joined graph $G$, and an $(n-m+1)$-vertex tree $T$ with $\Delta(T)\leq \Delta$ to embed into $G$. Note that, moreover, we can assume that $1/n\ll \mu$, as we could take some $\mu'$ with $\mu\ll \mu'\ll 1/\Delta,1/m$ and prove this with $\mu'$ in place of $\mu$ first, where $m\leq \mu n$ implies that $1/n\leq \mu$ and thus $1/n\ll \mu'$.

Let $\gamma$ and $\beta$ satisfy $\mu \ll \gamma\ll \beta \ll 1/\Delta,1/m$.
Applying Proposition~\ref{prop:splittree} with an arbitrary $t\in V(T)$, we find subtrees $T_0$ and $T_1$ of $T$ sharing exactly one vertex $t_1$, such that $E(T)=E(T_1)\cup E(T_2)$ and $4\beta n\leq |T_0|\leq 8\Delta\beta n$. Label the remaining vertices in $T$ as $t_2,\ldots,t_{n-m+1}$ such that $T[\{t_1,\ldots,t_i\}]$ is a tree for each $i\in [n-m+1]$, and $T[t_1,\ldots,t_{|T_0|}]=T_0$. Moreover, do this so that, for a set $J\subset \{4,\ldots,|T_0|\}$ with $|J|= \gamma n$, for each $j\in J$, $t_{j-3}t_{j-2}t_{j-1}t_j$ is a path in $T$
 and the neighbours of $t_j$ in $T$ except for $t_{j-1}$ (if there are any) appear in the labelling directly after $t_j$, and the vertices $t_j$, $j\in J$ are all at distance at least 3 apart in $T$ (that this is possible follows easily as $\Delta(T)\leq \Delta$ and $\gamma\ll \beta$).   %Furthermore, do this so that the vertices in $T_0$ are labelled in a breadth-first manner

Now, if $V(G)$ contains a set $U$ of $m$ vertices with degree at most $n/2m$, then $|N_G(U)|\leq m\cdot n/2m=n/2<n-\mu n-m$, and hence $G$ is not $(m,\mu n)$-joined, a contradiction. Therefore, setting $V_0$ to be the set of vertices in $V(G)$ with degree at least $n/2m$, we have $|V(G)\setminus V_0|<m$.

We now randomly embed $T_0$ into $G[V_0]$ vertex by vertex, as follows, creating an embedding $\phi$. Arbitrarily, choose $v_1\in V_0$ and set $\phi(t_1)=v_1$. Then, for each $2\leq i\leq |T_0|$ in turn, let $w_i$ be the image under $\phi$ of the sole neighbour of $t_i$ in $T[\{t_1,\ldots,t_i\}]$, choose $v_i\in N_G(w_i,V_0)\setminus \{v_1,\ldots,v_{i-1}\}$ uniformly at random and set $\phi(t_i)=v_i$.
Note that, for each $2\leq i\leq |T_0|$, we have that $w_i\in V_0$, and therefore
\[
|N_G(w_i,V_0)\setminus \{v_1,\ldots,v_{i-1}\}|\geq \frac{n}{2m}-m-|T_0|\geq \frac{n}{2m}-m-8\Delta\beta n\geq \frac{n}{4m},
\]
so that this embedding successfully embeds $T_0$. We will show the following claim.

\begin{claim}\label{claim:new} For each $U\subset V(G)$ with $|U|=m$ and $v\in V_0\setminus U$, with probability $1-o(n^{-m-1})$, there are at least $4 \mu n$ values of $j\in J$ such that $v_jv\in E(G)$ and $N_{\phi(T_0)}(v_j)\subset N_G(u)$ for some $u\in U$.
\end{claim}
\begin{proof} %We will show that the property in the lemma holds for each
Fix $U\subset V(G)$ with $|U|=m$ and $v\in V_0\setminus U$.
Now, for each $y\in V_0\setminus U$, as $|N_G(y)|\geq n/2m$ and $G$ is $(m,\mu n)$-joined, we have
\[|N_G(U)\cap N_G(y)\cap V_0)|\geq \frac{n}{2m}-m-\mu n-m\geq \frac{n}{4m}.
\]
%\[
%\big|\left(\cup_{u\in U}N_G(u)\right)\cap N_G(y,V_0)\big|\geq \frac{n}{2m}-m-\mu n\geq \frac{n}{4m}.
%\]
Thus, there is some $u\in U$ with $|N_G(u)\cap N_G(y)\cap V_0|\geq n/4m^2$. As this holds for every $y\in V_0\setminus U$ and $|N_G(v,V_0)\setminus U|\geq n/2m-2m\geq n/4m$, there is some $u\in U$ such that, for at least $n/4m^2$ vertices $y\in N(v,V_0)\setminus U$, we have $|N_G(u)\cap N_G(y)\cap V_0|\geq n/4m^2$. Say the set of these vertices $y$ is $Y_{v,u}$.

We now show that, with probability at least $1-o(n^{-(m+1)})$, for at least $4\mu n$ values of $j\in J$, we have $v_j\in Y_{v,u}$ and $N_{\phi(T_0)}(v_j)\subset N_G(u)$.
%$j\in J$, there are at least $\mu \Delta n$ values of $i\in [r]$ such that $v_i\in Y_v$ and $N_{S_2}(v_i)\subset N_G(u)$.
 Let $J=\{j_1,\ldots,j_{\gamma n}\}$, in order, and, for each $i\in [\gamma n]$,
 let $X_i$ be the indicator function for $v_{j_i}\in Y_{v,u}$ and $N_{\phi(T_0)}(v_{j_i})\subset N_G(u)$.

Consider $i\in [\gamma n]$ and suppose we have embedded $t_1,\ldots,t_{j_i-3}$. As $v_{j_i-3}\in V_0$, for each $y\in Y_{v,u}$, note that, as $|N_G(u)\cap N_G(y)\cap V_0|\geq n/4m^2$, all but at most $m$ vertices in $N_G(v_{j_i-3})$ have at least $n/8m^3$ neighbours in $|N_G(u)\cap N_G(y)\cap V_0|$, for otherwise, as before, we get that $G$ is not $(m,n/8m^2)$-joined, a contradiction as $n/8m^2\geq \mu n$.
Therefore, the probability that $v_{j_i-1}\in N_G(u)\cap N_G(y)\cap V_0$, $v_{j_i}=y$ and any subsequent neighbours of $t_{j_i}$ are embedded into $N_G(u)\cap N_G(y)\cap V_0$ is at least
\[
\frac{|N_G(v_{j_i-3},V_0)\setminus \{v_1,\ldots,v_{j_i-3}\}|-m}{|N_G(v_{j_i-3},V_0)\setminus \{v_1,\ldots,v_{j_i-3}\}|}\cdot \frac{n/8m^3}{n}\cdot \frac{1}{n}\cdot \left(\frac{n/4m^2-|T_0|}{n}\right)^\Delta\geq \frac{1}{n(8m)^{2\Delta+3}}.
\]
As these events are distinct for each $y\in Y_{v,u}$, we have that, conditioned on any possible values of $v_1,\ldots,v_{j_i-3}$,
\[
\mathbb{P}(X_i=1)\geq |Y_{v,u}|\cdot \frac{1}{n(8m)^{2\Delta+3}}\geq \frac{n}{4m^2}\cdot \frac{1}{n(8m)^{2\Delta+3}}\geq \frac{1}{(8m)^{2\Delta+5}}.
\]
Therefore, setting $\alpha=\frac{1}{(8m)^{2\Delta+5}}$ so that $\alpha\gg \mu$, and letting $Z_i=\sum_{i'=1}^i(X_{i'}-\alpha)$ for each $i\in [\gamma n]$, we have that $(Z_0,Z_1,\ldots,Z_{\gamma n})$ is a submartingale, using that, for all $1\leq i<i'\leq \gamma n$, $t_{j_i}$ and $t_{j_{i'}}$ have distance at least 3 apart in $T$, and $t_{j_i}$ and its neighbours appear in the sequence $t_1,\ldots,t_{|T_0|}$ before $t_{j_{i'}}$ or any of its neighbours.
Furthermore, for each $i\in [\gamma n]$, we have $|Z_i-Z_{i-1}|\leq 1$. Therefore, by Azuma's inequality (Lemma~\ref{lem:azuma}), we have 
\[
\mathbb{P}\Big(\sum_{i\in [\gamma n]}(X_i-\alpha )\leq \alpha \gamma n/2\Big)\leq 2\exp\Big(-\frac{(\alpha \gamma n/2)^2}{2\gamma n}\Big)=2\exp(-\alpha^2\gamma n/8),
\]
so that $\sum_{i\in [\gamma n]}X_i\geq \alpha \gamma n/2\geq 4\mu n$ with probability at least $1-2\exp(-\alpha^2\gamma n/8)=1-o(n^{-(m+1)})$, as required, using $\alpha,\gamma\gg \mu\gg 1/n$.
\end{proof}
Therefore, by Claim~\ref{claim:new} and a union bound, we can assume that $\phi$ has the property that, for each $U\subset V(G)$ with $|U|=m$ and $v\in V_0\setminus U$, there are at least $4\mu n$ values of $i\in J$ such that $v_iv\in E(G)$ and $N_{\phi(T_0)}(v_i)\subset N_G(u)$ for some $u\in U$. Next, we greedily extend the embedding of $T_0$ to one of $T$, where if we cannot simply extend the embedding by embedding $t_i$ we use the property from Claim~\ref{claim:new} to swap a single vertex in the embedding to allow this extension. That is, we do the following.

Let $\phi_{|T_0|}=\phi$. Then, for each $i$ with $|T_0|<i\leq n-m+1$ in turn, we will take $\phi_{i-1}$, an embedding of $T[\{t_1,\ldots,t_{i-1}\}]$ into $G[\{v_1,\ldots,v_{i-1}\}]$, and find $v_i\in V(G)\setminus \{v_1,\ldots,v_{i-1}\}$ and an embedding $\phi_i$ of $T[\{t_1,\ldots,t_i\}]$ into $G[\{v_1,\ldots,v_i\}]$, as follows. Let $w_i$ be the image under $\phi_{i-1}$ of the sole neighbour of $t_i$ in $T[\{t_1,\ldots,t_i\}]$ and,
\begin{enumerate}[label = \roman{enumi})]
\item if possible, choose $v_i\in N_G(w_i,V_0\setminus \{v_1,\ldots,v_{i-1}\})$ and set $\phi_i(t_i)=v_i$ and $\phi_i(t_j)=\phi_{i-1}(t_j)$ for each $j<i$,
\item otherwise, if possible, pick $v_i\in V(G)\setminus \{v_1,\ldots,v_{i-1}\}$ and $j\in J$ such that $\phi_{i-1}(t_j)=v_j$, $w_iv_j\in E(G)$ and $\phi_{i-1}(N_{T_0}(t_j))\subset N_G(v_i)$, and set $\phi_{i}(t_i)=\phi_{i-1}(t_j)$, $\phi_i(t_j)=v_i$ and $\phi_{i-1}(t_{j'})=\phi_{i-1}(t_{j'})$ for each $j'\in [i-1]\setminus \{j\}$ (noting that this does embed $T[\{t_1,\ldots,t_i\}]$ into $G[\{v_1,\ldots,v_i\}]$ as $N_{T_0}(t_j)=N_T(t_j)$), and,
%$v_iv\in E(G)$ and $N_{S_2}(v_i)\subset N_G(u)$ for some $u\in U$.
%$v_i\in N(w_i,V_0)\setminus \{v_1,\ldots,v_{i-1}\}$ and set $\phi(t_i)=v_i$, and,
\item otherwise, stop the embedding.
\end{enumerate}
Suppose the process first fails to embed $t_\ell$ for some $|T_0|<\ell\leq |T|$, otherwise we have embedded all of $T$ in $G$ and are done. We first show that step ii) has been carried out at most $3\mu n$ times. Indeed, suppose otherwise, and let $I$ be the set of the \emph{first} $\mu n$ values of $i$ with $|T_0|<i<\ell$ for which $v_i$ was embedded using step ii), so that $i<\ell-2\mu n$ for each $i\in I$.
Now, for each $i$ with $|T_0|<i<\ell$, $\phi_{i-1}$ and $\phi_{i}$ differ on at most one vertex in $\{t_1,\ldots,t_{i-1}\}$ which (if it exists) must be some vertex $t_j$ with $j\in J$ for which we still have $\phi_{i-1}(t_j)=v_j$. Therefore, as $w_i\in \{v_1\}\cup \phi_{i-1}(V(T)\setminus V(T_0))$, once a vertex has been embedded to $w_i$ by some $\phi_{i'}$, it is embedded to $w_i$ by any subsequent embedding. Thus, as $\Delta(T)\leq \Delta$, there must be at most $\Delta$ values of $i'>|T_0|$ with $w_{i'}=w_i$. Therefore, letting $W=\{w_i:i\in I\}$, we have
$|W|\geq |I|/\Delta$. Let $V_1=V_0\setminus \{v_1,\ldots,v_{\ell-2\mu n}\}$, and note that there are no edges in $G$ between $W$ and $V_1$, for otherwise there is an edge from $w_i$ to $V_1$ for some $i\in I$, which contradicts that $t_i$ was not embedded by a step i). Thus, as $|V_1|\geq 2\mu n-m\geq \mu n$ and
$|W|\geq \mu n/\Delta\geq m$ (using $m\ll1/\mu\ll n$), this contradicts that $G$ is $(m,\mu n)$-joined. Thus, we have that step ii) has been carried out at most $3\mu n$ times.

Let then $J_0=\{j\in J:\phi_{\ell-1}(v_j)=\phi(v_j)\}$, so that $|J\setminus J_0|\leq 3\mu n$. Now, we have $w_\ell\in \{v_1\}\cup \phi_{\ell-1}(V(T)\setminus V(T_0))$, and from the process above, we have that $\phi_{\ell-1}(t_1)=v_1$ and, for each $|T_0|<i<\ell$, $t_i$ is either embedded into $V_0$ at step i) or embedded to some vertex $\phi_{i-1}(t_j)=\phi(t_j)$, which is also in $V_0$. Thus, $w_{\ell}\in V_0$. Let $U_\ell\subset V(G)\setminus \{v_1,\ldots,v_{\ell-1}\}$ have size $m$, possible as $| V(G)\setminus \{v_1,\ldots,v_{\ell-1}\}|\geq n-(|T|-1)\geq m$. Then, noting $w_{\ell}\in V_0\setminus U_\ell$ and using the property in Claim~\ref{claim:new}, there is some $j\in J_0$ and $v_\ell\in U_\ell$ with $w_\ell v_j\in E(G)$ and $\phi(N_{T_0}(t_j))\subset N_G(v_\ell)$. As the vertices in $J$ have distance at least 3 apart in the tree, for each $j'$ with $t_{j'}\in N_{T_0}(t_j)$, we have $j'\notin J$, and therefore $\phi_{\ell-1}(N_{T_0}(t_j))=\phi(N_{T_0}(t_j))\subset N_G(v_\ell)$. Note that $v_\ell$ and $j$ show that ii) can be carried out to embed $t_\ell$, contradiction. Therefore, the process runs until all the vertices of $T$ have been embedding, showing that $G$ contains a copy of $T$.
\end{proof}

%% file: 4coververtices.tex
In this section, we prove the following key embedding result, which we use in Section~\ref{section:k=2} to embed each piece of the tree into the vortex while covering a prescribed set of vertices, represented in Lemma~\ref{lemma:iterativeembed} by the set $X$.

\begin{restatable}{lemma}{iterativeembed}\label{lemma:iterativeembed} Let $\Delta\ge 2$, $d\geq20$, $m\in\mathbb{N}$ and $0<\gamma<1/10$ satisfy $m\geq d^8$, $d\geq\Delta$ and $d\gg\Delta^3/\gamma$. 
%, and let $m$ be sufficiently large.
Let $G$ be an $m$-joined graph containing a vertex $v$. Let $T$ be a tree satisfying $|T|\geq2d^2m$ and $\Delta(T)\leq \Delta$, and let $t\in V(T)$.
Suppose $X\subset V(G)\setminus \{v\}$ contains at most $(1-\gamma)|T|$ vertices, and $I(X\cup \{v\})$ is $(d,m)$-extendable in $G$. Let $t'$ be a leaf of $T$ which is not $t$, and suppose $|G|\geq|T|+20dm$.

Then, $G$ contains a copy of $T$ that covers $X$, in which $t$ is copied to $v$, and $t'$ is copied into $V(G)\setminus X$.
\end{restatable}

In order to prove this, we first show the following result, which embeds a tree to cover most of a prescribed set of vertices.

%we will prove two key embedding results (Lemma~\ref{lemma:covering:0} and Lemma~\ref{lemma:covering:2}) that will allow us to embed a tree while covering a prescribed set of vertices. If we want to cover a set $X$, we will use Lemma~\ref{lemma:covering:0} to cover most of $X$ and then use Lemma~\ref{lemma:covering:2} in order to cover the leftover.
\begin{lemma}\label{lemma:covering:0}
Let $d,m,\Delta\in\mathbb N$ satisfy $d\ge 3$ and $d\geq \Delta$. Let $G$ be an $m$-joined graph which contains a set $X\subset V(G)$ and a vertex $v\in V(G)\setminus X$ such that $I(X\cup \{v\})$ is $(d,m)$-extendable in $G$.
Let $T$ be a tree with $\Delta(T)\leq \Delta$ and let $t\in V(T)$. Suppose that
\begin{equation}
|T|\geq |X|+\Delta m+2\;\;\text{ and }\;\;|G|\geq |T|+(2d+4)m+1.\label{eqn:TGsize}
\end{equation}
%Let $T$ be a tree such that  $|X|+\Delta m'\leq|T|\leq|G|-(2d+4)m-m'$  and $\Delta(T)\le \Delta$, and let $t\in V(T)$.
Then, there is a copy $S$ of $T$ such that $t$ is copied to $v$, $S\cup I(X)$ is $(d,m)$-extendable in $G$,  and $|X\setminus V(S)|<m$.
\end{lemma}
\begin{proof}
Let $\ell=|T|$. Let $t_1=t$, and label the vertices of $V(T)\setminus \{t\}$ as $t_2,\ldots,t_\ell$ so that, for each $i\in [\ell]$, $T_i=T[\{t_1,\ldots,t_i\}]$ is a tree. For each $2\leq i\leq \ell$, let $s_i$ be the unique neighbour of $t_i$ in $T[\{t_1,\ldots,t_i\}]$.
%$ and let $T_1\subset \dots\subset T_{\ell}=T$ be a sequence of subtrees of $T$ obtained as follows. $T_1$ consists of just the vertex $t_1=t$. For each $2\leq i\leq\ell $, $T_i$ is formed from $T_{i-1}$ by adding a single vertex $t_i$ and the unique edge $t_ir_{i-1}$ connecting $t_i$ to $T_{i-1}$ in $T$.
Let $v_1=v$ and let $S_1$ be the graph containing just the vertex $v_1$, so that $S_1$ is a copy of $T_1$ with $t_1$ copied to $v_1$, and $S_1\cup I(X)$ is $(d,m)$-extendable in $G$.  Now, carry out the following process, where for each $j=2,\ldots,\ell$ in turn, if possible we perform the following Step $\mathbf C_j$ to embed $t_j$ and produce a copy $S_j$ of $T_j$ in which $t$ is copied to $v$ and $S_j\cup I(X)$ is $(d,m)$-extendable in $G$.

%\phantom{ }\\\noindent   %Let $u_{i}\in V(S_{i-1})$ be the copy of the vertex $r_{i}\in V(T_{i-1})$, which recall is the unique neighbour of $t_i$ in $T_{i-1}$. Set $X_{i-1}=X\setminus S_{i-1}$.
\stepcounter{propcounter}
\begin{itemize}
    \item[$\mathbf{\Alph{propcounter}}_j$ $\bullet$]\label{covering:goodstep} Let $u_j$ be the copy of $s_j$ in $S_{j-1}$. If there exists a vertex $v_j\in X\setminus V(S_{j-1})$ so that $u_jv_j\in E(G)$, then let $S_j=S_{j-1}+u_jv_j$ and note that $S_j$ is a copy of $T_j$ in $G$. Moreover, $S_j\cup I(X)$ is $(d,m)$-extendable in $G$ by Lemma \ref{lemma:adding:edge}. In this case, we say Step $\mathbf C_j$ is a \textit{good} step.

    \item[\textcolor{white}{$\mathbf{\Alph{propcounter}}_j$} $\bullet$]\label{covering:badstep} Otherwise, if possible let $v_j$ be a neighbour of $u_{j}$ in $V(G)\setminus(V(S_{j-1})\cup X)$ such that, setting $S_j=S_{j-1}+u_jv_j$, $S_j\cup I(X)$ is $(d,m)$-extendable in $G$. Note that $S_j$ is a copy of $T_j$ in $G$. In this case, we say Step $\mathbf C_j$ is a \textit{neutral} step.
%    If $|X_{i-1}|\ge m$ but no such vertex exists in $X_{i-1}$, we instead use Lemma~\ref{lemma:adding:leaf} to find a neighbour $s_i$ of $u_{i-1}$ in $V(G)\setminus(V(S_{i-1})\cup X)$, such that if $S_i$ is obtained from $S_{i-1}$ by adding the edge $s_iu_{i-1}$, then $S_i\cup I(X)$ is $(d,m)$-extendable in $G$. Note that $S_i$ is a copy of $T_i$. In this case, we say that Step $\mathbf A_i$ is a \textit{bad} step.
%    \item\label{covering:laststep}If $|X_{i-1}|<m$, we find the copy $S_i$ exactly as in \ref{covering:badstep}, except that we call Step $\mathbf A_i$ is a \textit{neutral} step in this case.
\end{itemize}
We will show that we can successfully perform steps $\mathbf{C}_j$ for each $2\leq j\leq \ell$, and that the resulting tree $S_\ell$ will satisfy $|X\setminus V(S_\ell)|<m$. Then, as $S_\ell$ is a copy of $T_\ell=T$ with $t_1=t$ copied to $v_1=v$, this completes the proof.

\begin{claim}\label{clm:bad} For each $2\leq j\leq \ell$, if we have reached the start of step $\mathbf{C}_j$ and $|X\setminus V(S_{j-1})|\geq m$, then fewer than $\Delta m$ neutral steps have been taken so far.%taken has been at most $\Delta m$.%the number of bad steps we have made is less than $\Delta m'$, then we can perform Step $\mathbf A_i$.
\end{claim}
\begin{proof}[Proof of Claim~\ref{clm:bad}] Suppose we have reached the start of step $\mathbf{C}_j$ and $|X\setminus V(S_{j-1})|\geq m$. Let $J\subset [j-1]$ be the set of indices $i$ for which the step $\mathbf{C}_i$ was a neutral step, noting that, for each $i\in J$, $u_i$ has no neighbour in $G$ in $X\setminus V(S_{i-1})$, and hence no neighbour in $X\setminus V(S_{j-1})$. Thus, as $G$ is $m$-joined, we have $|\{u_i:i\in J\}|<m$.
As $\Delta(T)\leq \Delta$, $|\{u_i:i\in J\}|\geq |J|/\Delta$, so $|J|< \Delta m$, and thus before the start of step $\mathbf{C}_j$ we took fewer than $\Delta m$ neutral steps.
\renewcommand{\qedsymbol}{$\boxdot$}
\end{proof}
\renewcommand{\qedsymbol}{$\square$}

Suppose then, for contradiction, that the process reaches the start of Step $\mathbf{C}_j$ for some $2\leq j\leq \ell$ and fails to take either a good or neutral step.
Note that the number of good steps that have been taken is $|V(S_{j-1})\cap X|$ and the number of neutral steps that have been taken is $|V(S_{j-1})\setminus X|-1$, which is at most $|S_{j-1}|-|X|+m-2$ if $|X\setminus V(S_{j-1})|<m$.
 Therefore, by Claim~\ref{clm:bad}, the number of neutral steps is at most $\max\{\Delta m,|S_{j-1}|-|X|+m-2\}$, so that
 \begin{align*}
 |S_{j-1}\cup I(X)|&=(|X|+1)+(|V(S_{j-1})\setminus X|-1)\leq |X|+1+\max\{\Delta m,|S_{j-1}|-|X|+m-2\}\\
 &\leq \max\{|X|+\Delta m+1,|S_{j-1}|+m\}\overset{\eqref{eqn:TGsize}}{\leq} |T|+m.%|S_{j-1}|\leq \max\{|X|+\Delta m,|S_j|+m-1\}\leq |T|+m-1.
 \end{align*}
Therefore, by \eqref{eqn:TGsize}, $|G|\geq |T|+(2d+4)m+1\geq |S_{j-1}\cup I(X)|+(2d+2)m+m+1$, so that, by Lemma~\ref{lemma:adding:leaf}, there is some $v_j\in V(G)\setminus(V(S_{j-1})\cup X)$ which is a neighbour of $u_j$ in $G$ and such that $S_{j-1}\cup I(X)+u_jv_j$ is $(d,m)$-extendable in $G$, contradicting that we did not take a good or neutral step at Step $\mathbf{C}_j$.

Therefore, the process has successfully taken Step $\mathbf{C}_j$ for each $2\leq j\leq \ell$, and produced $S_\ell$, a copy of $T_\ell$ in $G$ with $t_1=t$ copied to $v_1=v$. Finally, note that at the start of Step $\mathbf{C}_\ell$ we will have taken at most $|X|$ good steps, and thus at least $(\ell-2)-|X|= |T|-|X|-2\geq \Delta m$ neutral steps by \eqref{eqn:TGsize}. Therefore, by Claim~\ref{clm:bad}, $|X\setminus V(S_\ell)|\leq |X\setminus V(S_{\ell-1})|<m$, and thus $S_\ell$ is the required copy of $T$.
\end{proof}

%Given a tree $T$, we say that two subtrees $T_1$ and $T_2$ divide $T$ if $V(T_1)\cup V(T_2)=V(T_2)$ and they intersect in exactly one vertex.
%\begin{lemma}[{\cite[Proposition 3.19]{montgomery2019}}]\label{lemma:divide:trees}Let $T$ be a tree with a subset $Q\subset V(T)$. Then we can find two subtrees $T_1$ and $T_2$ which divide $T$ such that $|Q\cap V(T_1)|,|Q\cap V(T_2)|\ge |Q|/3$.
%\end{lemma}

Given a graph $G$ and a subset $Q\subset V(G)$, we say that $Q$ is $k$-separated in $G$ if each pair of vertices in $Q$ is at distance at least $k$ in $G$.
We now need two results from~\cite{montgomery2019}, which we state in a slightly simpler form that follow directly from \cite[Corollary 3.16]{montgomery2019} and \cite[Lemma 4.1]{montgomery2019}, respectively.
\begin{proposition}\label{prop:tree:separate}Let $k\ge 0$ and $\Delta\ge 2$. Let $T$ be a tree with $\Delta(T)\le \Delta$ and $|T|\ge 3\Delta^k$. Then, there is a subset $Q\subset V(T)$ which is $(2k+2)$-separated in $T$ such that $|Q|\ge |T|/(8k+8)\Delta^k$.
\end{proposition}
\begin{lemma}\label{lemma:covering:1}
Let $k,d,m\in\mathbb N$ with $d\ge 20$. Let $G$ be an $m$-joined graph and let $R\subset G$ be a subgraph with $\Delta(R)\le d/4$. Suppose $X\subset V(G)\setminus V(R)$ is such that $R\cup I(X)$ is $(d,m)$-extendable in $G$.
Let $T$ be a tree with $\Delta(T)\le d/4$ which has a set of $3|X|$ vertices which is $(4k+4)$-separated in $T$.
Let $t\in V(T)$ and $r\in V(R)$, and suppose $|R|+|X|+|T|\le |G|-10dm-2k$. %Suppose that $Q\subset V(T)$ is a $(4k+4)$-separated set in $T$ which satisfies $|Q|\ge 3|X|$.

Then, there is a copy $S$ of $T$ in $G-V(R\setminus\{r\})$ so that $t$ is copied to $r$, $R\cup I(X)\cup S$ is $(d,m)$-extendable in $G$ and $|X\setminus V(S)|\le 2m/(d-1)^k$.%, and all vertices in $X\cap V(S)$ have a vertex in $Q$ copied to them.
\end{lemma}

We can now prove Lemma~\ref{lemma:iterativeembed}. The proof takes a tree decomposition using Corollary~\ref{corollary:ngammadecomposition}, embed the first piece to cover most of $X$ using Lemma~\ref{lemma:covering:1}, and then repeatedly uses Proposition~\ref{prop:tree:separate} and Lemma~\ref{lemma:covering:1} to embed the remaining pieces while covering more and more of the uncovered vertices in $X$ until all the vertices in $X$ are covered.

%\iterativeembed*
\begin{proof}[Proof of Lemma~\ref{lemma:iterativeembed}] %Let $n=|G|\geq $
Let $\gamma_2=\gamma/10$, $\gamma_1=\gamma_2/16\Delta$ and $T'=T-t'$. Let $\ell$ be the smallest integer such that $2m/(d-1)^{\ell-1}<1$, so that $\ell\geq 9$ as $m\geq d^8$. By minimality, we also have that $2m\geq(d-1)^{\ell-2}$, and thus, using $d\gg\Delta^3/\gamma$, we have
\begin{equation}\label{eqn:gamma1}
|T'|\geq 2d^2m-1\geq d^2(d-1)^{\ell-2}-1>(16\Delta/\gamma_2)^{\ell+1}.%\gamma_1^2|T|/(2m)\geq \Delta/\gamma_2.
\end{equation}
Then, by Corollary~\ref{corollary:ngammadecomposition}, there is a $(\gamma_1,\gamma_2)$-descending tuple $\mathbf{n}=(n_1,\cdots,n_\ell)\in\mathbb{N}^{\ell}$ such that $T'$ has an $\mathbf{n}$-decomposition $(T_1,\ldots,T_\ell)$ with $t\in V(T_1)$.
Let $t_0=t$ and, for each $i\in[\ell-1]$, let $t_{i}$ be the unique vertex shared by $T_i$ and $T_{i+1}$.

Note that $\sum_{i=2}^\ell n_i\leq 2\gamma_2n_1\leq 2\gamma_2|T|$, so
\[
|T_1|\geq(1-2\gamma_2)|T|-1\geq(1-\gamma/2)|T|\geq|X|+\gamma|T|/2\geq |X|+\Delta m+3.
\]
As $|G|\geq|T|+20dm$, we may apply Lemma~\ref{lemma:covering:0} to find a copy $S_1$ of $T_1$ in $G$ in which $t$ is copied to $v$, $S_1\cup I(X)$ is $(d,m)$-extendable in $G$ and $|X\setminus V(S_1)|<m$. Next, for each $2\leq j\leq \ell$ in turn, do the following.
\stepcounter{propcounter}
\begin{enumerate}[label = \textbf{\Alph{propcounter}$_j$}]
\item If possible, find a copy $S_j$ of $T_j$ in $G-(\cup_{i\in [j-1]}V(S_i)\setminus \{v_{j-1}\})$ such that $t_{j-1}$ is copied to $v_{j-1}$, $(\cup_{i\in [j]}S_i)\cup I(X)$ is $(d,m)$-extendable in $G$, and $|X\setminus (\cup_{i\in [j]}S_i)|<2m/(d-1)^{j-1}$.\label{stepDi}
\end{enumerate}
Suppose that for some $2\leq j\leq \ell$ this was not possible. Note that the set $X_j=X\setminus  (\cup_{i\in [j-1]}S_i)$ satisfies $|X_j|<2m/(d-1)^{j-2}$, and $|T_j|=n_j+1\geq\gamma_1^{j-1}n_1\geq\gamma_1^{j-1}|T|/2\geq\gamma_1^{j-1}d^2m$. Hence, we may apply Proposition~\ref{prop:tree:separate} to obtain a set $Q_j$ of vertices in $T_j$ that is $(4j)$-separated, with 
$$|Q_j|\geq\frac{|T_j|}{16j\Delta^{2j-1}}\geq\frac{\gamma_1^{j-1}d^2m}{16j\Delta^{2j-1}}\ge \frac{6m}{(d-1)^{j-2}}\cdot\frac{\gamma_1^{j-1}(d-1)^j}{96j\Delta^{2j-1}}
\geq\frac{6m}{(d-1)^{j-2}}\geq 3|X_j|,$$
where the second last inequality follows from $d\gg\Delta^3/\gamma$. Furthermore, we have
$$|\cup_{i\in[j-1]}S_i|+|X_j|+|T_j|\leq |T|+|X_j|\leq |T|+2m\leq |G|-10dm-2(j-1),$$
% such that
%\[
%|Q_j|\geq \frac{|T_j|}{(8k+8)\Delta^j}\geq \frac{\gamma_2^{j}|T|}{(8k+8)\Delta^j}\geq 3|X_j|.
%\]
so we may apply Lemma~\ref{lemma:covering:1} with $R=\cup_{i\in [j-1]}S_i$, $X=X_j$, $T=T_j$, $r=v_{j-1}$ and $t=t_{j-1}$ to find a copy $S_j$ of $T_{j}$, satisfying all the required conditions in~\ref{stepDi}, which is a contradiction.

Therefore, we can complete \ref{stepDi} for each $2\leq j\leq \ell$. Taking $S'=\cup_{j\in [\ell]}S_j$, we have, then, a copy of $T'$ in which $t$ is copied to $v$ such that $I(X)\cup S'$ is $(d,m)$-extendable and $|X\setminus V(S')|<2m/(d-1)^{\ell-1}<1$. In other words, the copy $S'$ of $T'$ covers $X$. Let $s'$ be the vertex of $S'$ which needs a leaf added to make $S'$ into a copy of $T$.
As $|G|\geq |S'|+(2d+2)m+m+1$, by Lemma~\ref{lemma:adding:leaf} there is some $v'\in V(G)\setminus V(S')\subset V(G)\setminus X$ which is a neighbour of $s'$. Let $S=S'+s'v'$, and note that this is a copy of $T$  which covers $X$, in which $t$ is copied to $v$ and $t'$ is copied into $V(G)\setminus X$, as required.
\end{proof}

%% file: 5maink2proof.tex
In this section, we prove Theorem~\ref{theorem:k=2}, where we embed a tree $T$ with $n-m+1$ vertices and $\Delta(T)\leq\Delta$ into an $n$-vertex $(m,\mu n)$-joined $(D,m)$-expander graph $G$, where $m\leq \mu n$ and $\mu,1/D,1/m\ll 1/\Delta$. We wish to do the following. First, we partition $T$ into trees $T_1\cup\ldots \cup T_\ell$ with geometrically-decreasing size. We then use the sizes of the trees in this partition to inform the sizes of a partition of all but at most $m/4$ vertices of $G$ as $V_0\cup V_1\cup \ldots \cup V_\ell$. Then, for each $i\in [\ell]$ in turn, we embed $T_i$ into $G[V_{i-1}\cup V_i]$, attached appropriately to the existing embedding, so that, if $i<\ell$, all the unused vertices in $V_{i-1}$ are covered, and the sole vertex shared by $T_i$ and $T_{i+1}$ is embedded into $V_i$. This is depicted in Figure~\ref{fig:embeddingpic}.

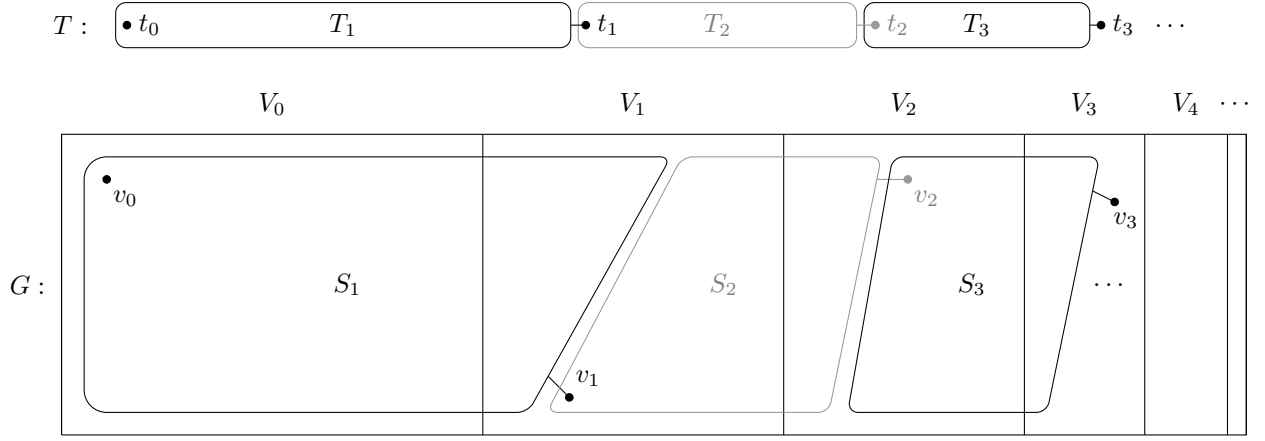
\begin{figure}
\input{embeddingpicture}
\caption{To embed $T$ for Theorem~\ref{theorem:k=2}, we divide $T$ into $T_1\cup\ldots\cup T_\ell$ using Lemma~\ref{lemma:ngammadecomposition}, then find an accompanying vortex partition $V_0\cup V_1\cup \ldots \cup V_\ell$ (of a large subgraph $G'$ of $G$) using Lemma~\ref{lemma:vortex:partition}. We then use Lemma~\ref{lemma:iterativeembed} to iteratively embed $T_i$ into $G[V_{i-1}\cup V_i]$, where, if $i<\ell$, any unused vertices in $V_{i-1}$ are covered and the common vertex $t_i$ of $T_i$ and $T_{i+1}$ is embedded to $v_i$ in $V_{i}$.}
\label{fig:embeddingpic}
\end{figure}

To embed each $T_i$, except for the last piece, $T_\ell$, where we have very few spare vertices, we use Lemma~\ref{lemma:iterativeembed}. In order to do this in the manner described for the partition $V_0\cup V_1\cup \ldots \cup  V_\ell$, we need several properties, which we record, as follows, as a vortex-partition. Essentially, we require the sets to have (approximately) the right size (see \ref{prop:vorpart1}), and that, for each $i\in [\ell]$, the subgraph we are using to embed $T_i$ (which will be a subgraph of $G[V_{i-1}\cup V_i]$) is joined (see \ref{prop:vorpart2}) and the set of any remaining vertices in $V_{i-1}$ to be covered is, considered as a empty graph, extendable (which will follow from \ref{prop:vorpart3}).

\begin{definition}[Vortex partition]\label{defn:vortexpart} Let $d\in \mathbb{N}$ and $\lambda>0$, and let $\mathbf{n}\in\mathbb N^{\ell+1}$. For a graph $G$ an $(\mathbf{n},\lambda,d)$-vortex-partition in $G$ is a partition $V_0\cup V_1\cup \ldots\cup V_\ell$ of $V(G)$ such that the following hold.
\stepcounter{propcounter}
\begin{enumerate}[label = \textbf{\Alph{propcounter}\arabic{enumi}}]
\item For each $i\in [\ell]_0$, $|V_i|=(1\pm \lambda)n_i$.\label{prop:vorpart1}
\item For each $i\in [\ell]$, $G[V_{i-1}\cup V_i]$ is $(\lambda n_{i-1})$-joined.\label{prop:vorpart2}
\item For each $i\in [\ell]$, $I(V_{i-1})$ is $(d,\lambda n_{i-1})$-extendable in $G[V_{i-1}\cup V_i]$.\label{prop:vorpart3}
%\item every subset $S\subset U_{i-1}$, with $|S|=m$, satisfies $|N(S,U_i)|\ge (1-\lambda)|U_i|$, and
%\item every subset $S\subset U_{i-1}$, with $|S|=\frac{m}{100}$, satisfies $|N(S,U_i)|\ge (\alpha-\lambda)|U_i|$.
\end{enumerate}
\end{definition}
We need the vortex partition that we find to cover all but very few (at most $m/4$) vertices in the graph $G$. To do this, we first find an initial `vortex' in Section~\ref{sec:vortex}, before `cleaning' this into a vortex partition in Section~\ref{sec:expandvortex}. We can then carry out our embedding in Section~\ref{sec:mainthmk=2}, completing the proof of Theorem~\ref{theorem:k=2}.

\subsection{Vortices}\label{sec:vortex}

To find our vortex partition, we first find a vortex in our graph $G$, which is a \emph{nested} sequence of vertex sets  with an increasingly good expansion property as the sets decrease in size. We define this precisely as follows.

\begin{definition}[Vortex] Let $\ell$ be a non-negative integer, $\lambda>0$, $m\in\mathbb N$, and let $\mathbf{n}=(n_0,\dots,n_{\ell})\in\mathbb N^{\ell+1}$. For a graph $G$ on $n_0$ vertices, an $(\mathbf{n},m,\lambda)$-vortex in $G$ is a sequence of subsets $U_0\supset U_1\supset\ldots\supset U_\ell$ such that $U_0=V(G)$ and, for each $i\in [\ell]$,
\stepcounter{propcounter}
\begin{enumerate}[label = \textbf{\Alph{propcounter}\arabic{enumi}}]
\item $|U_i|=n_i$, and\label{prop:vor1}
\item every subset $U\subset U_{i-1}$ with $|U|=m$ satisfies $|N(U,U_i)|\ge (1-\lambda)|U_i|$.\label{prop:vor2}
%\item every subset $U\subset U_{i-1}$, with $|U|=\frac{m}{100}$, satisfies $|N(U,U_i)|\ge (\alpha-\lambda)|U_i|$.
\end{enumerate}
\end{definition}

For an appropriate descending tuple $\mathbf{n}=(n_0,\ldots,n_\ell)\in \mathbb{N}^{\ell+1}$ and an $n$-vertex graph $G$ which is $(m,\mu n)$-joined we will find an $(\mathbf{n},m,2\lambda)$-vortex $U_0\supset U_1\supset \ldots \supset U_{\ell}$ in $G$ (with $\mu\ll \lambda$) by randomly choosing such a nested subsequence subject only to $|U_i|=n_i$ for all $i\in [\ell]_0$, and showing that it is an $(\mathbf{n},m,2\lambda)$-vortex with high probability (see Lemma~\ref{lemma:vortex}). In this we are inspired by a similar vortex used by Barber, K\"uhn, Lo, and Osthus~\cite{barber2016edge} in a graph with high minimum degree, and we use some similar calculations in the analysis. %Essentially, we choose the sets in order of decreasing size and track the non-neighbours of the $m$-sets

More challengingly though, we want to show that we can ensure that $G[U_{\ell-1}]$ is also $(\lambda m)$-joined. To this end, we additionally choose the vortex with two vertex sets $V_0\supset V_1$ such that, with high probability, $V_0$ and $V_1$ fit into the randomly chosen nested subsequence that will be (with high probability) our vortex so that $V_1$ contains $U_{\ell-1}$.
Then, we show that with high probability, every $m$-set in $V_0$ has at most $\lambda m$ non-neighbours in $V_0$, and, conditioned on this, that every $(\lambda m)$-set in $V_1$ has at most $\lambda m$ non-neighbours in $V_0$. For this, we need the random sets to jump down in size by more than the standard ratio between the sets in the vortex, which is why we take the additional sets $V_0$ and $V_1$. We also use sets $W$ and $W'$ chosen disjointly which will fit within $U_1\setminus U_2$ with high probability, for a similar ease of analysis for the expansion of small sets into $U_1\setminus U_2$. Using all this, we prove the following lemma finding the vortex that we will need.

\begin{lemma}\label{lemma:vortex}
Let
\[
\frac{1}{m}\ll  \mu \ll\frac{1}{K},\frac{1}{D}\ll\lambda,\gamma_1,\gamma_2\leq \frac{1}{9},
\]
with $\gamma_1<\gamma_2$ and let $n$ be such that $m\leq \mu n$.
Let $G$ be an $n$-vertex $(m,\mu n)$-joined graph which is a $(D,m)$-expander, and let $\mathbf{n}=(n_0,n_1,\dots, n_\ell)$ be a $(\gamma_1,\gamma_2)$-descending tuple with $n_0=n$ and $\gamma_1Km\leq n_\ell\leq 2Km$.

Then, $G$ contains an $(\mathbf{n},m,2\lambda)$-vortex $U_0\supset U_1\supset \ldots \supset U_{\ell}$ such that $G[U_{\ell-1}]$ is $\lambda m$-joined and, for every $U\subset V(G)$ with $|U|=\lfloor \frac{m}{4}\rfloor$, $|N_G(U,U_1\setminus U_2)|\ge \gamma_1 Dm/200$.
\end{lemma}
\begin{proof}
Let $K_0$ be such that $\mu\ll \frac{1}{K_0}\ll\frac1K$ and let $k\in [\ell]$ be some index with $\gamma_1K_0m\leq n_k\leq K_0m$. Let $p_0=\frac{2n_k}{n_0}$ and $q=\frac{n_1-n_2}{4n_0}$, and take random disjoint sets $V_0$, $W$ and $W'$ in $V(G)$ so that the location of each vertex is selected uniformly at random and is in $V_0$ with probability $p_0$, in $W$ with probability $q$ and in $W'$ with probability $q$.
%We colour $V(G)$ with colours $\{0,1,2,3\}$ so that each vertex $v\in V(G)$, independently from each other, has colour $0$ with probability , colour 1 with probability , colour $2$ also with probability $q$, or colour 3 with probability $1-p_0-2q$. Let $V_0$ be the vertices in colour $0$, $W$ be the vertices in colour $1$ and $W'$ be the vertices in colour $2$. 
Let $V_1\subset V_0$ be formed by including each element of $V_0$ independently at random with probability $p_1:=\frac{2n_{\ell-1}}{n_k}$. Let $F$ be the event that $n_k\leq |V_0|\leq 3n_k$, $n_{\ell-1}\leq |V_1|\leq 9n_{\ell-1}$, and $\frac{n_1-n_2}{6}\le |W|,|W'|\le \frac{n_1-n_2}{3}$. 
If $F$ holds, then let $U_0\supset U_1\supset \ldots \supset U_{\ell}$ be a sequence of sets chosen uniformly at random subject to $U_{k-1}\supset V_0\supset U_{k}$, $U_{\ell-2}\supset V_1\supset U_{\ell-1}$, $U_1\setminus U_2\supset W\cup W'$, and $|U_i|=n_i$ for each $i\in [\ell]_0$ (where we have used that $|W\cup W'|\le \frac{2(n_1-n_2)}{3}$,  $|V_1|\leq 9n_{\ell-1}\leq n_{\ell-1}/\gamma_2\leq n_{\ell-2}$, and, similarly, $|V_0|\leq n_{k-1}$). If $F$ does not hold, then let $U_0\supset U_1\supset \ldots \supset U_{\ell}$ be a sequence of sets chosen uniformly at random subject to $U_i=n_i$ for each $i\in [\ell]_0$.

\begin{claim}\label{claim:three}
With probability more than $3/4$, $F$ holds.
\end{claim}
\begin{claim}\label{claim:one}
With probability more than $3/4$, $U_0\supset U_1\supset \ldots \supset U_{\ell}$ is an $(\mathbf{n},m,2\lambda)$-vortex in $G$.
\end{claim}
\begin{claim}\label{claim:oneandahalf}
With probability more than $3/4$, for every set $U\subset V(G)$ with $|U|=\lfloor \frac{m}{4}\rfloor$, we have $|N_G(U,W\cup W')|\geq \gamma_1 Dm/200$.
\end{claim}
\begin{claim}\label{claim:two}
With probability more than $3/4$, $G[V_1]$ is $(\lambda m)$-joined.
\end{claim}
These claims easily imply the lemma. Indeed, by them, we have that, with positive probability, $F$ holds, $U_0\supset U_1\supset \ldots \supset U_{\ell}$ is an $(\mathbf{n},m,2\lambda)$-vortex in $G$, every subset $U\subset V$ with $|U|=\lfloor \frac{m}{4}\rfloor$ satisfies $|N(U, W\cup W')|\ge \gamma_1Dm/200$, and $G[V_1]$ is $(\lambda m)$-joined. Moreover, as $F$ holds, we have $U_{\ell-1}\subset V_1$ and thus $G[U_{\ell-1}]$ is $(\lambda m)$-joined, and, as $W\cup W'\subset U_1\setminus U_2$, we have that every subset $U\subset V(G)$ with $|U|=\lfloor \frac{m}{4}\rfloor$ satisfies $|N(U, U_1\setminus U_2)|\ge \gamma_1Dm/200$, as required. Thus, it is left only to prove the four claims.

\begin{proof}[Proof of Claim~\ref{claim:three}] Now, by Lemma~\ref{lemma:chernoff}, using $p_0=2n_k/n_0$ and $n_k\geq \gamma_1K_0m$, we have
\[
\P\left(\big||V_0|-2n_k\big|>n_k\right)\leq 2\exp(-2n_k/12)=2\exp(-n_k/6)< 1/16.
\]
Conditioning on $n_k\leq |V_0|\leq 3n_k$ and again by Lemma~\ref{lemma:chernoff}, using $p_1=2n_\ell/n_k$ and $n_\ell\geq \gamma_1Km$, we have
\[
\P(|V_1|\leq n_\ell\text{ or }|V_1|\geq9n_\ell)\leq\P(||V_1|-p_1|V_0||\geq p_1|V_0|/2)\leq 2\exp(-n_\ell/6)< 1/16.
\]
Similarly, by Lemma~\ref{lemma:chernoff}, using $q=\frac{n_1-n_2}{4n_0}$ and $n_2\leq\gamma_2n_1$, we have for each $i\in[2]$,
\begin{equation}\label{equation:size:W1}\mathbb P\left(\big||W_i|-\tfrac{n_1-n_2}{4}\big|>\tfrac{n_1-n_2}{12}\right)\le 2\exp(-(n_1-n_2)/108)<\frac{1}{16}.
\end{equation}
Thus, with probability more than $1-4/16=3/4$, we have $n_k\leq |V_0|\leq 3n_k$,  $n_\ell\leq |V_1|\leq 9n_\ell$, and $\frac{n_1-n_2}{6}\le|W|,|W'|\le \frac{n_1-n_2}{3}$, and hence $F$ holds.
\renewcommand{\qedsymbol}{$\boxdot$}
\qedhere
\renewcommand{\qedsymbol}{$\square$}
\end{proof}

\begin{proof}[Proof of Claim~\ref{claim:one}] Note that, as $V_0,V_1,W,W'$ are themselves random sets, the sets $U_0\supset U_1\supset \ldots \supset U_{\ell}$ have the same distribution as a nested sequence of sets chosen uniformly at random subject to $|U_i|=n_i$ for each $i\in [\ell]_0$, whether $F$ holds or not. Now, for each $i\in [\ell]_0$, let $\eps_i=\left({m}/{n_i}\right)^{1/4}$ and $\mathbf{n}_i=(n_0,n_1,\dots,n_i)$.
For each $i\in [\ell]$, let $E_i$ be the event that $U_0\supset U_1\supset \ldots \supset U_{i-1}$ is an $(\mathbf{n}_i,m,\lambda+\eps_{i-1})$-vortex, but $U_0\supset U_1\supset \ldots \supset U_{i}$ is not an $(\mathbf{n}_i,m,\lambda+\eps_i)$-vortex.
As $\eps_0\geq 0$ and $\lambda\geq\mu$, $U_0=V(G)$ is an $(\mathbf{n}_0,m,\lambda+\eps_0)$-vortex. If no event $E_i$, $i\in [\ell]$, holds, then $U_0\supset U_1\supset \ldots \supset U_{\ell}$ is an $(\mathbf{n}_0,m,\lambda+\eps_\ell)$-vortex, and hence an $(\mathbf{n}_0,m,2\lambda)$-vortex as $\eps_\ell=(m/n_\ell)^{1/4}\leq (\gamma_1 K)^{-1/4}\leq\lambda$.
Therefore, to prove the claim it is sufficient to show that
\[
\sum_{i\in [\ell]}\P(E_i)<\frac{1}{4}.
\]
Let $i\in [\ell]$, and suppose that $U_0\supset U_1\supset \ldots \supset U_{i-1}$ is an $(\mathbf{n}_{i-1},m,\lambda+\eps_{i-1})$-vortex. For each $U\subset U_{i-1}$ with $|U|=m$, as $|N(U,U_{i-1})|\geq (1-\lambda-\eps_{i-1})n_{i-1}$, we have $\E|N(U,U_i)|\geq (1-\lambda-\eps_{i-1})n_{i}\ge n_i/2$ and, by Lemma~\ref{lemma:chernoff} with $\varepsilon=(m/n_i)^{1/3}$, that
\begin{equation}\label{eqn:first}
\P\left(|N(U,U_i)|< \Big(1-\lambda-\eps_{i-1}-\left(\frac{m}{n_i}\right)^{1/3}\Big)n_i\right)\leq 2\exp\left(-\left(\frac{m}{n_i}\right)^{2/3}\cdot \frac{n_i}{6}\right)=2\exp\left(-\frac{m}{6}\cdot \left(\frac{n_i}{m}\right)^{1/3}\right).
\end{equation}
Now, as $n_{i}\leq \gamma_2n_{i-1}$, we have
\begin{equation}\label{eqn:second}
\eps_i-\eps_{i-1}=\left(\frac{m}{n_i}\right)^{1/4}-\left(\frac{m}{n_{i-1}}\right)^{1/4}
=\left(\frac{m}{n_i}\right)^{1/4}\left(1-\left(\frac{n_i}{n_{i-1}}\right)^{1/4}\right)
\geq \left(\frac{m}{n_i}\right)^{1/4}(1-\gamma_2^{1/4})\geq \left(\frac{m}{n_i}\right)^{1/3},
\end{equation}
as $n_i\geq n_\ell\geq  \gamma_1 Km$ and $1/K\ll \gamma_1,\gamma_2$, so that $(1-\gamma_2^{1/4})\geq(1/\gamma_1K)^{1/12}\geq (m/n_i)^{1/12}$.
Furthermore, in preparation for taking a union bound over all sets $U\subset U_{i-1}$ with size $m$, as $\gamma_1n_{i-1}\leq n_{i}$, we have
\begin{align}
\binom{n_{i-1}}{m}\cdot 2\exp\left(-\frac{m}{6}\cdot \left(\frac{n_i}{m}\right)^{1/3}\right)&\leq \left(\frac{en_{i-1}}{m}\right)^m\cdot 2\exp\left(-\frac{m}{6}\cdot \left(\frac{n_i}{m}\right)^{1/3}\right)\nonumber\\
&\leq 2\exp\left(\frac{m}{6}\cdot \left(6\log\left(\frac{e n_i}{\gamma_1 m}\right)-\left(\frac{n_i}{m}\right)^{1/3}\right)\right)\nonumber\\
&\leq 2\exp\left(-\frac{m}{12}\cdot \left(\frac{n_i}{m}\right)^{1/3}\right)\leq \frac{m}{n_i},\label{eqn:third}
\end{align}
where the last line of inequalities holds as $1/K\ll \gamma_1,\gamma_2$ and $n_i/m\geq n_\ell/m\geq \gamma_1K$.
Therefore, by \eqref{eqn:first}, \eqref{eqn:second} and \eqref{eqn:third}, and by taking a union bound, we have that, with probability more than $1-(m/n_i)$, for each $U\subset U_{i-1}$ with $|U|=m$ we have $|N(U,U_i)|\geq (1-\lambda-\eps_i)n_i$.

%For each $U\subset U_{i-1}$ with $|U|=m/100$, by a similar calculation to \eqref{eqn:first} with $\alpha$ in place of $1-\mu$, and by \eqref{eqn:second}, we have $\P(|N(U,U_i)|\geq (2\alpha-\eps_i)n_i)\leq 2\exp(-m(n_i/m)^{1/3})$. And therefore, by a similar calculation to \eqref{eqn:third} and a union bound, we have that, with probability more than $1-m/n_i$, for each $U\subset U_i$ with $|U|=m$ we have $|N(U,U_i)|\geq (1-\mu-\eps_i)n_i$.
%Therefore, together, we have $\P(E_i)\leq 2m/n_i$.
Hence, we have
\[
\sum_{i\in [\ell]}\P(E_i)\leq \sum_{i\in [\ell]}\frac{m}{n_i}\leq \sum_{i\in [\ell]}\frac{m}{n_\ell}\cdot \gamma_2^{\ell-i}\leq \frac{1}{1-\gamma_2}\cdot \frac{m}{n_\ell}\leq \frac{2m}{n_\ell}\leq \frac{2}{\gamma_1K}< \frac{1}{4},
\]
as required.
%Thus, together, with probability at least $1-2$, $U_0\supset U_1\supset \ldots \supset U_{i-1}\supset U_i$ is an $(\mathbf{n}_i,m,\alpha-\eps_i,\lambda-\eps_i)$-vortex. Therefore, $G$ contains an $(\mathbf{n}_i,m,\alpha-\eps_i,\lambda-\eps_i)$-vortex, as required. Thus,
\renewcommand{\qedsymbol}{$\boxdot$}
\end{proof}
\renewcommand{\qedsymbol}{$\square$}
\begin{proof}[Proof of Claim~\ref{claim:oneandahalf}]
%Then, for each $U\subset V(G)$ with $|U|=m$, we have $\mathbb{E}|N^c_G(U)\cap W|\le 2\mu qn$, and thus, by Lemma~\ref{hypergeom} with $t=\mu qn$ and a union bound, we have that $|N^c_G(U)\cap W|\ge 3\mu qn$ for some $U\subset V(G)$ with $|U|=m$ with probability  at most
%\[
%\binom{n}{m}\cdot 2\exp\left(\frac{-\mu qn}{4}\right)\le 2\exp\left(-\frac{\mu(n_1-n_2)}{20}\right),
%\]
%as $m\le\mu n$ and $\mu\ll \gamma_1,\gamma_2\le 1/8$. Let $E$ be the event that $\frac{n_1-n_2}{6}\le |W|\le \tfrac{n_1-n_2}{3}$ and every subset $U\subset V(G)$ with $|U|=m$ satisfies $|N^c_G(U)\cap W|\le 3\mu qn$. If $E$ holds, then for every set $U\subset V(G)$ with $|U|=m$ we have 
%\[|N(U,W)|\ge \frac{n_1-n_2}{6}-4\mu n\ge \frac{\gamma_1 n}{8},\]
Let $E$ be the event that $|W|\ge \frac{n_1-n_2}{6}$. By~\eqref{equation:size:W1}, $\mathbb P(E)\geq 7/8$. Note that, if $E$ holds, then $|W|\ge {\gamma_1 n}/{8}\ge\mu n+(2\gamma_1D+2)m$,
as $n_1-n_2\ge (1-\gamma_2)\gamma_1 n\ge 7\gamma_1n/8$ and $\mu\ll1/D\ll \gamma_1$. Thus, for every choice of $W$ for which $E$ holds, we can apply Proposition~\ref{prop:expander:subgraphnewnew} to find $B_W\subset V(G)$ such that $|B_W|<m$ and, for each $U\subset V(G)\setminus B_W$ with $|U|\le m$, we have $|N_G(U,W)|\ge \gamma_1D|U|$. 

Let $E'$ be the event that $E$ holds and $|N_G(U,W')|\ge \gamma_1 Dm/100$ for every $U\subset B_W$ with $|U|=\lfloor m/8\rfloor$ and $|N_G(U,W)|< \gamma_1 Dm/100$. Note that, conditioned on any choice of $W$ for which $E$ holds, as $G$ is a $(D,m)$-expander, for every set $U\subset B_W$ with $|U|=\lfloor m/8\rfloor$ and $|N_G(U,W)|< \gamma_1 Dm/100$, we have $|N_G(U,V(G)\setminus W)|\geq D|U|-\gamma_1Dm/100\geq D|U|/2$. Thus, by Lemma~\ref{hypergeom} with $\varepsilon=1/1000$ and a union bound, 
\begin{equation*}%\label{equation:prob:Z}
\mathbb{P}(E'\text{ does not hold}|E\text{ holds})\leq \binom{m}{\lfloor{m/8}\rfloor}\cdot 2\exp\left(\frac{-\gamma_1Dm}{6\cdot 10^8}\right)<\frac{1}{8},
\end{equation*}
as $1/m,1/D\ll \gamma_1$. Therefore, $\P(E' \text{ holds})\geq \P(E)-\P(E'\text{ does not hold}|E\text{ holds})>3/4$.

We claim that if $E'$ holds, then, for every subset $U\subset V(G)$ with $|U|=\lfloor {m}/{4}\rfloor$, $|N_G(U,W\cup W')|\ge \gamma_1Dm/200$. Indeed, for a set $U\subset V(G)$ with $|U|=\lfloor m/4\rfloor$, if $|U\cap B_W|\ge \lfloor m/8\rfloor$, then either $|N_G(U,W)|\ge \gamma_1 Dm/100-|U|\ge {\gamma_1 Dm}/{200}$, or, as $E'$ holds, $|N_G(U,W')|\ge \gamma_1 Dm/100-|U|\ge {\gamma_1 Dm}/{200}$. On the other hand, if $|U\cap B_W|<\lfloor m/8\rfloor$, then $|U\setminus B_W|\geq \lfloor m/8\rfloor$, and so by the choice of $B_W$, $|N_G(U)\cap W|\ge\gamma_1Dm/8-|U|\ge {\gamma_1 Dm}/{200}$, as required. Since $E'$ holds with probability in excess of $3/4$, the claim follows.
\renewcommand{\qedsymbol}{$\boxdot$}
\end{proof}
\renewcommand{\qedsymbol}{$\square$}
\begin{proof}[Proof of Claim~\ref{claim:two}]
We first show that $G[V_0]$ is not $(m,\lambda m)$-joined with probability less than $1/12$. Note that, for each $U\subset V(G)$ with $|U|=m$ we have $|V(G)\setminus (U\cup N(U))|\leq \mu n$, and therefore
\begin{equation}\label{eqn:forV1}
\P(U\subset V_0\text{ and }|V_0\setminus(U\cup N(U))|> \lambda m)\leq p_0^m\cdot \binom{\mu n}{\lambda m}p_0^{\lambda m}.
\end{equation}
Thus, the probability that $G[V_0]$ is not $(m,\lambda m)$-joined is, as $p_0=2n_k/n_0=2n_k/n$, at most
\begin{align}
\binom{n}{m}p_0^m\cdot \binom{\mu n}{\lambda m}p_0^{\lambda m}
&\leq \left(\frac{ep_0n}{m}\right)^m\cdot \left(\frac{ep_0\mu n}{\lambda m}\right)^{\lambda m}= \left(\frac{2en_k}{m}\right)^m\cdot \left(\frac{2e\mu n_k}{\lambda m}\right)^{\lambda m}
\leq \left(2eK_0\right)^m\cdot \left(\frac{2e\mu K_0}{\lambda}\right)^{\lambda m}\nonumber\\
&= \left(2eK_0\right)^{(1+\lambda)m}\cdot \left(\frac{\mu}{\lambda}\right)^{\lambda m}\leq K_0^{3m}\cdot \left(\frac{\mu}{\lambda}\right)^{\lambda m}
\leq \frac{1}{12},\label{eqn:forV2}
\end{align}
where the last inequality follows as $\mu\ll1/K_0\ll\lambda$.

From the proof of Claim~\ref{claim:three}, we have $\P(|V_0|> 4n_k)< 1/12$. Furthermore, if $G[V_0]$ is $(m,\lambda m)$-joined and $|V_0|\leq 4n_k$, then, by similar calculations to those in \eqref{eqn:forV1} and \eqref{eqn:forV2} and as $p_1=2n_{\ell-1}/n_k$, $n_{\ell-1}\leq 2Km/\gamma_1$ and $n_k\geq \gamma_1K_0m$, we have that $G[V_1]$ is not $(\lambda m)$-joined with probability at most
\begin{align*}
\binom{4n_k}{\lambda m}p_1^{\lambda m}\cdot \binom{m}{\lambda m}p_1^{\lambda m}
&\leq \left(\frac{4ep_1n_k}{\lambda m}\right)^{\lambda m}\cdot \left(\frac{ep_1 m}{\lambda m}\right)^{\lambda m}
= \left(\frac{16e^2n_{\ell-1}^2}{\lambda^2mn_k}\right)^{\lambda m}
\leq  \left(\frac{64e^2K^2}{\lambda^2\gamma_1^3K_0}\right)^{\lambda m}< \frac{1}{12},
\end{align*}
where the last inequality holds as $1/K_0\ll1/K\ll \lambda,\gamma_1$ and $1/m\ll \lambda$.

Altogether then, with probability in excess of $3/4$, $G[V_1]$ is $(\lambda m)$-joined.
\renewcommand{\qedsymbol}{$\boxdot$}
\qedhere
\renewcommand{\qedsymbol}{$\square$}
\qedsymbol
\end{proof}
\renewcommand{\qedsymbol}{}
\end{proof}
\renewcommand{\qedsymbol}{$\square$}

%%%%%%%%%%%%%%%%%%%%%%%%%%%%%%%%%%%%%%%%%%%%%%%%%%%%%%%%%%%%%%%%%%%%%%%%%%%%%%%%%%%%%%%%%%%%%%%%%
%%%%%%%%%%%%%%%%%%%%%%%%%%%%%%%%%%%%%%%%%%%%%%%%%%%%%%%%%%%%%%%%%%%%%%%%%%%%%%%%%%%%%%%%%%%%%%%%%%%%%%%%%
%%%%%%%%%%%%%%%%%%%%%%%%%%%%%%%%%%%%%%%%%%%%%%%%%%%%%%%%%%%%%%%%%%%%%%%%%%%%%%%%%%%%%%%%%%%%%%%%%%%%%%%%%
%%%%%%%%%%%%%%%%%%%%%%%%%%%%%%%%%%%%%%%%%%%%%%%%%%%%%%%%%%%%%%%%%%%%%%%%%%%%%%%%%%%%%%%%%%%%%%%%%%%%%%%%%
%%%%%%%%%%%%%%%%%%%%%%%%%%%%%%%%%%%%%%%%%%%%%%%%%%%%%%%%%%%%%%%%%%%%%%%%%%%%%%%%%%%%%%%%%%%%%%%%%%%%%%%%%

\subsection{Vortex partition}\label{sec:expandvortex}
We now use Lemma~\ref{lemma:vortex} to find a vortex $U_0\supset U_1\supset \ldots \supset U_{\ell}$ and then run a simple `cleaning procedure' to find disjoint sets $V_1,V_2,\ldots,V_\ell$ for a vortex partition (see Definition~\ref{defn:vortexpart}) such that $V_i$ is a subset of $U_i\setminus U_{i+1}$ which contains almost all those vertices for each $i\in [\ell]$, before then completing the vortex partition by using all but at most $m/4$ of the remaining vertices not in $\cup_{i\in [\ell]}V_i$ to form $V_0$. This will give us the following result.

\begin{lemma}\label{lemma:vortex:partition}
Let
\[
\frac{1}{m}\ll\mu \ll\frac{1}{K},\frac{1}{D}\ll\lambda\ll\frac1d\ll\gamma_1,\gamma_2\leq \frac{1}{16},
\]
with $\gamma_1<\gamma_2$ and
let $n$ be such that $m\leq \mu n$.
Let $G$ be an $n$-vertex $(m,\mu n)$-joined graph which is a $(D,m)$-expander. Let $\ell\in \mathbb{N}$ and let $\mathbf{n}=(n_0,n_1,\dots,n_\ell)$ be a $(\gamma_1,\gamma_2)$-descending tuple with $n=\sum_{i\in [\ell]_0}n_i$ and $\gamma_1Km\leq n_\ell\leq 2Km$.

Then, $G$ contains a subgraph $G'$ with at least $n-\lfloor{m/4}\rfloor$ vertices that has an $(\mathbf{n},2\lambda,d)$-vortex partition $V_0\cup V_1\cup \ldots \cup V_{\ell}$ such that $G[V_{\ell-1}\cup V_{\ell}]$ is $(\lambda m)$-joined.
\end{lemma}
\begin{proof}
Let $\mathbf{n}'=(n_0',n_1',\ldots,n_\ell')$ with $n_i'=\sum_{j=i}^\ell n_i$ for each $i\in [\ell]_0$.
Note that $\mathbf{n}'$ is $(\gamma_1/2,\gamma_2)$-descending, $n'_j/2\leq n_j\leq n'_j$ for all $j\in[\ell]_0$, and $\gamma_1 Km\leq n'_\ell\leq 2Km$.
Therefore, by Lemma~\ref{lemma:vortex}, $G$ contains an $(\mathbf{n}',m,2\lambda)$-vortex $U_0\supset U_1\supset \ldots \supset U_{\ell}$ such that $G[U_{\ell-1}]$ is $(\lambda m)$-joined and, for every set $U\subset V(G)$ with $|U|=\lfloor \frac{m}{4}\rfloor$, $|N_G(U,U_1\setminus U_2)|\geq \gamma_1Dm/200$.
For each $i\in [\ell-1]_0$, let $V'_i=U_i\setminus U_{i+1}$, and let $V'_\ell=U_\ell$. Note that, for each $i\in [\ell]_0$, $|V'_i|=n_i$.

%\propexpand*

Let $W_{\ell+1}=\emptyset$. Now, for each $j=\ell,\ell-1,\ldots,1$ in turn, do the following, where $W_{j+1}$ is a set chosen (if $j<\ell$) in the previous iteration such that $|W_{j+1}|<m$.
\stepcounter{propcounter}
\begin{enumerate}[label = \textbf{\Alph{propcounter}\arabic{enumi}}]
\item Note that, as $\lambda\ll\gamma_1$ and $d\ll K$, we have $|V_{j}'\setminus W_{j+1}|\geq n_j-m\geq 2\lambda n_{j-1}+(2d+2)m$.\label{prop:constructvorpart1}
\item Note that, by \ref{prop:vor2} and $V'_{j-1}\cup V'_j\subset U_{j-1}$, we have that $G[V'_{j-1}\cup (V_j'\setminus W_{j+1})]$ is $(m,\lambda n'_{j-1})$-joined, and hence $(m,2\lambda n_{j-1})$-joined.\label{prop:constructvorpart2}
\item Using Proposition~\ref{prop:expander:subgraphnewnew}, take $W_j\subset V'_{j-1}\cup (V_j'\setminus W_{j+1})$ such that $|W_j|<m$ and, for each $U\subset (V'_{j-1}\cup V_j')\setminus (W_j\cup W_{j+1})$ with $|U|\leq m$, $|N_G(U,V_{j}'\setminus(W_j\cup W_{j+1}))|\geq d|U|$.\label{prop:constructvorpart3}
\item Let $V_j=V_j'\setminus (W_{j}\cup W_{j+1})$.\label{prop:constructvorpart4}
\end{enumerate}
Now, for every set $U\subset V(G)$ with $|U|=\lfloor{m/4}\rfloor$, we have
$$|N_G(U,V_1)|\geq|N_G(U,V'_1)|-|W_1\cup W_2|\geq \gamma_1Dm/200-2m\geq\gamma_1Dm/500,$$
as $1/D\ll\gamma_1$. Let $W_0\subset V(G)$ be a maximal subset subject to $|W_0|< m/2$ and $|N_G(W_0,V_1)|\leq 10d|W_0|$. Similarly to the proof of Proposition~\ref{prop:expander:subgraphnewnew} and using $d\ll\gamma_1D$, we have $|W_0|<m/4$, and, for each $U\subset V(G)\setminus W_0$ with $|U|\le m/4$, $|N_G(U,V_1)|\geq 10d|U|$.
Let $V_0=V(G)\setminus (W_0\cup V_1\cup \ldots \cup V_\ell)$ and $G'=G[V_0\cup V_1\cup \ldots \cup V_\ell]$.

Note that $|G'|\geq n-m/4$, and that, as $V_{\ell-1}\cup V_{\ell}\subset V_{\ell-1}'\cup V_{\ell}'=U_{\ell-1}$, $G[V_{\ell-1}\cup V_{\ell}]$ is $(\lambda m)$-joined. Therefore, to complete the proof it is left only to show that $V_0\cup V_1\cup\dots \cup V_\ell$ is an $(\mathbf{n},2\lambda,d)$-vortex partition of $G'$.

Since $V'_0\setminus W_0\subset V_0\subset V_0'\cup W_1\cup \ldots \cup W_\ell$, $|W_0|\leq \floor{m/4}$ and $|W_1\cup \ldots \cup W_\ell|\leq\ell m$, we have $(1-2\lambda)n_0\leq|V_0|\leq(1+2\lambda)n_0$, proving the $j=0$ case of \ref{prop:vorpart1}. For all $j\in[\ell]$, $|V'_j|=n_j$, $V_j=V'_j\setminus (W_j\cup W_{j+1})$ and $|W_j|,|W_{j+1}|<m\leq\lambda n_{j}$, so we have $|V_j|\geq n_j-2m\geq(1-2\lambda)n_j$ and \ref{prop:vorpart1} holds. Now let $2\leq j\leq \ell$. Since $V_{j-1}\subset V'_{j-1}$ and $V_j\subset V_j'\setminus W_{j+1}$, $G[V_{j-1}\cup V_j]$ is $(m,2\lambda n_{j-1})$-joined by \ref{prop:constructvorpart2}, hence $(2\lambda n_{j-1})$-joined, and thus \ref{prop:vorpart2} holds. For \ref{prop:vorpart3}, let $U\subset V_{j-1}\cup V_j\subset (V'_{j-1}\cup V'_j)\setminus (W_j\cup W_{j+1})$ satisfy $|U|\leq 4\lambda n_{j-1}$. If $|U|\leq m$, we have $|N(U,V_j)|\geq d|U|$ by \ref{prop:constructvorpart3}. If $|U|>m$, then since $G[V_{j-1}\cup V_j]$ is $(m,2\lambda n_{j-1})$-joined and $d\lambda\ll\gamma_1$, we have $$|N(U,V_j)|\geq|V_j|-|U|-2\lambda n_{j-1}\geq(n_j-2m)-4\lambda n_{j-1}-2\lambda n_{j-1}\geq 4d\lambda n_{j-1}\geq d|U|.$$ Hence, $I(V_{j-1})$ is $(d,2\lambda n_{j-1})$-extendable in $G[V_{j-1}\cup V_j]$ by Proposition~\ref{prop:weakerextendability}, and so \ref{prop:vorpart3} holds.

For the $j=1$ case, \ref{prop:vorpart2} holds as $G\supset G[V_0\cup V_1]$ is $(m,\mu n)$-joined, thus $(m,2\lambda n_0)$-joined and $(2\lambda n_0)$-joined. Let $U\subset V_{0}\cup V_1$ satisfy $|U|\leq 4\lambda n_0$. If $|U|\leq \floor{m/4}$, then $|N(U,V_1)|\geq10d|U|\geq d|U|$. If $\floor{m/4}<|U|\leq m$, then let $U'\subset U$ have size $\floor{m/4}$, and thus \[|N(U,V_1)|\geq|N(U',V_1)|-|U\setminus U'|\geq10d|U'|-|U|\geq2dm-m\geq dm\geq d|U|.\] If $|U|>m$, then, since $G[V_0\cup V_1]$ is $(m,2\lambda n_0)$-joined, we have 
\[|N(U,V_1)|\geq|V_1|-|U|-2\lambda n_0\geq n_1-2m-6\lambda n_0\geq 4d\lambda n_0\geq d|U|.\]
Thus, we have that $I(V_0)$ is $(d,2\lambda n_0)$-extendable in $G[V_0\cup V_1]$, proving the $j=1$ case of \ref{prop:vorpart3}. This completes the proof that $V_0\cup V_1\cup\ldots\cup V_\ell$ is an $(\mathbf{n},d,2\lambda)$-vortex partition of $G'$.
\end{proof}
	%%%%%%%%%%%%%%%%%%%%%%%%%%%%%%%%%%%%%%%%%%%%%%%%%%%%%%%%%%%%%%%%%%%%%%%%%%%%%%%%%%%%%%
	%%%%%%%%%%%%%%%%%%%%%%%%%%%%%%%%%%%%%%%%%%%%%%%%%%%%%%%%%%%%%%%%%%%%%%%%%%%%%%%%%%%%

\subsection{Proof of Theorem~\ref{theorem:k=2}}\label{sec:mainthmk=2}

Using Lemma~\ref{lemma:vortex:partition}, we can now prove Theorem~\ref{theorem:k=2} following the outline at the start of this section.

\begin{proof}[Proof of Theorem~\ref{theorem:k=2}]
Let $d=10^4\Delta^5$, $\gamma=1/10\Delta$ and note that $d\gg\Delta^3/\gamma$. Let
\[\frac1m\ll\mu \ll \frac1K,\frac1D\ll \lambda\ll \frac1d.\]
Let $\gamma_1=\gamma/4\Delta$ and $\gamma_2=2\gamma$, and let $T$ be a tree with $n-m+1$ vertices and $\Delta(T)\le\Delta$. Using Lemma~\ref{lemma:ngammadecomposition}, we can find some $\ell\in\mathbb N$, some $\mathbf{n'}=(n'_1,\ldots,n'_\ell)\in\mathbb{N}^\ell$ that is $(\gamma_1,\gamma_2)$-descending with $\gamma Km/3\Delta\leq n'_\ell\leq Km$, and an $\mathbf{n}'$-decomposition $(T_1,\ldots,T_\ell)$ of $T$.
For $1\le i\le\ell-1$, let $t_{i}$ be the leaf of $T_i$ in $T_{i+1}$, and let $t_{\ell}$ be an arbitrary leaf of $T_{\ell}$ which is not $t_{\ell-1}$.
%, and let $t_{i+1}$ be the end of $e_i$ that is in $T_{i+1}$.
As every tree with at least 2 vertices has at least 2 leaves, we may also take a leaf $t_0$ of $T_1$ which is not $t_1$.%$\in V(T_0)$ not intersecting $e_0$.

Let $n_0=\ceil{(1-\gamma_2)n'_1}$, let $n_j=\floor{\gamma_2n_j'}+\ceil{(1-\gamma_2)n'_{j+1}}$ for each $j\in [\ell-1]$, and let $n_\ell=\floor{\gamma_2n'_\ell}+m-1$.
Note that, for each $j\in [\ell]$, $n_j\leq n_j'$ and 
\begin{equation}\label{eqn:nisum}
\sum_{i=j}^\ell n_i= \floor{\gamma_2n_j'}+\sum_{i=j+1}^{\ell}n'_i+(m-1),
\end{equation} and
\begin{equation}\sum_{i=0}^\ell n_i=\sum_{i=1}^\ell n'_i+(m-1)=|T|+(m-1)=n.\label{eqn:nisumall}
\end{equation}
Moreover, using $\mathbf{n}'$ is $(\gamma_1,\gamma_2)$-descending, we can verify that $\textbf{n}=(n_0,n_1,\ldots,n_{\ell})$ is $(\gamma_1/2,3\gamma_2)$-descending.
Let $K'=K/5\Delta$ and note that $n_\ell=\floor{\gamma_2n'_\ell}+m-1\leq 2\gamma_2n'_\ell\leq 2K\gamma_2 m=2Km/5\Delta=2K'm$ as $\gamma_2=1/5\Delta$, while $n_\ell\geq \floor{\gamma_2n'_\ell}\geq \gamma_2\cdot \gamma Km/6\Delta\geq \gamma_2\cdot 5\gamma K'm/6=\gamma K'm/6\Delta\geq \gamma_1K'm/2$ as $\gamma_1=\gamma/4\Delta$. Since $\mu\ll 1/K'\ll \lambda$, using Lemma~\ref{lemma:vortex:partition}, we can take a subgraph $G'$ of $G$ with at least $n-\lfloor m/4\rfloor$ vertices which has an $(\mathbf{n},2\lambda,d)$-vortex partition, $V_0\cup V_1\cup \ldots \cup V_{\ell}$ say, such that $G[V_{\ell-1}\cup V_{\ell}]$ is $(\lambda m)$-joined.

For each $j\in[\ell]_0$, say that we have a \textit{stage j situation} if we have distinct vertices $v_0,v_1,\ldots,v_{j}$ and copies $S_1,\ldots,S_j$ of the trees $T_1,\ldots, T_j$, respectively, so that the following hold.
\stepcounter{propcounter}
\begin{enumerate}[label = \textbf{\Alph{propcounter}\arabic{enumi}}]
%\item
\item If $i\in[j]_0$, then $v_{i}\in V_{i}$.\label{prop:stagejsit1}
\item If $i\in[j]$, then $\cup_{i'=0}^{i-1}V_{i'}\subset \cup_{i'=1}^{i}V(S_{i'})\subset \cup_{i'=0}^{i}V_{i'}$.\label{prop:stagejsit2}
\item For each $i\in[j]$, $t_{i-1}$ is copied to $v_{i-1}$ in $S_{i}$ and $t_{i}$ is copied to $v_{i}$ in $S_i$.\label{prop:stagejsit3}
\item For each $1\leq i<i'\leq j$, $V(S_i)\cap V(S_{i'})=\{v_{i}\}\cap \{v_{i'-1}\}$.\label{prop:stagejsit4}
%\item For each $0\leq i\in \leq j$, $V(S_i)\subset V_{i-1}\cup V_i$, $t_{i}$ is copied to $v_i$ and $t_{i+1}$ is copied to $v_{i+1}$.
%\item For each $0\leq i\leq j$, if $i<\ell$, then $v_{i+1}$ is in $V_{i+1}$.
%\item For each $0\leq i\in \leq j$, $V_i\setminus V(S_{i-1})\subset V(S_i)\subset \{v_i\}\cup (V_i\cup V_{i+1})\setminus V(S_{i-1})$.
\end{enumerate}
Arbitrarily, pick $v_0\in V_0$, and note that this gives us a stage $0$ situation as \ref{prop:stagejsit1} holds and \ref{prop:stagejsit2}--\ref{prop:stagejsit4} are vacuous.
Furthermore, if we have a stage $\ell$ situation, with vertices $v_0,v_1,\ldots,v_{\ell}$ and copies $S_1,\ldots,S_\ell$ of the trees $T_{1},\ldots, T_\ell$, respectively, then \ref{prop:stagejsit3} and \ref{prop:stagejsit4} imply that $\cup_{i\in [\ell]}S_i$ is a copy of $T=\cup_{i\in [\ell]}T_i$ in $G'$, and hence $G$ has a copy of $T$, as required. Thus, the lemma is implied by the following claim and induction.

\begin{claim}\label{clm:improvesituation}
For each $j\in [\ell]$, if we have a stage $j-1$ situation, then we can create a stage $j$ situation.
\end{claim}
\begin{proof}[Proof of Claim~\ref{clm:improvesituation}] 
Set $S_0$ to be the tree containing only the vertex $v_0$. Fix $j\in[\ell]$ and let $X_j=V_{j-1}\setminus V(S_{j-1})$.
Now, as $X_j\cup\{v_{j-1}\}\subset V_{j-1}$, and $V_0\cup V_1\cup \ldots \cup V_{\ell}$ is an $(\mathbf{n},2\lambda,d)$-vortex partition, we have the following by \ref{prop:vorpart3}.
\stepcounter{propcounter}
\begin{enumerate}[label = \textbf{\Alph{propcounter}}]%\arabic{enumi}
\item $I(X_j\cup \{v_{j-1}\})$ is $(d,\lambda n_{j-1})$-extendable in $G[X_j\cup \{v_{j-1}\}\cup V_j]$.\label{prop:clmpf1}
\end{enumerate}
Then, if $j\leq \ell-1$,
\begin{align}
|X_j\cup &\{v_{j-1}\}\cup  V_{j}|=1+|(V_{j-1}\cup V_{j})\setminus V(S_{j-1})|\overset{\ref{prop:stagejsit2}}{=}1+\left|\left(\cup_{i=0}^{j}V_i\right)\setminus \left(\cup_{i=0}^{j-1}V(S_{i})\right)\right|\nonumber\\
%|U_j|&=1+|(V_j\cup V_{j+1})\setminus V(S_{j})|=1+|\left(\cup_{i\in [j+1]_0}V_i\right)\setminus V\left(\cup_{i\in [j]_0}S_j\right)|\\
&=1+|G'|-|\cup_{i=j+1}^\ell V_i|-\Big(|T|-\sum_{i=j}^\ell n'_i\Big)\nonumber\\
&\overset{\ref{prop:vorpart1}}{\geq} 1+\left(n-\floor{\frac m4}\right)-\sum_{i=j+1}^\ell(1+\lambda)n_i-|T|+\sum_{i=j}^\ell n'_i\nonumber\\
&\overset{\eqref{eqn:nisum}}{\geq} \left(n+1-|T|-\floor{\frac m4}\right)-\sum_{i=j+1}^\ell n_i-2\lambda n_{j+1}+\sum_{i=j}^\ell n'_i\nonumber\\
%(1+\lambda)\left(\left\lfloor \frac{|T_{j+1}|}{100}\right\rfloor+\sum_{i=j+2}^{\ell}|T_i|+(m-1)-(\ell-(j+1))\right)+\sum_{i=j}^\ell |T_i|-(\ell-j)\nonumber\\
&\overset{\eqref{eqn:nisum}}{\geq} \left(m-\floor{\frac m4}\right)-\Big(\floor{\gamma_2n'_{j+1}}+\sum_{i=j+2}^\ell n_i'+m-1\Big)-2\lambda n'_{j+1}+\sum_{i=j}^\ell n'_i\nonumber
\end{align}
\begin{align}
&\geq n'_j+\frac{n'_{j+1}}2-\frac m2\geq\left(1+\frac{\gamma_1}{2}\right)n'_j-\frac m2\hspace{4.4cm} \textcolor{white}{.}\nonumber\\
&\geq\left(1+\frac{\gamma_1}{2}\right)|T_j|-m\geq |T_j|+20dm,\label{prop:clmpf2}
%\left(m-1-\frac{m}{100}\right)+|T_j|+|T_{j-1}|
%-(1+\lambda)\left\lfloor \frac{|T_{j+1}|}{100}\right\rfloor-\lambda\sum_{i=j+2}^{\ell}|T_i|+|T_j|+|T_{j+1}|\nonumber\\
%&\geq |T_j|+\frac{99m}{100}+\frac{(|T_{j+1}|-1)}{2}.\nonumber
\end{align}
where the last inequality follows from $\gamma_1|T_j|\geq\gamma_1\cdot \gamma Km/3\Delta\geq20dm$.
Furthermore,

\begin{align}
|X_j|&=|V_{j-1}\setminus V(S_{j-1})|\overset{\ref{prop:stagejsit2}}{=}\left|\left(\cup_{i=0}^{j-1}V_i\right)\setminus \left(\cup_{i=0}^{j-1}V(S_{i})\right)\right|\nonumber\\
%|U_j|&=1+|(V_j\cup V_{j+1})\setminus V(S_{j})|=1+|\left(\cup_{i\in [j+1]_0}V_i\right)\setminus V\left(\cup_{i\in [j]_0}S_j\right)|\\
&=|G'|-|\cup_{i=j}^\ell V_i|-(|T|-\sum_{i=j}^\ell n'_i)\nonumber\\
&\leq n-\sum_{i=j}^\ell (1-\lambda)n_i-|T|+\sum_{i=j}^\ell n'_i\nonumber\\%-(\ell-j)-1
&\leq m+2\lambda n_j-\sum_{i=j}^\ell n_i +\sum_{i=j}^\ell n'_i\nonumber\\
&\leq n'_j+2\lambda n'_j-\floor{\gamma_2n'_j}+1\leq\left(1-\frac{\gamma_2}2\right)n'_j\nonumber\\
&\leq\left(1-\frac{\gamma_2}2\right)|T_j|=(1-\gamma)|T_j|. \label{prop:clmpf3}
\end{align}
Finally, we have $|T_j|\geq n'_j\geq\gamma_1n'_{j-1}\geq\gamma_1n_{j-1}\geq2d^2\lambda n_{j-1}$ and $G[X_j\cup \{v_{j-1}\}\cup V_j]$ is $(\lambda n_{j-1})$-joined by \ref{prop:vorpart2}. Therefore, by \ref{prop:clmpf1}, \eqref{prop:clmpf2}, \eqref{prop:clmpf3} and
Lemma~\ref{lemma:iterativeembed}, we can find a copy $S_j$ of $T_j$ in $G[X_j\cup \{v_{j-1}\}\cup V_j]$ with $t_{j-1}$ copied to $v_{j-1}$, $t_j$ copied into $V_j$, and $X_j\subset V(S_j)$. Then, letting $v_j$ to be the copy of $t_j$, we have a stage $j$ situation.

Suppose then that $j=\ell$. From \ref{prop:clmpf1}, we have that $I(X_\ell\cup\{v_{\ell-1}\})$ is $(d,\lambda n_{\ell-1})$-extendable in $G[X_\ell\cup \{v_{\ell-1}\}\cup V_\ell]$, and hence $(d,\lambda m)$-extendable.
As $X_\ell\cup \{v_{\ell-1}\}\cup V_\ell\subset V_{\ell-1}\cup V_{\ell}$, we have that $G[X_\ell\cup \{v_{\ell-1}\}\cup V_\ell]$ is $(\lambda m)$-joined. Furthermore,
\begin{align*}
|X_\ell\cup \{v_{\ell-1}\}\cup V_\ell|&
=1+|(V_{\ell-1}\cup V_{\ell})\setminus V(S_{\ell-1})|\overset{\ref{prop:stagejsit2}}{=}1+\left|\left(\cup_{i=0}^{\ell}V_i\right)\setminus \left(\cup_{i=0}^{\ell-1}V(S_{i})\right)\right|\nonumber\\
&\geq 1+|G'|-(|T|-|T_\ell|+1)\geq n-\floor{\frac m4}-|T|+|T_\ell|.\\
&\geq |T_\ell|+\frac m2\geq |T_\ell|+10\lambda d m.\label{eqn:penum}
\end{align*}
Therefore, by Corollary~\ref{corollary:extendable:embedding}, there is a copy $S_\ell$ of $T_\ell$ in $G[X_\ell\cup\{v_{\ell-1}\}\cup V_\ell]$ in which $t_{\ell-1}$ is copied to $v_{\ell-1}$. Let $v_\ell$ be the copy of $t_\ell$ and note that we have a stage $\ell$ situation, as required.
This completes the proof of the claim and hence the theorem.
\renewcommand{\qedsymbol}{$\boxdot$}
\qedhere
\renewcommand{\qedsymbol}{$\square$}
\qedsymbol
\end{proof}
\renewcommand{\qedsymbol}{}
\end{proof}
\renewcommand{\qedsymbol}{$\square$}

%% file: embeddingpicture.tex
\begin{center}
\begin{tikzpicture}

\def\wi{16}
\def\hgt{4}
\def\n{2}
\def\m{11}
\def\vx{0.05cm}

\coordinate (A1) at (0,0);
\coordinate (A2) at ($(\wi,0)-\wi*(0.65,0)$);
\coordinate (A3) at ($(\wi,0)-\wi*(0.4,0)$);
\coordinate (A4) at ($(\wi,0)-\wi*(0.2,0)$);
\coordinate (A5) at ($(\wi,0)-\wi*(0.1,0)$);
\coordinate (A6) at ($(\wi,0)-\wi*(0.03125,0)$);
\coordinate (A7) at ($(\wi,0)-\wi*(0.015625,0)$);
\coordinate (A8) at ($(\wi,0)-\wi*(0.015625,0)$);

\draw [rounded corners] ($(A1)+(0.3,0)$) -- ++(0,-0.3) -- ($(A2)+(0.75,-0.3)$) -- ($(A2)+(0.75,0.3)$) -- ($(A1)+(0.3,0.3)$) -- ($(A1)+(0.3,0)$);
\draw ($0.5*(A1)+0.5*(A2)+0.5*(0.3+0.75,0)$) node {$T_1$};
\draw [rounded corners,black!40] ($(A2)+(0.85,0)$) -- ++(0,-0.3) -- ($(A3)+(0.55,-0.3)$) -- ($(A3)+(0.55,0.3)$) -- ($(A2)+(0.85,0.3)$) -- ($(A2)+(0.85,0)$);
\draw [black!50] ($0.5*(A2)+0.5*(A3)+0.5*(0.85+0.55,0)$) node {$T_2$};
\draw [rounded corners] ($(A3)+(0.65,0)$) -- ++(0,-0.3) -- ($(A4)+(0.45,-0.3)$) -- ($(A4)+(0.45,0.3)$) -- ($(A3)+(0.65,0.3)$) -- ($(A3)+(0.65,0)$);
\draw ($0.5*(A3)+0.5*(A4)+0.5*(0.65+0.45,0)$) node {$T_3$};

\coordinate (t0) at ($(0.45,0)$);
\coordinate (t1) at ($(A2)+(0.95,0)$);
\coordinate (t2) at ($(A3)+(0.8,0)$);
\coordinate (t3) at ($(A4)+(0.6,0)$);

\draw ($(A1)-(0.3,0)$) node {$T:$};

\draw [fill] (t0)  circle [radius=\vx];
\draw [fill] (t1)  circle [radius=\vx];
\draw  (t1) -- ++($0.2*(-1,0)$);
\draw [fill=black!40,black!40] (t2)  circle [radius=\vx];
\draw [black!40] (t2) -- ++($0.25*(-1,0)$);
\draw [fill] (t3)  circle [radius=\vx];
\draw  (t3) -- ++($0.15*(-1,0)$);

\draw ($(t0)+(0.3,0)$) node {$t_0$};
\draw ($(t1)+(0.3,0)$) node {$t_1$};
\draw [black!50] ($(t2)+(0.3,0)$) node {$t_2$};
\draw ($(t3)+(0.3,0)$) node {$t_3$};

\draw ($(-0.6,0)+0.5*(A5)+0.5*(A6)$) node {$\dots$};

%\draw ($0.5*(A2)$) circle [x radius=3cm,y radius=0.75cm];
\end{tikzpicture}
\end{center}

\vspace{0.1cm}

\begin{tikzpicture}

\def\wi{16}
\def\hgt{4}
\def\n{2}
\def\m{11}
\def\vx{0.05cm}

\coordinate (A1) at (0,0);
\coordinate (A2) at ($(\wi,0)-\wi*(0.65,0)$);
\coordinate (A3) at ($(\wi,0)-\wi*(0.4,0)$);
\coordinate (A4) at ($(\wi,0)-\wi*(0.2,0)$);
\coordinate (A5) at ($(\wi,0)-\wi*(0.1,0)$);
\coordinate (A6) at ($(\wi,0)-\wi*(0.03125,0)$);
\coordinate (A7) at ($(\wi,0)-\wi*(0.015625,0)$);
\coordinate (A8) at ($(\wi,0)-\wi*(0.015625,0)$);

\foreach \x in {1,...,7}
{
\coordinate (B\x) at ($(A\x)+(0,\hgt)$);
}

\foreach \x in {1,...,7}
{
\draw (A\x) -- (B\x);
}

\draw (A1) -- (B1) -- (B7) -- (A7) -- (A1);

\foreach \x/\y/\z in {1/2/0,2/3/1,3/4/2,4/5/3,5/6/4}
{
%\draw ($0.25*(A\x)+0.25*(A\y)+0.25*(B\x)+0.25*(B\y)$) node {$V_\z$};
\draw ($0.5*(B\x)+0.5*(B\y)+(0,0.4)$) node {$V_\z$};
}
\foreach \x/\y/\z in {6/7/\cdots}
{
%\draw ($0.25*(A\x)+0.25*(A\y)+0.25*(B\x)+0.25*(B\y)$) node {$V_\z$};
\draw ($0.5*(B\x)+0.5*(B\y)+(0,0.4)$) node {$\z$};
}

\coordinate (t0) at ($(B1)+(0.6,-0.6)$);
\coordinate (t1) at ($(A2)+(1.15,0.5)$);
\coordinate (t2) at ($(B3)+(1.65,-0.6)$);
\coordinate (t3) at ($(B4)+(1.2,-0.9)$);

\draw [fill] (t0)  circle [radius=\vx];
\draw [fill] (t1)  circle [radius=\vx];
\draw  (t1) -- ++($0.28*(-1,1)$);
\draw [fill=black!40,black!40] (t2)  circle [radius=\vx];
\draw [black!40] (t2) -- ++($0.41*(-1,0)$);
\draw [fill] (t3)  circle [radius=\vx];
\draw  (t3) -- ++($0.29*(-1,0.5)$);

\draw ($0.5*(A1)+0.5*(B1)-(0.45,0)$) node {$G:$};

\draw ($(t0)+(0,0.3)$) to [out=180,in=90] ++(-0.3,-0.3) to [out=270,in=90] ++(0,-\hgt+1.2) to [out=270,in=180] ++(0.3,-0.3) -- ($(A2)+(0,0.3)$);% to [out=0,in=236.3] ++($0.5*\hgt*(0.67,1)$) ;%($(B2)+(1.5,-0.3)$);

\draw [rounded corners] ($(A2)+(0,0.3)$) -- ($(A2)+(0.6,0.3)$) -- ($(B2)+(2.5,-0.3)$) -- ($(B2)+(0,-0.3)$) -- ($(t0)+(0,0.3)$);

\draw [rounded corners,black!40] ($(B3)+(0,-0.3)$) -- ($(B2)+(2.65,-0.3)$) -- ($(A2)+(0.85,0.3)$) -- ($(A3)+(0.6,0.3)$) -- ($(B3)+(1.3,-0.3)$) -- ($(B3)+(0,-0.3)$);
\draw [rounded corners,black] ($(B4)+(0,-0.3)$) -- ($(B3)+(1.45,-0.3)$) -- ($(A3)+(0.85,0.3)$) -- ($(A4)+(0.3,0.3)$) -- ($(B4)+(1,-0.3)$) -- ($(B4)+(0,-0.3)$);

\draw ($0.5*(B5)+0.5*(A6)-(1,0)$) node {$\dots$};

\draw ($0.5*(A1)+0.5*(B2)+(1,0)$) node {$S_1$};
\draw [black!50] ($0.5*(A2)+0.5*(B3)+(1.2,0)$) node {$S_2$};
\draw ($0.5*(A3)+0.5*(B4)+(0.9,0)$) node {$S_3$};

\draw ($(t0)+(0.25,-0.25)$) node {$v_0$};
\draw ($(t1)+(0.25,+0.25)$) node {$v_1$};
\draw [black!50] ($(t2)+(0.25,-0.25)$) node {$v_2$};
\draw ($(t3)+(0.15,-0.25)$) node {$v_3$};

\end{tikzpicture}

%% file: 6inductiononk.tex
In this section, we prove Theorem~\ref{theorem:main:1} by induction from Theorem~\ref{corollary:k=2}.
To embed a tree $T$ in a graph $G$ in the setting of Theorem~\ref{theorem:main:1}, we will again show we can assume some extra expansion condition, before dividing into 3 (overlapping) cases. Roughly, these cases are the following.
%We will have three (overlapping) cases, as follows.
\begin{enumerate}[label = \alph{enumi})]
\item $T$ has linearly many leaves.
\item $G$ is not well connected.%, as removing some set of $o(n)$ vertices allows its vertices to partitioned into 2 disjoint classes with no edges between them.%separates it into  it has a partition $V(G)=V_0\cup V_1\cup V_2$
\item $G$ is well connected and $T$ does not have linearly many leaves.
\end{enumerate}
In each case the embedding is different, and we sketch these at the start of Sections~\ref{sec:manyleaves}--\ref{sec:fewleaveswellconnected} respectively, before combining them to prove Theorem~\ref{theorem:main:1} in Section~\ref{sec:finalproof}.

%%%%%%%%%%%%%%%%%%%%%%%%%%%%%%%%%%%%%%%%%%%%%%%%%%%%%%%%%%%%%%%%%%%%%%%%%%%%%%%%%%%%%%%%%%%%%%%%%%%
%%%%%%%%%%%%%%%%%%%%%%%%%%%%%%%%%%%%%%%%%%%%%%%%%%%%%%%%%%%%%%%%%%%%%%%%%%%%%%%%%%%%%%%%%%%%%%%%%%%
%%%%%%%%%%%%%%%%%%%%%%%%%%%%%%%%%%%%%%%%%%%%%%%%%%%%%%%%%%%%%%%%%%%%%%%%%%%%%%%%%%%%%%%%%%%%%%%%%%%

\subsection{Trees with linearly many leaves}\label{sec:manyleaves}
%We now prove Theorem~\ref{theorem:dense} for trees with linearly many leaves, resulting below in Lemma~\ref{lemma:trees:manyleaves}. This result was shown by Balla, Pokrovskiy, and Sudakov~\cite{balla2018}, and we follow their method while using the extendability framework. For completion we include the full proof. To embed the tree, we remove a large set of leaves, and embed the resulting tree into the graph in an extendable fashion. We then show that the relevant extendability conditions imply that a version of Hall's matching criterion holds in the appropriate subgraph to allow unused vertices to be attached to the embedded tree as leaves in order to create a copy of the original tree.
%To embed a tree $T$ with many leaves in a graph $G$ for Lemma~\ref{lemma:trees:manyleaves}, we remove a large set of leaves which share no neighbours to get a tree $T'$. We then use Proposition~\ref{prop:expander:subgraphnew} to get a small set $W\subset V(G)$ such that $G-W$ has an enhanced expansion condition that we use to get an extendable copy of $T'$ using Corollary~\ref{corollary:extendable:embedding}. Then, using that this copy of $T'$ is extendable to satisfy Hall's matching criterion in the appropriate subgraph, we find a matching in $G$ to attach leavess to the copy of $T'$ to get a copy of $T$.

We first show that, in the setting of Theorem~\ref{theorem:main:1}, if $T$ has linearly many leaves and $G$ satisfies a simple expansion condition (\itref{prop:manyleaves2}), then $G$ contains a copy of $T$.
% which will arise naturally by induction.
This was shown by Balla, Pokrovskiy, and Sudakov~\cite{balla2018}, but for completion we include the full proof, following their method while using the extendability framework.

To embed $T$, we remove a large set of leaves, and embed the resulting tree into the graph in an extendable fashion. We then show that the relevant extendability conditions imply that a version of Hall's matching criterion holds in the appropriate subgraph to allow unused vertices to be attached to the embedded tree as leaves in order to create a copy of the original tree.

\begin{restatable}{lemma}{lemmalinearlymanyleaves}\label{lemma:linearlymanyleaves}
Let $1/n\ll \mu\ll 1/ \Delta$. Let $G$ be a graph with at least $n$ vertices in which the following holds.
\stepcounter{propcounter}
\begin{enumerate}[label = \textbf{\emph{\Alph{propcounter}}}]%\arabic{enumi}
%\item Every set $U\subset V(G)$ with $|U|=m$ satisfies $|N_G(U)|\geq (1-\lambda)n$.\label{prop:manyleaves1}
\item Every set $U\subset V(G)$ with $|U|=\mu n$ satisfies $|U\cup N_G(U)|\geq n$.\label{prop:manyleaves2}
\end{enumerate}
Then, $G$ contains a copy of every $n$-vertex tree $T$ with $\Delta(T)\le \Delta$ and at least $10\Delta^2\mu n$ leaves.
%Let $n,m,m',\Delta\in\mathbb N$ satisfy $n\ge (4\Delta+2)m'+m$. Let $G$ be an $(m,m')$-joined graph on $n$ vertices and let $T$ be a tree with $|T|\le n-m+1$, $\Delta(T)\le \Delta$, and containing at least $13\Delta^2 m'$ leaves.  Then, $G$ contains a copy of $T$.
\end{restatable}
\begin{proof} Let $T$ be an $n$-vertex tree with $\Delta(T)\le \Delta$ and at least $10\Delta^2\mu n$ leaves.
Let $d=2\Delta$, $m=\mu n$ and $n_0=|G|-n+1$, so that  $G$ is $(m,n_0)$-joined by \itref{prop:manyleaves2}, and
\begin{equation}\label{eqn:Gsize}
|G|=n_0+n-1\geq n_0+(1-10\Delta\mu)n+(4d+2)m.
\end{equation}
Using Proposition~\ref{prop:expander:subgraphnew}, we can find $W\subset V(G)$ with $|W|<m$ such that $G'=G-W$ is a $(2d,m)$-expander.

Arbitrarily, pick $v\in V(G')$. Note that, for each $U\subset V(G')$ with $1\leq |U|\leq m$, as $G'$ is a $(2d,m)$-expander, we have $|N_{G'}(U)|\geq 2d|U|\geq 1+d|U|$. By \itref{prop:manyleaves2}, if $U\subset V(G')$ with $m\leq |U|\leq 2m$, we also have $|N_{G'}(U)|\geq n-2m\geq 1+2dm\geq 1+d|U|$.
Thus, by Proposition~\ref{prop:weakerextendability}, $I(\{v\})$ is $(d,m)$-extendable in $G'$.

As $T$ has at least $10\Delta^2\mu n$ leaves, we can find a set of leaves $L$ such that $|L|=10\Delta \mu n$ and no pair of leaves in $L$ has a common neighbour in $T$. Letting $T'=T-L$, we will now embed $T'$ into $G'$. Note that
\[
1+|T'|\leq 1+n-10\Delta \mu n\overset{\eqref{eqn:Gsize}}{\leq} |G'|-(2d+2)m-n_0.
\]
Then, as $G$, and thus $G'$, is $(m,n_0)$-joined, by Corollary~\ref{corollary:extendable:embedding} (applied with an arbitrary $t\in V(T')$), there is a $(d,m)$-extendable subgraph $S'$ of $G'$ which is a copy of $T'$.

Let $A\subset V(S')$ be the copy in $S'$ of the set of parents in $T$ of the leaves in $L$, and let $B=V(G)\setminus V(S')$. Observe that to make $S'$ into a copy of $T$, it is sufficient to find a matching in $G$ from $A$ to $B$ and add this to $S'$. We will find such a matching by showing that the appropriate Hall's matching criterion holds.

%First, note that
%\begin{equation}\label{eqn:AB}
%|B|=n-|T'|=n-(n-m+1)-|L|=|L|+m-1=|A|+m-1.
%\end{equation}
If $U\subset A$, with $|U|\le m$, then, as $S'$ is $(d,m)$-extendable in $G'$, by \eqref{def:extendability} we have
\[
|N_G(U,B)|\ge |N_{G'}(U)\setminus V(S')|=|N'_{G'}(U)\setminus V(S')|{\ge} (d-1)|U|-(\Delta-1)|U|\ge |U|.
\]
%If $U\subset A$ with $m\leq |U|\leq |A|-\mu n$, then, as  $G$ is $(m,\mu n)$-joined we have $|B\setminus N_G(U)|<\mu n$,
%\[
%|N_G(U,B)|\geq |N_G(U)|-|T'|\overset{\itref{prop:manyleaves1}}{\geq} (1-\lambda)n-(n-|A|)\geq |A|-\lambda n\geq |U|.
%\]
On the other hand, if $U\subset A$ with $|U|\geq m$, then
\[
|N_G(U,B)|\geq |U\cup N_G(U)|-|T'|\overset{\itref{prop:manyleaves2}}{\geq} n-(n-|L|)= |A|\geq |U|.
\]
Therefore, for each $U\subset A$ we have $|N_G(U,B)|\geq |U|$, and thus, as Hall's matching criterion holds, $G$ has a matching from $A$ to $B$, as required.
\end{proof}

%%%%%%%%%%%%%%%%%%%%%%%%%%%%%%%%%%%%%%%%%%%%%%%%%%%%%%%%%%%%%%%%%%%%%%%%%%%%%%%%%%%%%%%%%%%%%%%%%%%
%%%%%%%%%%%%%%%%%%%%%%%%%%%%%%%%%%%%%%%%%%%%%%%%%%%%%%%%%%%%%%%%%%%%%%%%%%%%%%%%%%%%%%%%%%%%%%%%%%%
%%%%%%%%%%%%%%%%%%%%%%%%%%%%%%%%%%%%%%%%%%%%%%%%%%%%%%%%%%%%%%%%%%%%%%%%%%%%%%%%%%%%%%%%%%%%%%%%%%%

\subsection{Non-well-connected graphs}\label{sec:nonwellconnected}

%We `clean'
%We have a condition \itref{prop:nonwell2} that will follow from the induction framework.

For the next case in the induction step of Theorem~\ref{theorem:main:1}, which we prove as Lemma~\ref{lemma:nonwellconnected} below, we will use properties that will follow from induction (see~\itref{prop:nonwell2} below) and our graph $G$ will not be well-connected, in the sense that it will have a vertex partition $V_0\cup V_1\cup V_2$ in which $V_0$ contains at most a small linear (in $n$) number of vertices and there are no edges in $G$ between $V_1$ and $V_2$, which are both large sets (see \itref{prop:nonwell3}). We will start by, for each $i\in [2]$, finding a large subset $V'_i\subset V_i$, so that $G[V_i']$ has an expansion condition, for otherwise we will be done by the induction properties.
Neither of these properties will be necessarily strong enough to embed $T$. However, if a vertex has $\Delta$ neighbours in both $V_1'$ and $V_2'$ (this is Case I in the proof of Lemma~\ref{lemma:nonwellconnected}), then this will allow us to embed part of the tree in $G[V_1']$ and part of the tree in $G[V_2']$ in such a way that they can be connected through this vertex to get a copy of $T$. Thus, we can then assume that no such vertex exists (Case II in the proof of Lemma~\ref{lemma:nonwellconnected}). This will allow us to partition $V(G)\setminus (V_1'\cup V_2')$ as $U_1\cup U_2$ so that there are very few edges between $V_1'$ and $U_2$ and very few edges between $V_2'$ and $U_1$. Assuming that $G$ contains no copy of $T$, we then apply induction to find complete partite graphs in $G^c[V_1'\cup U_1]$ and $G^c[V_2'\cup U_2]$ which we can combine (with the deletion of some vertices from the larger parts) to get a complete $k$-partite graph in $G^c$ with the class sizes we need.%copy of $K_{\lambda n}^k\times K_m^c$ in $G^c$, where the parameters of .

\begin{restatable}{lemma}{lemmanonwellconnected}\label{lemma:nonwellconnected} Let $k\geq 3$ and $1/n\ll \lambda\ll \mu\ll 1/\Delta$. Let $m\leq \lambda n$ and  let $G$ be a graph with $(k-1)(n-1)+m$ vertices. Let $T$ be an $n$-vertex tree with $\Delta(T)\leq \Delta$. Suppose the following hold.
\stepcounter{propcounter}
\begin{enumerate}[label = \upshape\textbf{\Alph{propcounter}\arabic{enumi}}]
%\item For every $U\subset V(G)$ with $|U|=m$, we have $|N_G(U)|\geq (1-\lambda)n$.\label{prop:nonwell1}
\item For any $2\leq k'\leq k-1$ and $m'\leq \mu n$, $R(T,K_{\mu n}^{k'-1}\times K_{m'}^c)=(k'-1)(n-1)+m'$.\label{prop:nonwell2}
\item There is a partition $V(G)=V_0\cup V_1\cup V_2$ such that $e_G(V_1,V_2)=0$, $|V_1|,|V_2|\geq m$ and $|V_0|\leq \lambda n$.\label{prop:nonwell3}
\end{enumerate}
Then, $G$ contains a copy of $T$ or $G^c$ contains a copy of $K_{\lambda n}^{k-1}\times K_m^c$.
\end{restatable}
\begin{proof}
Assume for a contradiction that $G$ contains no copy of $T$ and $G^c$ contains no copy of $K_{\lambda n}^{k-1}\times K_m^c$. We start by showing that we have the following property.

\begin{enumerate}[label =\textbf{{\Alph{propcounter}\arabic{enumi}}}]\addtocounter{enumi}{2}
\item For every $U\subset V(G)$ with $|U|=m$, we have $|N_G(U)|\geq (1-2\lambda)n$.\label{prop:nonwell1}
\end{enumerate}

Indeed, if there is some set $U\subset V(G)$ with $|U|=m$ with $|N_G(U)|<(1-2\lambda)n$, then $|G-N_G(U)-U|>(k-2)(n-1)+\lambda n$, and thus, by \itref{prop:nonwell2}, as $G-N_G(U)-U$ contains no copy of $T$, its complement contains a copy of $K_{\lambda n}^{k-1}$. Adding $U$ as a vertex class, we obtain a copy of $K_{\lambda n}^{k-1}\times K_m^c$ in $G^c$, contradiction. Thus, \ref{prop:nonwell1} holds.%we can assume we have the following property.

Now, let $d=4\Delta$ and consider the partition $V(G)=V_0\cup V_1\cup V_2$ given by \ref{prop:nonwell3}. For each $i\in [2]$ and any $U\subset V_i$ with size $m$, we have $N_G(U)\subset V_i\cup V_0$, so that, by \ref{prop:nonwell1}, $|N_G(U,V_i)|\geq (1-2\lambda)n-|V_0|\geq (1-3\lambda)n$.
In particular, $|V_i|\geq m+(1-3\lambda)n$. For each $i\in [2]$, then, letting $n_i=|V_i|-(1-3\lambda)n$, we have that $G[V_i]$ is $(m,n_i)$-joined and, as $\lambda\ll 1/\Delta$ and $m\leq \lambda n$,
\begin{equation}\label{eqn:Visize}
|V_i|=n_i+(1-3\lambda)n\geq n_i+\left(1-\frac{1}{4\Delta}\right)n+1+(4d+2)m+2m.
\end{equation}
In particular, $|V_i|\geq n_i+(4d+2)m$, so, by Proposition~\ref{prop:expander:subgraphnew}, there is some $W_i\subset V_i$ with $|W_i|<m$ such that $G[V_i]-W_i$ is a $(2d,m)$-expander.
For each $i\in [2]$, let $V'_i=V_i\setminus W_i$.

We now consider 2 cases, depending on whether there is some $v\in V(G)\setminus (V_1'\cup V_2')$ with at least $\Delta$ neighbours both in $V_1'$ and in $V_2'$ (Case I) or not (Case II). In Case I, we will show that $G$ contains a copy of $T$, and in Case II, we will show that $G^c$ contains a copy of $K_{\lambda n}^{k-1}\times K_m^c$, a contradiction either way.

\medskip

\noindent\textbf{Case I.} Suppose then that there is some vertex $v\in V(G)\setminus (V_1'\cup V_2')$ with at least $\Delta$ neighbours both in $V_1'$ and in $V_2'$. Using Proposition~\ref{prop:splittree}, we can find two subtrees $T_1$ and $T_2$ of $T$ that share exactly one vertex $t$ that is a leaf in $T_1$, such that $E(T)=E(T_1)\cup E(T_2)$ and $n/4\Delta\leq|T_2|\leq n/2$. Let $t_0$ be the unique vertex in $T_1$ adjacent to $t$ and let $t_1,\ldots,t_r$ be the neighbours of $t$ in $T_2$. For each $i\in [r]$, let $T'_i$ be the tree in $T_2-t$ containing $t_i$. Using that $v$ has at least $\Delta\geq r$ neighbours in $V_1'$ and $V_2'$, we can find a neighbour $v_0$ of $v$ in $V_1'$ and $r$ distinct neighbours $v_1,\ldots,v_r$ of $v$ in $V'_2$.

Now, as $G[V_1]-W_1=G[V_1']$ is an $(m,n_1)$-joined $(2d,m)$-expander, we have that $I(\{v_0\})$ is $(d,m)$-extendable in $G[V_1']$. Futhermore, as $|T_1|=n-|T_2|+1\leq(1-1/4\Delta)n+1$, by \eqref{eqn:Visize}, we have $|V_1'|\geq |T_1|+(2d+2)m+n_1$.
Thus, by Corollary~\ref{corollary:extendable:embedding} there is a copy $S_1$ of $T_1-t$ in $G[V_1']$ in which $t_0$ is copied to $v_0$.

Similarly, as $G[V_2]-W_1=G[V_2']$ is an $(m,n_2)$-joined $(2d,m)$-expander, we have that $I(\{v_1,\ldots,v_r\})$ is $(d,m)$-extendable in $G[V_2']$ as $r\leq \Delta= d/4$. Futhermore, as $\sum_{i\in [r]}|T'_i|=|T_2|-1\leq n/2$,
by \eqref{eqn:Visize}, we have $|V_2'|\geq r+\sum_{i\in [r]}|T'_i|+(2d+2)m+n_2$. Thus, by induction and Corollary~\ref{corollary:extendable:embedding}, for each $j\in [r]$, there are vertex-disjoint copies $S'_1,\ldots,S'_j$ of $T'_1,\ldots,T'_j$, respectively, such that $t_i$ is copied to $v_i$ for each $i\in [j]$, and $I(\{v_1,\ldots,v_{r}\})\cup(\cup_{i\in [j]}S_i)$ is $(d,m)$-extendable in $G[V_2']$.
Note that at the end of the induction, $S_2=\left(\cup_{i=1}^rS'_i\right)$ is a copy of $T_2-t$ in $G[V'_2]$. Hence, $S_1\cup S_2$ together with the vertex $v$ and edges $vv_0,\ldots,vv_r$ form a copy of $T$ in $G$, contradiction. 
%\stepcounter{propcounter}
%\begin{enumerate}[label = \textbf{\Alph{propcounter}\arabic{enumi}}]
%\item $I(\{v_0\})$ is $(d,m)$-extendable in $G[V_1']$.
%\item $I(\{v_1,\ldots,v_r\})$ is $(d,m)$-extendable in $G[V_2']$.
%\item
%\item
%\end{enumerate}

\medskip

\noindent\textbf{Case II.} %Suppose then that every vertex $v\in V(G)\setminus (V_1'\cup V_2')$ has either fewer than $\Delta$ neighbours in $V_1'$ or fewer than $\Delta$ neighbours in $V_2'$.
Suppose, then, that we can partition $V(G)\setminus (V_1'\cup V_2')$ as $U_1\cup U_2$ so that every vertex in $U_1$ has fewer than $\Delta$ neighbours in $V_2'$ and every vertex in $U_2$ has fewer than $\Delta$ neighbours in $V_1'$.
For each $i\in [2]$, let $k_i$ and $m_i$ be such that $|V_i'\cup U_i|=k_i(n-1)+m_i$ and $0\leq m_i<n-1$. Switching the labels if necessary, we can assume that $m_1\geq m_2$. Now, as $V(G)=V_1'\cup U_1\cup V_2'\cup U_2$ is a partition,
\[
(k-1)(n-1)+m=(k_1+k_2)(n-1)+(m_1+m_2),
\]
and thus, we have that the following hold.
\begin{itemize}
\item If $m_1+m_2<n-1$ then $k-1=k_1+k_2$ and $m=m_1+m_2$.
\item If $m_2<m$, then as $(k-1)(n-1)+(m-m_2)=(k_1+k_2)(n-1)+m_1$ and $m_1<(n-1)$, we have $k-1=k_1+k_2$ and $m=m_1+m_2$.
\item $k_1+k_2\geq k-2$, so if $k_1+k_2\neq k-2$, then $k-1=k_1+k_2$, and hence $m=m_1+m_2$.
\end{itemize}
Thus, we either have $m_1\geq (n-1)/2$, $m_2\geq m$ and $k_1+k_2=k-2$ (Case II.1) or $k-1=k_1+k_2$ and $m=m_1+m_2$ (Case II.2).
Futhermore, as $|V_i'|\geq|V_i|-m\geq (1-3\lambda )n$ for each $i\in [2]$, we have that $k_1,k_2<k-1$.

\smallskip

\noindent\textbf{Case II.1:} $m_1\geq (n-1)/2$, $m_2\geq m$ and $k_1+k_2=k-2$. As $k_1<k-1$ and $m_1\geq n/3$, by \itref{prop:nonwell2}, we have that $V_1'\cup U_1$ contains disjoint sets $Y_1,\ldots,Y_{k_1+1}$ with $|Y_i|=\mu n$ for each $i\in [k_1+1]$, and there are no edges in $G$ between any pair $Y_i$ and $Y_j$ for each $1\leq i<j\leq k_1+1$.
Furthermore, as $k_2<k-1$ and $m_2\geq m$, by \itref{prop:nonwell2}, we have that $V_2'\cup U_2$ contains disjoint sets $Z_1,\ldots,Z_{k_2+1}$ with $|Z_i|=\mu n$ for each $i\in [k_2]$ and $|Z_{k_2+1}|=m$ so that there are no edges in $G$ between any pair $Z_i$ and $Z_j$ for $1\leq i<j\leq k_2+1$.

Note that
\begin{equation}\label{eqn:Uetc}
|U_1\cup (N(Z_{k_2+1})\cap V_1')|\leq |V_0|+|W_1|+|W_2|+(\Delta-1)|Z_{k_2+1}|\leq \lambda n+2m+\Delta m\leq \mu n/2.
\end{equation}
Hence, we can find $Y'_i\subset Y_i\setminus (U_1\cup (N(Z_{k_2+1})\cap V_1'))$ with size $\lambda n$ for each $i\in [k_1+1]$. Similarly, since $|U_2|\leq|V_0|+|W_1|+|W_2|\leq\lambda n+2m\leq \mu n/2$, we can find $Z_i'\subset Z_i\setminus U_2$ having size $\lambda n$ for each $i\in [k_2]$. Let $Z_{k_2+1}'=Z_{k_2+1}$, and observe that there no edges in $G$ between any of the sets
\[
Y_1',Y_2',\ldots, Y_{k_1+1}',Z'_1,Z_2',\ldots,Z'_{k_2+1},
\]
and these are $k_1+k_2+1=k-1$ sets with size $\lambda n$ and 1 set of size $m$, so that $G^c$ contains a copy of $K_{\lambda n}^{k-1}\times K_m^c$, contradiction.

\smallskip

\noindent\textbf{Case II.2:} $m_1+m_2=m$ and $k_1+k_2=k-1$. As $k_1<k-1$, by \itref{prop:nonwell2}, we have that $V_1'\cup U_1$ contains disjoint sets $Y_1,\ldots,Y_{k_1+1}$ with $|Y_i|=\mu n$ for each $i\in [k_1]$, and $|Y_{k_1+1}|=m_1$, such that there are no edges in $G$ between any pair $Y_i$ and $Y_j$ for each $1\leq i<j\leq k_1+1$.
Similarly, by \itref{prop:nonwell2}, we have that $V_2'\cup U_2$ contains disjoint sets $Z_1,\ldots,Z_{k_2+1}$ with $|Z_i|=\mu n$ for each $i\in [k_2]$ and $|Z_{k_2+1}|=m_2$, such that there are no edges in $G$ between any pair $Z_i$ and $Z_j$ for  each $1\leq i<j\leq k_2+1$.
For each $i\in [k_1]$, using a similar calculation to \eqref{eqn:Uetc}, let $Y'_i\subset Y_i\setminus (U_1\cup (N(Z_{k_2+1})\cap V_1'))$ have size $\lambda n$ and let $Y'_{k_1+1}=Y_{k_1+1}$. For each $i\in [k_2]$, again using a similar calculation to \eqref{eqn:Uetc}, let $Z_i'\subset Z_i\setminus (U_2\cup (N(Y_{k_1+1})\cap V_2'))$ have size $\lambda n$, and let $Z_{k_2+1}'=Z_{k_2+1}$. Then, observe that there no edges in $G$ between any of the sets
\[
Y_1',Y_2',\ldots, Y_{k_1}',Z'_1,Z_2',\ldots,Z'_{k_2}, Y'_{k_1+1}\cup Z'_{k_2+1},
\]
and these are $k_1+k_2=k-1$ sets with size $\lambda n$ and 1 set of size $m_1+m_2=m$, so that  $G^c$ contains a copy of $K_{\lambda n}^{k-1}\times K_m^c$, a contradiction.
\end{proof}

%%%%%%%%%%%%%%%%%%%%%%%%%%%%%%%%%%%%%%%%%%%%%%%%%%%%%%%%%%%%%%%%%%%%%%%%%%%%%%%%%%%%%%%%%%%%%%%%%%%
%%%%%%%%%%%%%%%%%%%%%%%%%%%%%%%%%%%%%%%%%%%%%%%%%%%%%%%%%%%%%%%%%%%%%%%%%%%%%%%%%%%%%%%%%%%%%%%%%%%
%%%%%%%%%%%%%%%%%%%%%%%%%%%%%%%%%%%%%%%%%%%%%%%%%%%%%%%%%%%%%%%%%%%%%%%%%%%%%%%%%%%%%%%%%%%%%%%%%%%

\subsection{Trees with few leaves in well-connected graphs}\label{sec:fewleaveswellconnected}
For the last case in the induction step for Theorem~\ref{theorem:main:1}, which we prove as Lemma~\ref{lemma:wellconnected}, we will have that the graph $G$ is well-connected, in that it has no partition $V(G)=V_0\cup V_1\cup V_2$ satisfying \itref{prop:nonwell3} (see \itref{prop:wellconn1}), as well as an additional expansion condition which will follow from our induction (see \itref{prop:wellconn2}), and the tree $T$ will have few leaves.  That is, we prove the following lemma.

\begin{restatable}{lemma}{lemmawellconnected}\label{lemma:wellconnected} Let $k\geq 3$ and $1/n\ll \eps \ll \lambda \ll 1/\Delta,1/k$. Let $m\leq \eps n$ and let $G$ be a graph with $(k-1)(n-1)+m$ vertices. 
Let $T$ be an $n$-vertex tree with fewer than $10\Delta^2\eps n$ leaves satisfying $\Delta(T)\leq \Delta$. Suppose the following properties hold.
\stepcounter{propcounter}
\begin{enumerate}[label = \textbf{\emph{\Alph{propcounter}\arabic{enumi}}}]
\item There is no partition $V(G)=V_0\cup V_1\cup V_2$ such that $e_G(V_1,V_2)=0$, $|V_1|,|V_2|\geq m$ and $|V_0|\leq \lambda n$.\label{prop:wellconn1}
\item For every set $U\subset V(G)$ with $|U|=m$, we have $|N_G(U)|\geq (1-\lambda)n$.\label{prop:wellconn2}
\end{enumerate}
Then, $G$ contains a copy of $T$ or $G^c$ contains a copy of $K^{k-1}_{\eps n}\times K_m^c$.
\end{restatable}
The property \itref{prop:wellconn1} will imply that any two disjoint sets $U$ and $U'\subset V(G)$, each with size $m$, will have many vertex-disjoint short paths between them in $G$. If we pick a small but linear-sized (in $n$) random set $Z\subset V(G)$, then for any pair of such sets $U$ and $U'$, plenty of these paths all have internal vertices in $Z$. This we prove as Proposition~\ref{prop:findW}, where we do a little additional work so that the paths we can find all have the same prescribed length.

In Lemma~\ref{lemma:wellconnected}, we have that $k\geq 3$, so that $G$ contains many more vertices than $T$. As $T$ has few leaves, most of $T$ will consist of bare paths (paths whose internal vertices have degree 2 in the tree $T$) of some long, constant, length. We will choose a random partition $V(G)=Z\cup V_0\cup V_1$, with $|V_0|\geq (1-o(1))\cdot 2n/5$, $|V_1|\geq 11n/10$ and so that $Z$ has the connection property described above. We then embed the tree $T$ steadily into a large subgraph of $G[V_0]$, but, at any available opportunity, try to save vertices from $V_0$ by embedding a long bare path of $T$ using many vertices in $V_1$ and at most a small number of vertices from $Z$, so that only the last vertex embedded is in $V_0$ (see Figure~\ref{fig:wellcon}). We do this so that the embedded tree remains extendable in the large subgraph of $G[V_0]$ that we use. After this embedding (described precisely in \textbf{a)}--\textbf{c)} of the proof of Lemma~\ref{lemma:wellconnected}), we analyse it and show that enough of the long bare paths are embedded outside of $V_0$ so that we do not run out of room within $V_0$, and thus can successfully embed $T$. This is shown in Claims~\ref{clm:prettyeasy}--\ref{clm:alsoprettyeasy}, but, roughly speaking, if we had $2m$ opportunities to embed a long path in $T$ and did not do so outside of $V_0$, then let $U$ be the set of vertices to which we could have attached a long bare path to extend the embedding of $T$, and do the following. We find vertex-disjoint long paths $Q_1,\ldots,Q_{2m}$ in $G$ using vertices in $V_1$ which have not yet been used in the embedding, and show that under our embedding rules we should have embedded one of the paths $Q_i$ connected to a vertex in $U$ using the connectivity property of $Z$. 
To find the paths $Q_1,\ldots,Q_{2m}$, we only need a loose result for the Ramsey numbers of paths versus general graphs as the paths have constant length compared to $n$, but for convenience we will use the following result.

\begin{theorem}[Pokrovskiy and Sudakov~\cite{pokrovskiy2017ramsey}]\label{theorem:ramseygoodnessforpaths}Let $k\ge 2$ and let $H$ be a graph with $\chi(H)=k$. Then, for all $n\geq4|H|$, $R(P_n,H)=(k-1)(n-1)+\sigma(H)$.
\end{theorem}

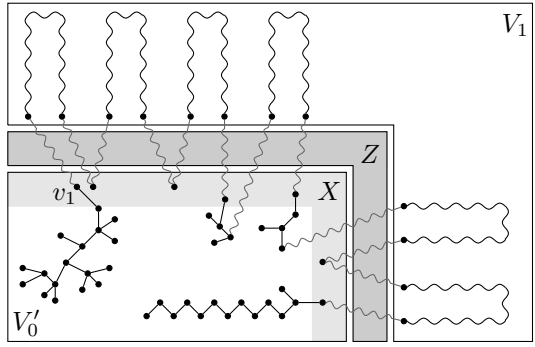
\begin{figure}
    \centering
      \input{wellconnectedpic}
    \caption{The embedding of $T$ in the proof of Lemma~\ref{lemma:wellconnected}. Starting with the extendable subgraph $I(X)$ in $G[V_0]$, in which a vertex of $T$ is embedded to $v_1$, we steadily embed $T$ into $G$ so that many of the long bare paths use vertices in $V_1$ along with some vertices in $Z$ to connect the path back into  $G[V_0]$ (the grey paths depicted have all their interior vertices in $Z$), and so that the union of $I(X)$ and the subgraph embedded in $V_0'$ remains extendable.}\label{fig:wellcon}
\end{figure}

We now prove a result showing that, in the above sketch, the random set $Z$ is likely to have the connectivity property we want.
\begin{proposition}\label{prop:findW} Let $k\geq 3$ and $1/n\ll \delta\ll 1/\ell\ll  \lambda \ll 1/k$. Let  $m\leq \delta n$ and let $G$ be a graph with $(k-1)(n-1)+m$ vertices such that \ref{prop:wellconn1} holds and $G^c$ contains no $K_{\delta n}^k$. % and \ref{prop:wellconn2} hold.
Let $Z\subset V(G)$ be a set formed by including each element independently at random with probability $1/5$.

Then, with probability more than $2/3$, for any two disjoint sets $U,U'\subset V(G)$ with size $m$, there are at least $\delta n$ internally vertex-disjoint paths with length $\ell$ through $Z$ connecting $U$ and $U'$ in $G$.
\end{proposition}
\begin{proof}
First, let $U,U'\subset V(G)$ be disjoint sets and let $Z_0\subset V(G)\setminus (U\cup U')$ satisfy $|U|=|U'|=m$ and $|Z_0|<\lambda n/2$. We will show that there is a path from $U$ to $U'$ with interior vertices not in $Z_0\cup U\cup U'$, and with length at least $2$ and at most $\ell$.

%Suppose, for contradiction, there is no such path, and
Let $r=\lfloor \ell/2\rfloor$. Let $U_0=U$ and, iteratively, for each $1\leq i\leq r$, let $U_i=U_{i-1}\cup N_{G-Z_0-U'}(U_{i-1})$.
Now, the sets $U_i\setminus U_{i-1}$, $i\in [r]$, are disjoint, so there must be some $j\in [r]$ with $|U_j\setminus U_{j-1}|\leq kn/r<\lambda n/2-m$, as $1/\ell\ll \lambda,1/k$ and $m\leq \delta n\leq \lambda n/4$. Then, $|Z_0\cup U'\cup (U_j\setminus U_{j-1})|<\lambda n$, $|U_{j-1}|\geq |U_0|=m$, and there are no edges in $G$ between $U_{j-1}$ and $V(G)\setminus (Z_0\cup U'\cup U_j)$, so
by \itref{prop:wellconn1} we must have $|V(G)\setminus (Z_0\cup U'\cup U_j)|<m$. Thus $|U_j|\geq |G|-|Z_0|-2m\geq |G|-\lambda n/2-2m>|G|/2$, and therefore $|U_r|\geq |U_j|>|G|/2$.

Similarly, letting $U_0'=U'$ and iteratively taking $U'_i=U'_{i-1}\cup N_{G-Z_0-U}(U'_{i-1})$ for each $1\leq i\leq r$, we have that $|U'_r|>|G|/2$. Thus, $U_r\cap U'_r\neq\emptyset$, and therefore there
is a path from $U$ to $U'$ with length at least 2 and at most $\ell$ in $G-Z_0$.

Thus, by taking $Z_0$ to be the internal vertices of a set of maximal internally-vertex-disjoint paths from $U$ to $U'$ in $G$ with length at least $2$ and at most $\ell$, we see that $G$ contains at least $\lambda n/3\ell$ internally-vertex-disjoint paths from $U$ to $U'$ with length at most $\ell$ and at least 2. Let $2\leq j\leq \ell$ be maximal subject to the condition that there are at least $r_j:=\lambda n/3^{j}\ell$ internally-vertex-disjoint paths from $U$ to $U'$ with length $j$. Such a $j$ exists as $\sum_{j=2}^\ell r_j<\lambda n/3\ell$.
If $j<\ell$, then let $P_1,\ldots,P_{r_j}$ be such a set of paths, and, for each $i\in [j]$, let $x_i$ be the vertex of the path neighbouring the end-vertex in $U$ (which is not in $U'$ as $j\geq 2$).
Note that if $G[\{x_i:i\in [r_j]\}]$ contains a matching with $r_{j+1}$ edges, then we can use this matching along with $P_1,\ldots,P_{r_j}$ to get $r_{j+1}$ internally-vertex-disjoint paths from $U$ to $U'$ in $G$ with length $j+1$, a contradiction.
On the other hand, if $G[\{x_i:i\in [r_j]\}]$ contains no matching with $r_{j+1}$ edges, then, removing a maximal matching shows that $G^c[\{x_i:i\in [r_j]\}]$, and hence $G^c$, contains a copy of $K_{r_j-2r_{j+1}}$, and hence a copy of $K^k_{\delta n}$ as $\delta\ll 1/\ell,\lambda$, contradiction.
Thus, we must have $j=\ell$.

Therefore, altogether, we have shown that, for any disjoint sets $U,U'\subset V(G)$ of size $m$, there are at least $\lambda n/3^\ell\ell$ internally-vertex-disjoint paths from $U$ to $U'$ with length $\ell$. For any such path, the probability all its internal vertices lie in $Z$ is $(1/5)^{\ell-1}$. Thus, as $\delta\ll \lambda,\ell$, by Lemma~\ref{lemma:chernoff}, the probability that there are fewer than $\delta n$ internally vertex-disjoint paths with length $\ell$ through $Z$ connecting $U$ and $U'$ in $G$ is at most $2\exp(-\lambda n/15^{\ell+1}\ell)$. Therefore, the probability there are disjoint sets $U,U'\subset V(G)$ of size $m$, with fewer than $\delta n$ internally-vertex-disjoint paths from $U$ to $U'$ with length $\ell$ and internal vertices in $Z$ is at most
\[
\binom{kn}{m}^2\cdot 2\exp\left(-\frac{\lambda n}{15^{\ell+1}\ell}\right)\leq \left(\frac{ekn}{m}\right)^{2m} \cdot 2\exp\left(-\frac{\lambda n}{15^{\ell+1}\ell}\right)\leq
\left(\frac{ek}{\delta}\right)^{2\delta n} \cdot 2\exp\left(-\frac{\lambda n}{15^{\ell+1}\ell}\right)<1/3,
\]
as required, where the last inequality follows as $1/n\ll \delta \ll \ell,1/k$.
\end{proof}

Next, we give a simple result to show that, in the sketch above, the set $V_0$ will have the expansion property in $G$ that we want.

\begin{proposition}\label{prop:findV0} Let $k\geq 3$ and $1/n\ll \lambda \ll 1/k$. Let $m\leq \lambda n$ and let $G$ be a graph with $(k-1)(n-1)+m$ vertices in which \ref{prop:wellconn2} holds. Let $V_0\subset V(G)$ be a set formed by including each element independently at random with probability $1/5$.
Then, with probability more than $2/3$, for any set $U\subset V(G)$ with $|U|=m$, $|N_G(U,V_0)|\geq n/10$.%there are at least $\delta n$ internally vertex-disjoint paths through $W$ connecting $U_1$ and $U_2$ in $G$.
\end{proposition}
\begin{proof} For each set $U\subset V(G)$ with $|U|=m$, we have, by \itref{prop:wellconn2}, that $\mathbb{E}|N_G(U,V_0)|\geq n/6$. Thus, by Lemma~\ref{lemma:chernoff}, the probability that there is some set $U\subset V(G)$ with $|U|=m$ and $|N_G(U,V_0)|< n/10$ is at most
\[
\binom{kn}{m}\cdot 2\exp\left(-\frac{n}{10^3}\right)\leq \left(\frac{ekn}{m}\right)^m\cdot 2\exp\left(-\frac{n}{10^3}\right)\leq \left(\frac{ek}{\lambda}\right)^{\lambda n}\cdot 2\exp\left(-\frac{n}{10^3}\right)<\frac{1}{3},
\]
as required, where we have used that $1/n\ll \lambda \ll 1/k$.
\end{proof}

Using Propositions~\ref{prop:findW} and~\ref{prop:findV0}, we can now prove Lemma~\ref{lemma:wellconnected} using the sketch above, as depicted in Figure~\ref{fig:wellcon}.

\begin{proof}[Proof of Lemma~\ref{lemma:wellconnected}] Let $L$, $\delta$ and $\ell$ satisfy $\eps\ll 1/L\ll \delta\ll 1/\ell\ll \lambda$, and let $d=4\Delta$.
Let $V(G)=Z\cup V_0\cup V_1$ be a partition chosen by selecting the set for each $v\in V(G)$ independently at random so that $\P(v\in Z)=\P(v\in V_0)=1/5$ and $\P(v\in V_1)=3/5$.

By a simple application of Lemma~\ref{lemma:chernoff}, as $|G|\geq 2n-1$ we have that, with probability greater than $2/3$, $|V_1|\geq 11n/10$. Therefore, by \itref{prop:wellconn1} and \itref{prop:wellconn2}, using Propositions~\ref{prop:findW} and~\ref{prop:findV0}, we can take a choice of partition $V(G)=W\cup V_0\cup V_1$ for which the following hold.

\stepcounter{propcounter}
\begin{enumerate}[label = \textbf{\Alph{propcounter}\arabic{enumi}}]
\item $|V_1|\geq 11n/10$.\label{prop:wellproof3}
\item For every pair of vertex disjoint sets $U_1,U_2\subset V(G)$ with size $m$, there are at least $\delta n$ internally vertex disjoint paths from $U_1$ to $U_2$ with length $\ell$ through $Z$.\label{prop:wellproof1}
\item For each $U\subset V(G)$ with $|U|=m$, $|N(U,V_0)|\geq n/10$.\label{prop:wellproof2}
\end{enumerate}
Now, let $X_0\subset V_0$ have size $n/20$ (possible by \ref{prop:wellproof2}), and note that, by \ref{prop:wellproof2}, for each $U\subset V(G)$ with $|U|=m$, $|N(U,V_0\setminus X_0)|\geq n/20$.
Then, similarly to the proof of Proposition~\ref{prop:expander:subgraphnewnew}, there is a set $W\subset V(G)$ with $|W|<m$ such that, for each $U\subset V(G)\setminus W$ with $|U|\leq m$, $|N_G(U,V_0\setminus (X_0\cup W)|\geq 3d|U|$. Let $V_0'=V_0\setminus W$ and $X=X_0\setminus W$. Then, by Proposition~\ref{prop:weakerextendability}, we have that $I(X)$ is $(d,m)$-extendable in $G[V_0']$. Let $n_0=|V_0'|-n/10$, and note that $G[V_0']$ is $(m,n_0)$-joined by \ref{prop:wellproof2} and
\begin{equation}\label{eqn:V0prime}
|V_0'|=n_0+n/10\geq n_0+|X|+(2d+2)m+n/40.
\end{equation}
Now, take $T$ and arbitrarily select a leaf $t_1$ of $T$. For each $i=2,\ldots,n$ in turn, if possible let $t_i$ be a neighbour of $t_{i-1}$ in $T$ which is not in $\{t_1,\ldots,t_{i-2}\}$, and otherwise, let $t_i$ be any vertex of $T$ not in $\{t_1,\ldots,t_{i-2}\}$ with a neighbour in $\{t_1,\ldots,t_{i-2}\}$.
For each $2\leq i\leq n$, let $s_i$ be the neighbour of $t_i$ in $T[\{t_1,\ldots,t_{i-1}\}]$.
We will build the embedding $\phi$ of $T$ by embedding $t_1,t_2,\ldots, t_n$ in this order. To initialise, set $I_a=\{1\}, I_b=\emptyset$ and $\phi(t_1)=v_1$ for an arbitrary vertex $v_1\in X$. Then, beginning with $j=2$, carry out the process whose step $j$ is as follows.

\begin{enumerate}[label = \textbf{\alph{enumi})}]
\item If $j\leq n-L$ and $t_j,\ldots,t_{j+L}$ all have degree 2 in $T$, and there is some $v_{j+L}\in X\setminus \phi(\{t_1,\ldots,t_{j-1}\})$ for which there is a $\phi(s_j),v_{j+L}$-path $\phi(s_j)v_j\ldots v_{j+L}$ in $G$ with internal vertices not in $V_0\cup \phi(\{t_1,\ldots,t_{j-1}\})$ and at most $2\ell$ vertices in $Z$, then let $\phi(t_i)=v_i$ for each $j\leq i\leq j+L$, and add $j+L$ to $I_a$.

If $j+L=n$ then stop, otherwise proceed to step $j+L+1$.
\item Otherwise, if $s_j$ satisfies $\phi(s_j)\in V_0'$, and there is some $v_j\in V_0'\setminus (X\cup\{t_1,\ldots,t_{j-1}\})$ such that $I(X)\cup \phi(T[\{t_i:i\in I_a\cup I_b\}])+\phi(s_j)v_j$ is $(d,m)$-extendable in $G[V'_0]$, then let $\phi(t_j)=v_j$ and add $j$ to $I_b$.

If $j=n$, then stop, otherwise proceed to step $j+1$.
\item If neither of these cases occur, then stop.
\end{enumerate}
Assume we are at the start of step $j$ of this process, note the following. If $j\leq n-L$ and $t_j,\ldots,t_{j+L}$ all have degree 2 in $T$, then by the ordering of the vertices in $T$, we have that $T[\{t_j,\ldots,t_{j+L}\}]$ is a path with length $L$. Also, $I_a$ is the set of indices $i\in[j-1]$ such that $\phi(t_i)\in X$, $I_b$ is the set of indices $i\in [j-1]$ such that $\phi(t_i)\in V'_0\setminus X$, and, if $i\in[j-1]$ and $t_i$ is adjacent to some $t_{j'}\in V(T)$ with $j'\geq j$, then $\phi(t_i)\in V'_0$ and $i\in I_a\cup I_b$.
We now further analyse this process via the following three claims and show that it will only stop at a step $j$ if $j+L=n$ or $j=n$, and that, in either case, this implies that we have embedded a copy of $T$ in $G$.

\begin{claim}\label{clm:prettyeasy}
If the process stops at step $j$ with $j\neq n$ and $j+L\neq n$, then $|I_b|\geq n/50$.
\end{claim}
\begin{proof}[Proof of Claim~\ref{clm:prettyeasy}] Suppose for a contradiction, that $|I_b|<n/50$. From the observations above, we have $\phi(s_j)\in V'_0$. Furthermore, using $|I(X)\cup\phi(T[\{t_i:i\in I_a\cup I_b\}])|=|X|+|I_b|$ and \eqref{eqn:V0prime}, we have
\[
|V'_0|\geq |I(X)\cup\phi(T[\{t_i:i\in I_a\cup I_b\}])|+(2d+2)m+n_0+1.
\]
As $I(X)\cup\phi(T[\{t_i:i\in I_a\cup I_b\}])$ is $(d,m)$-extendable in $G[V'_0]$, which is $(m,n_0)$-joined, and the degree of $\phi(s_i)$ is at most $\Delta-1$ in $I(X)\cup\phi(T[\{t_i:i\in I_a\cup I_b\}])$, by Lemma~\ref{lemma:adding:leaf}, there is $v_j\in V_0'\setminus (X\cup\{t_1,\ldots,t_{j-1}\})$ adjacent to $\phi(s_i)$, such that $I(X)\cup\phi(T[\{t_i:i\in I_a\cup I_b\}])+\phi(s_j)v_j$ is $(d,m)$-extendable in $G[V_0']$. This contradicts that the process stopped at step $j$ as we could have carried out \textbf{b)}.
\renewcommand{\qedsymbol}{$\boxdot$}
\end{proof}
\renewcommand{\qedsymbol}{$\square$}

\begin{claim} $|X\cap \phi(\{t_1,\ldots,t_{j-1}\})|\leq n/100$ and $|Z\cap \phi(\{t_1,\ldots,t_{j-1}\})|<\delta n/4$.\label{clm:easy}
\end{claim}
\begin{proof}[Proof of Claim~\ref{clm:easy}] Note that, for each $j\geq 2$, $\phi(t_j)\in X$ only if step $j-L$ is carried out via part \textbf{a)}, and exactly 1 vertex is embedded into $X$ in this step. Thus, $|X\cap \phi(\{t_1,\ldots,t_{j-1}\})|\leq 1+(n-1)/L\leq n/100$.
Similarly, the only vertices embedded into $Z$ are embedded via part \textbf{a)}, where at most $2\ell$ vertices are used each step. Thus, $|Z\cap \phi(\{t_1,\ldots,t_{j-1}\})|\leq 2\ell (n-1)/L<\delta n/4$.
\renewcommand{\qedsymbol}{$\boxdot$}
\end{proof}
\renewcommand{\qedsymbol}{$\square$}

\begin{claim} If $|I_b|\geq n/50$, then there are vertex disjoint sets $U_1,U_2\subset V'_0$ with size $m$ such that there is no path in $G$ from $U_1$ to $U_2$ of length $L$ with internal vertices not in $V'_0\cup \phi(\{t_1,\ldots,t_{j-1}\})$ and at most $2\ell$ vertices in $Z$.\label{clm:alsoprettyeasy}
\end{claim}
\begin{proof}[Proof of Claim~\ref{clm:alsoprettyeasy}]
Suppose $|I_b|\geq n/50$. Let $J\subset I_b$ be the set of $j\in I_b$ such that $t_j,\ldots,t_{j+L}$ all have degree 2 in $T$. Suppose for a contradiction that $|J|<|I_b|/2$. For each $j\in J\setminus I_b$, there is some $j\leq i\leq {j+L}$ such that $t_i$ either has degree 1 or degree at least 3 in $T$, so in total there are at least $|I_b|/2(L+1)\geq n/100(L+1)$ such indices $i$. As $T$ has at most $10\Delta^2\eps n$ leaves and $\eps\ll 1/L$, $T$ must then have fewer leaves than vertices with degree at least 3. However,
\[
2(n-1)=2e(T)=\sum_{i\in [n]}d_T(t_i)\geq |\{i:d_T(t_i)=1\}|+2(n-|\{i:d_T(t_i)=1\}|)+|\{i:d_T(t_i)\geq 3\}|,
\]
so it follows that $T$ has at least as many leaves as vertices with degree at least 3, contradiction. Thus, $|J|\geq|I_b|/2\geq n/100\geq\Delta m$.

Let $U_1\subset \phi(\{s_i:i\in J\})$ have size $m$, which is possible as $|J|\geq\Delta m$ and let $U_2\subset X\setminus \phi(\{t_1,\ldots,t_{j-1}\})$ have size $m$ (possible by Claim~\ref{clm:easy}). Suppose there exists a path in $G$ with length $L$ connecting some $\phi(s_i)\in U_1$ with some vertex in $U_2$, such that all its internal vertices are not in $V_0'\cup\phi(\{t_1,\ldots,t_{j-1}\})$ and at most $2\ell$ of them are in $Z$, then we could have carried out step $i$ with \textbf{a)} to embed $t_i$. But as $i\in J\subset I_b$, we actually carried out step $i$ with \textbf{b)} instead, a contradiction. Hence, such a path does not exist, proving the claim.
\renewcommand{\qedsymbol}{$\boxdot$}
\end{proof}
\renewcommand{\qedsymbol}{$\square$}

Now, suppose for a contradiction that $|I_b|\geq n/50$. Using Claim~\ref{clm:alsoprettyeasy}, take vertex-disjoint sets $U_1,U_2\subset V'_0$ with size $m$ such that there is no path in $G$ from $U_1$ to $U_2$ of length $L$ with internal vertices not in $V'_0\cup \phi(\{t_1,\ldots,t_{j-1}\})$ and at most $2\ell$ of them in $Z$. Let $V_1''=V_1\setminus \phi(\{t_1,\ldots,t_{j-1}\})$, so that $|V_1''|\geq n/10$ by \ref{prop:wellproof3}.
By Theorem~\ref{theorem:ramseygoodnessforpaths}, as $G^c$, and hence $G^c[V_1'']$, contains no copy of $K_{\eps n}^{k-1}\times K_m^c$, there must be a path with length $2Lm$ in $G[V_1'']$. On this path, we can find vertex-disjoint paths $Q_1,\dots, Q_{2m}$ in $G[V_1'']$, each with length $L-2\ell$.
For each $i\in [2m]$, let $x_i$ and $y_i$ be the end vertices of $Q_i$.

Let $I_1\subset [2m]$ be the set of $i\in [2m]$ for which there are at least $\ell$ vertex-disjoint paths of length $\ell$ from $x_i$ to $U_1$ with internal vertices in $Z\setminus \phi(\{t_1,\ldots,t_{j-1}\})$. Then, there are at most $\ell\cdot |[2m]\setminus I_1|<\delta n/2$ vertex-disjoint paths of length $\ell$ from $\{x_i:i\in [2m]\setminus I_1\}$ to $U_1$ with internal vertices in $Z\setminus \phi(\{t_1,\ldots,t_{j-1}\})$. By Claim~\ref{clm:easy} and \ref{prop:wellproof1}, we must then have that $|[2m]\setminus I_1|<m$, so that $|I_1|\geq m+1$. Similarly, if $I_2\subset [2m]$ is the set of $i\in [2m]$ for which there are at least $\ell$ vertex-disjoint paths of length $\ell$ from $y_i$ to $U_2$ with internal vertices in $Z\setminus \phi(\{t_1,\ldots,t_{j-1}\})$, then $|I_2|\geq m+1$. Hence, there exists some $i\in I_1\cap I_2$. Let $P_1$ be a path of length $\ell$ from $x_i$ to $U_1$ with internal vertices in $Z\setminus\phi(\{t_1,\ldots,t_{j-1}\})$. Since there are at least $\ell$ vertex disjoint paths of length $\ell$ from $y_i$ to $U_2$ with internal vertices in $Z\setminus\phi(\{t_1,\ldots,t_{j-1}\})$, we can find one, say $P_2$, that is disjoint from $P_1$. Then $P_1, Q_i$ and $P_2$ attached together form a path in $G$ from $U_1$ to $U_2$ of length $L$ with internal vertices not in $V'_0\cup \phi(\{t_1,\ldots,t_{j-1}\})$ and at most $2\ell$ of them in $Z$, a contradiction.

Thus, we must have that $|I_b|<n/50$. Therefore, by Claim~\ref{clm:prettyeasy}, the process stops with $j=n$ or $j=n-L$, and in either case the process has then embedded a copy of $T$. Thus, $G$ contains a copy of $T$, as required.
\end{proof}

\subsection{Proof of Theorem~\ref{theorem:main:1}}\label{sec:finalproof}
%\maintheorem*
%\lemmalinearlymanyleaves*
%\lemmanonwellconnected*
%\lemmawellconnected*

 %but before delving into the proof, let us sketch the proof strategy in more detail. As said in Section~\ref{sec:overview}, the proof is by induction on $\chi(H)$, and note that  Theorem~\ref{corollary:k=2} already implies the base case $\chi(H)=2$. For the inductive step, let us assume that $\chi(H)=k\ge 3$ and that Theorem~\ref{theorem:main:1} holds for $k-1$. Suppose $G$ is a graph on $(k-1)(n-1)+m_1$ vertices such that $G^c$ contains no copy of $K_{m_1,\ldots,m_k}$, and let $T$ be an $n$-vertex with maximum degree at most $\Delta$.
%As noted by Balla, Pokrovskiy, and Sudakov~\cite{balla2018}, using the induction hypothesis we can show that every set of size $m_k$ in $G$ has a neighbourhood of size at least $n-m_k$, as otherwise we can remove a subset of size $m_k$ with few external neighbours and then use the inductive hypothesis to conclude that $T\subset G$.

Finally, using Lemmas~\ref{lemma:linearlymanyleaves},~\ref{lemma:nonwellconnected} and \ref{lemma:wellconnected}, we can prove Theorem~\ref{theorem:main:1} from Theorem~\ref{corollary:k=2} by induction.

\begin{proof}[Proof of Theorem~\ref{theorem:main:1}]
For each $\Delta$, the proof is by induction on $k\ge 2$. When $k=2$, we have shown that the required $\mu_{\Delta,2}$ exists by Theorem~\ref{corollary:k=2}, so let us assume that $k\ge 3$ and that the required $\mu_{\Delta,k'}$ exist for each $2\leq k'\leq k-1$. % Let $C_{\Delta,k-1}$ be the constant in Theorem~\ref{theorem:main:1} for $k-1$ and let
Let $\mu\leq\min\{\mu_{\Delta,k'}:2\leq k'\leq k-1\}$, $\mu\ll 1/\Delta,1/k$, and let $\eps$ and $\lambda$ satisfy $\eps\ll \lambda\ll \mu$. Note that we can assume that $1/n\ll \eps$, for if we have proved it for all $n\geq n_0$ with $1/n_0\ll \eps$, then we can reduce $\eps$ to $1/n_0$ to get a result for all $n$. Note that this choice of $\mu$ implies the following property.
\stepcounter{propcounter}
\begin{enumerate}[label = \textbf{{\Alph{propcounter}\arabic{enumi}}}]
\item For any $2\leq k'\leq k-1$ and $m'\leq \mu n$, $R(T,K_{\mu n}^{k'-1}\times K_{m'}^c)=(k'-1)(n-1)+m'$.\label{prop:finalproof1}
\end{enumerate}
We will show that we can take $\mu_{\Delta,k}=\eps$.

For this, let $m\leq \eps n$ and let $T$ be any $n$-vertex tree with $\Delta(T)\le \Delta$. Let $G$ be a graph on $(k-1)(n-1)+m$ vertices such that $G$ contains no copy of $T$ and $G^c$ contains no copy of $K_{\eps n}^{k-1}\times K_m^c$.
Similarly to the start of the proof of Lemma~\ref{lemma:nonwellconnected} (with \ref{prop:nonwell1}), we can assume we have the following property.
\begin{enumerate}[label = \textbf{{\Alph{propcounter}\arabic{enumi}}}]\stepcounter{enumi}
\item For every $U\subset V(G)$ with $|U|=m$, we have $|N_G(U)|\geq (1-\lambda)n$.\label{prop:finalproof2}
\end{enumerate}
%that every set $U\subset V(G)$ with $|U|= m$ satisfies $|N_G(U)|\geq (1-\lambda)n$.
Furthermore, if $U\subset V(G)$, $|U|=\eps n$ and $|U\cup N_G(U)|< n$, then we have $|G-U-N_G(U)|\geq (k-2)(n-1)+m$, and, thus, as $\eps \ll \mu$, the complement of $G-U-N_G(U)$ contains a copy of $K^{k-2}_{\eps n}\times K_m^c$, which, with all the edges from this to $U$, forms a copy of $K^{k-1}_{\eps n}\times K_m^c$ in $G^c$, a contradiction. Thus, we can assume that,

\begin{enumerate}[label = \textbf{{\Alph{propcounter}\arabic{enumi}}}]\addtocounter{enumi}{2}
\item For every $U\subset V(G)$ with $|U|=\eps n$, we have $|U\cup N_G(U)|\geq n$.\label{prop:finalproof3}
\end{enumerate}
Now, if $T$ has more than $10\Delta^2\eps n$ leaves, then, by \ref{prop:finalproof3} and Lemma~\ref{lemma:linearlymanyleaves}, $G$ contains a copy of $T$, a contradiction. Thus, we can assume that $T$ has fewer than  $10\Delta^2\eps n$ leaves. Therefore, by  \ref{prop:finalproof1} and Lemma~\ref{lemma:nonwellconnected}, there must be no
partition $V(G)=V_0\cup V_1\cup V_2$ such that $e_G(V_1,V_2)=0$, $|V_1|,|V_2|\geq m$ and $|V_0|\leq \lambda n$. Finally, then, by
%, then, by $G^c$ contains a copy of $K_{\lambda n}^{k-1}\times K_m^c$.
%Thus, we can assume that $T$ has fewer than $10\Delta^2\lambda n$ leaves and there is no partition $V(G)=V_0\cup V_1\cup V_2$ such that $e_G(V_1,V_2)=0$, $|V_1|,|V_2|\geq m$ and $|V_0|\leq \lambda n$.
\ref{prop:finalproof2} and Lemma~\ref{lemma:wellconnected}, $G$ contains a copy of $T$ or $G^c$ contains a copy of $K_{\eps n}^{k-1}\times K_m^c$, a contradiction.

Therefore, for every $m\leq \eps n$, every $n$-vertex tree $T$ with $\Delta(T)\le \Delta$, and every graph $G$ with $(k-1)(n-1)+m$ vertices,  $G$ contains a copy of $T$ or $G^c$ contains a copy of $K_{\eps n}^{k-1}\times K_m^c$, so that setting $\mu_{\Delta,k}=\eps$ completes the proof.
\end{proof}

%% file: wellconnectedpic.tex
\begin{tikzpicture}[scale=.9]

\def\zpath{black!60}

\def\wid{10};
\def\hgt{5};

%Xboxes
\draw [black!10,fill=black!10] (0,0.4*\hgt) rectangle (0.5*\wid,0.5*\hgt);
\draw [black!10,fill=black!10] (0.45*\wid,0) rectangle (0.5*\wid,0.5*\hgt);

%
%Zboxes
\draw [black!20,fill=black!20] (0,0.52*\hgt) rectangle (0.56*\wid,0.62*\hgt);
\draw [black!20,fill=black!20] (0.51*\wid,0) rectangle (0.56*\wid,0.62*\hgt);

\draw [black] (0,0.52*\hgt) -- (0,0.62*\hgt) -- (0.56*\wid,0.62*\hgt) -- (0.56*\wid,0)  -- (0.51*\wid,0)  -- (0.51*\wid,0.52*\hgt) -- (0,0.52*\hgt);

\draw[black] (0,0) rectangle (0.5*\wid,0.5*\hgt);
%\draw[black!50] (0,0) rectangle (0.55*\wid,0.6*\hgt);

\draw [black] (0,0.64*\hgt) -- (0,\hgt) -- (0.775*\wid,\hgt) -- (0.775*\wid,0) -- (0.57*\wid,0) -- (0.57*\wid,0.64*\hgt) -- (0,0.64*\hgt);

%\draw[black!50] (0,0) rectangle (0.775*\wid,\hgt);

\draw (0.275,0.275) node {$V_0'$};
\draw (0.535*\wid,0.555*\hgt) node {$Z$};
\draw (0.475*\wid,0.445*\hgt) node {$X$};
\draw (0.75*\wid,0.935*\hgt) node {$V_1$};

\tikzstyle{every node}=[circle, draw, fill, inner sep=0pt, minimum width=2pt]

\draw node (new1) at (0.425*\wid,0.435*\hgt){};
\draw node (new2) at ($(new1)-(0,0.3)$){};
\draw node (new3) at ($(new2)-(0.2,0.2)$){};
\draw node (new4) at ($(new3)-(0,0.3)$){};
\draw node (new5) at ($(new3)-(0.3,0)$){};

\draw node (new6) at ($(new4)+(0.6,-0.2)$){};
\draw node (new7) at ($(new6)+(0,-0.6)$){};
\draw node (new8) at ($(new7)+(-0.4,0)$){};
\draw node (new9) at ($(new8)+(-0.2,0.2)$){};
\draw node (new10) at ($(new8)+(-0.2,-0.2)$){};

\draw (new1) -- (new2) -- (new3) -- (new4);
\draw (new3) -- (new5);
\draw (new10) -- (new8) -- (new9);
\draw (new8) -- (new7);

\foreach \x/\y/\z in {11/10/1,12/11/-1,13/12/1,14/13/-1,15/14/1,16/15/-1,17/16/1,18/17/-1,19/18/1,20/19/-1}
{
\draw node (new\x) at ($(new\y)+(-0.2,\z*0.2)$){};
}
\foreach \x/\y in {11/10,12/11,13/12,14/13,15/14,15/16,17/16,18/17,19/18,20/19}
{
\draw (new\x) -- (new\y);
}

\foreach \n in {0,1,...,3}
{
\draw node (A\n) at ($(0.3+1.2*\n,0.6*\hgt+0.325)$){};
\draw node (B\n) at ($(0.8+1.2*\n,0.6*\hgt+0.325)$){};
\coordinate (C\n) at ($(0.3+1.2*\n,0.9*\hgt+0.325)$){};
\coordinate (D\n) at ($(0.8+1.2*\n,0.9*\hgt+0.325)$){};
\draw[snake=coil,segment aspect=0, segment amplitude=.4mm,segment length=2.9mm] (A\n) -- (C\n);
\draw[snake=coil,segment aspect=0, segment amplitude=.4mm,segment length=2.9mm]  (C\n) -- (D\n);
\draw[snake=coil,segment aspect=0, segment amplitude=.4mm,segment length=2.9mm] (D\n) -- (B\n);
}

\begin{scope}[rotate=270,shift={(-2.3,2.65)}]
\foreach \n in {0,1}
{
\draw node (A1\n) at ($(0.3+1.2*\n,0.6*\hgt+0.2)$){};
\draw node (B1\n) at ($(0.8+1.2*\n,0.6*\hgt+0.2)$){};
\coordinate (C1\n) at ($(0.3+1.2*\n,0.9*\hgt+0.2)$){};
\coordinate (D1\n) at ($(0.8+1.2*\n,0.9*\hgt+0.2)$){};
\draw[snake=coil,segment aspect=0, segment amplitude=.4mm,segment length=2.9mm] (A1\n) -- (C1\n);
\draw[snake=coil,segment aspect=0, segment amplitude=.4mm,segment length=2.9mm]  (C1\n) -- (D1\n);
\draw[snake=coil,segment aspect=0, segment amplitude=.4mm,segment length=2.9mm] (D1\n) -- (B1\n);
}
\end{scope}

\begin{scope}[shift={(1.5,0.845)},scale=0.8]
%nodes tree 1
\draw node (v1) at (-1,0){};
\draw node (v2) at (-1.2,.2){};
\draw node (v3) at (-.8,.4){};
\draw node (v4) at (-1.6,0){};
\draw node (v5) at (-1.6,.3){};
\draw node (v6) at (-1.2,-.2){};
\draw node (v7) at (-1,-.3){};
\draw node (v8) at (-.4,.2){};
\draw node (v9) at (0,.3){};
\draw node (v10) at (-.5,-.1){};
\draw node (v11) at (-.2,.0){};
\draw node (w1) at (-.5,.7){};
\draw node (w2) at (-.9,.9){};
\draw node (w3) at (-.6,1.8){};
\draw node (w4) at (.1,.8){};
\draw node (w5) at (-.2,1){};
\draw node (w6) at (-.2,1.4){};
\draw node (w7) at (.1,1.2){};

\draw node (z1) at (-.3,1.8){};
\draw node (z2) at (1.2,1.8){};

\tikzstyle{every node}=[]
\draw ($(w3)-(0.225,0.175)$) node {$v_1$};
\tikzstyle{every node}=[circle, draw, fill, inner sep=0pt, minimum width=2pt]

\draw[snake=coil,segment aspect=0, segment amplitude=.4mm,segment length=2mm,\zpath] (w3) -- (A0);
\draw[snake=coil,segment aspect=0, segment amplitude=.4mm,segment length=2mm,\zpath] (z1) -- (B0);
\draw[snake=coil,segment aspect=0, segment amplitude=.4mm,segment length=2mm,\zpath] (z1) -- (A1);
\draw[snake=coil,segment aspect=0, segment amplitude=.4mm,segment length=2mm,\zpath] (z2) -- (B1);
\draw[snake=coil,segment aspect=0, segment amplitude=.4mm,segment length=2mm,\zpath] (z2) -- (A2);
\draw[snake=coil,segment aspect=0, segment amplitude=.4mm,segment length=2mm,\zpath] (new1) -- (B3);
\draw[snake=coil,segment aspect=0, segment amplitude=.4mm,segment length=2mm,\zpath] (new4) -- (A10);
\draw[snake=coil,segment aspect=0, segment amplitude=.4mm,segment length=2mm,\zpath] (new7) -- (B11);
\draw[snake=coil,segment aspect=0, segment amplitude=.4mm,segment length=2mm,\zpath] (new6) -- (A11);
\draw[snake=coil,segment aspect=0, segment amplitude=.4mm,segment length=2mm,\zpath] (new6) -- (B10);

%edges tree 1
\draw (v3) -- (v1) -- (v2);
\draw (v2) -- (v4);
\draw (v2) -- (v5);
\draw (v6) -- (v1) -- (v7);
\draw (v9) -- (v8) -- (v10);
\draw (v3) -- (v8)-- (v11);
\draw (v3) -- (w1) -- (w2);
\draw (w1) -- (w5);
\draw (w6) -- (w5) -- (w7);
\draw (w3) -- (w6);% -- (Q4);
\draw (w5) -- (w4);
\end{scope}

\begin{scope}[shift={(2.25,2.1)},scale=0.8]
 %nodes tree 2
\draw node (u1) at (1.3,-.7){};
\draw node (u2) at (1.1,-.5){};
\draw node (u3) at (1,-.8){};
\draw node (u4) at (.9,-.3){};
\draw node (u5) at (1.2,0){};
\draw[snake=coil,segment aspect=0, segment amplitude=.4mm,segment length=2mm,\zpath] (u5) -- (B2);
\draw[snake=coil,segment aspect=0, segment amplitude=.4mm,segment length=2mm,\zpath] (u1) -- (A3);

%edges tree 2

\draw (u3) -- (u1) -- (u2);
\draw (u5) -- (u2) -- (u4);
\end{scope}
%\draw[snake=coil,segment aspect=0, segment amplitude=.4mm,segment length=2mm] (Q3) -- (u1);
\end{tikzpicture}